\documentclass[review,letterpaper]{elsarticle}
\usepackage[hmargin={1.75truecm,2truecm},vmargin={2truecm,2truecm}]{geometry}

\makeatletter
\def\ps@pprintTitle{%
   \let\@oddhead\@empty
   \let\@evenhead\@empty
   \let\@oddfoot\@empty
   \let\@evenfoot\@oddfoot}
\makeatother

\usepackage{algorithm}
\usepackage{algpseudocode}
\usepackage{tikz}
\usepackage{subcaption}
\usepackage{arydshln}
\usepackage{multirow}
\usepackage{rotating}
\usepackage{float}
\usepackage{amsmath}
\usepackage{amsfonts}
\usepackage{url}
\usepackage{amssymb}

\newtheorem{propo}{Proposition}

\newtheorem{defin}{Definition}

\newtheorem{rem}{Remark}
\newtheorem{teorema}{Theorem}
\newtheorem{example}{Example}

\usepackage{natbib}
 \bibpunct[, ]{(}{)}{,}{a}{}{,}%

\title{Branch \& Solve for Hub Location}
\author[UCA]{Elena Fernández}
\ead{elena.fernandez@uca.es}

\author[UCA]{Nicol\'as Zerega}
\ead{nicolas.zerega@uca.es}

\address[UCA]{Departamento de Estadística e Investigación Operativa, Universidad de Cádiz, República Saharaui 11519, Cádiz, Spain}

\begin{document}
    \begin{abstract}
        This paper introduces a new formulation and solution framework for hub location problems. The formulation is based on 2-index aggregated flow variables and incorporates a set of aggregated demand constraints, which are novel in hub location. With minor adaptations, the approach applies to a large class of single- and multiple-allocation  models, possibly incorporating flow bounds on activated arcs. General-purpose feasibility and optimality inequalities are also developed. Because of the small number of continuous variables, there is no need to project them out, differentiating the method from solution algorithms that rely heavily on feasibility and optimality cuts. The proposed Branch \& Solve solution framework leverages the nested structure of the problems, by solving auxiliary subproblems at selected nodes of the enumeration tree. Extensive computational experiments on benchmark instances from the literature confirm the good performance of the proposal: the basic version of the algorithm is able to solve to proven optimality instances with up to 200 nodes for several hub location families.
    \end{abstract}
    \maketitle
    \graphicspath{{../Figures/}}
    \allowdisplaybreaks

    \section{Introduction}\label{sec:Intro}
	Hub location problems (HLPs) have been intensively studied over the past few decades, driven by both their theoretical interest and their relevant applications. Core models, such as variations of $p$-median or $p$-center, have gradually evolved into more sophisticated and realistic models incorporating additional features and releasing earlier simplifying assumptions. The growing interest in HLPs within the research community is evident from the number of recent surveys and book chapters dedicated to these models {\citep[see, e.g.,][]{Alumur2021,CampbellOkelly25,ContrerasOkelly}}.

	HLPs involve two levels of decisions. At the strategic (higher) level, the structure of the \textit{solution network} must be decided. This includes selecting a set of hubs to activate and establishing interhub links that determine the \textit{backbone network} as well as \textit{access and distribution} arcs connecting non-hub nodes to the backbone network. At the operational (lower) level, it is necessary to determine a routing path through the established network for each \textit{commodity}, defined as an origin/destination (OD) pair with a specific service demand. A typical objective function in HLPs is to minimize the total cost, which comprises both the  setup (activation) costs of the hubs and interhub links and the costs for routing the commodities. Indeed, variations or alternative objectives exist, particularly in cases where strategic decisions may not incur setup costs.

	One of the main challenges in formulating HLPs, common across virtually all models, arises from  service demand, which links pairs of users within a given network \citep{ContrerasFernandez12}. For each OD pair with a demand, a specific flow must be routed through a subset of hubs, whose location and inter-connections are part of the decision-making process. As is typical in network design, binary decision variables are used to model the solution network, whereas continuous variables usually suffice for the lower-level routing decisions. Despite efforts to develop tight formulations with a small number of continuous variables \citep[see, e.g.][]{LabbeYaman2004,ContrerasFernandez14,Espejoetal23}, most HLP formulations involve a large number of such variables. In particular, most formulations use either 4-index \textit{path variables} or 3-index \textit{flow variables}. While formulations with 4-index path variables usually produce very tight Linear Programming (LP) bounds, due to their large number of variables, they scale poorly with instance size, both in terms of computational efficiency and memory requirements. On the other hand, formulations with 3-index flow variables  are typically more scalable but produce weaker LP bounds so they may become inefficient already for moderate size instances. This  burden can be partially mitigated with Benders-type reformulations \citep{Benders62}, which are stated in terms of the discrete variables only, by projecting out  the continuous variables, whose domain is expressed in terms of a family of \textit{feasibility} and \textit{optimality} inequalities, exponential in number. Such reformulations have become the state of the art for many HLP variants \citep[see, e.g. ][]{Contrerasetal11,deCamargo2009,Taherkhanietal20,Wandelt22,Espejoetal23}.

	In this paper we introduce an alternative framework for tackling a broad class of HLPs, which combines a novel mathematical programming formulation with a tailored solution strategy. The key features of the formulation are the following:

	\begin{itemize}
		\item Integrated structure: The formulation combines binary first-level variables to define the topology of the solution network and continuous lower-level  2-index variables to represent 
        flows along its arcs. These variables 
        are an aggregation of the 3-index flow variables 
        often used in the literature. Flow conservation is enforced \textit{via} standard flow-balance constraints. Unlike some existing models, our approach does not restrict the number of interhub arcs used in the routing paths.
		\item Reduced number of continuous variables: Thanks to the aggregate flow modeling, the number of continuous variables is kept low, eliminating the need for projecting them out, which is  typically used in Benders-type decompositions.
		\item Tight lower bounds: The formulation yields tight LP relaxations, largely due to a class of aggregated demand constraints. To the best of our knowledge, these constraints are novel in the HLP literature, and significantly enhance the formulation.
		\item Flow feasibility via logical constraints: When necessary, the feasibility of the obtained flows can be enforced through additional families of logic-based constraints.
		\item High versatility: The formulation framework is flexible and can be applied to a wide range of HLP variants. With minimal modifications, it is valid for both single and multiple allocation (SA and MA, respectively), and can be easily extended to incorporate flow bounds on activated arcs or capacity constraints.
	\end{itemize}

	The solution strategy that we propose, referred to as \textit{Branch \& Solve (B\&S)}, leverages the nested structure of HLPs. It emerges from the observation that, once a solution network is fixed, its associated lower-level subproblem is \textit{easy} to solve. In the absence of capacity or other additional constraints, this subproblem reduces to finding a shortest path for each commodity. When bounds or capacity constraints on the arcs are introduced, the subproblem generalizes to a multi-commodity flow problem. In either case, optimal flows through the network can be efficiently obtained.

	B\&S can be framed within the Branch \& Check  algorithmic framework \citep{Thorsteinsson21}. This framework models the problem to solve using a combination of a \textit{basic} Mixed Integer Program (MIP) and a \textit{delayed} Constrained Linear Program (CLP). The MIP part may also include an explicit formulation of the CLP or a relaxation of it. The solution method involves exploring an enumeration tree, where the LP relaxation of the MIP part is solved at each node, and branching is carried out on the discrete variables. The \textit{delayed}  CLP is solved only at selected nodes, based on custom criteria. The solutions obtained from the delayed CLP can be used to derive bounds or generate cuts, which may help strengthen the MIP part. Branch \& Check generalizes both Branch \& Bound (B\&B) and Benders Decomposition (BD), encompassing them as special cases within a unified framework.
	\begin{itemize}
		\item At one end, 
        B\&B consists solely of a \textit{basic} MIP part, with no \textit{delayed} CLP part.
		\item At the other end, in BD, the \textit{basic} MIP part is the master problem, which may initially be very \textit{thin} and is gradually enriched with feasibility and optimality cuts derived from the \textit{delayed} part.
	\end{itemize}

	The proposed B\&S resembles B\&B, in that the MIP part is not enriched with cuts derived from the \textit{delayed} CLP part. However, unlike standard B\&B, the MIP of B\&S evolves dynamically. Since the MIP includes a family of inequalities whose number is exponential in the number of nodes of the input graph, in B\&S, this part is handled using a Branch \& Cut (B\&C) approach. The CLP part does not influence the MIP part explicitly; instead, it contributes by providing feasible solutions that are optimal  at selected nodes of the search tree. Moreover, nogood cuts are added at such nodes to help the solver partition the solution space.

	This algorithmic proposal offers an alternative to traditional Benders-type methods, where feasibility and optimality cuts are fundamental to the solution process. Even when cut generation is limited to integer solutions of the master problem, many infeasible nodes may have to be explored before reaching a node producing a feasible solution network. Furthermore, while the generated cuts are violated at the nodes where they are identified, their usefulness may diminish at other nodes, as a different network configuration will yield different routing paths for at least some commodities.

	Hence, we choose to alleviate the burden for generating and incorporating feasibility and optimality cuts. Instead, we retain the 2-index low level continuous variables within the MIP formulation, which do not need to be projected out, and we solve  with \textit{ad hoc} algorithms the delayed CLP part (subproblems associated with feasible solution networks). This approach proves particularly effective given that the bounds produced by the \textit{basic} MIP part are already quite tight.

	Despite its simplicity, B\&S offers some advantages. The main ones are: $(i)$ the original HLP can be solved to optimality using an enumeration tree that only explores the CLP at nodes where the MIP produces a feasible solution network; and $(ii)$ the method remains valid for any formulation that produces feasible first-level solutions, even if lower-level feasibility is not explicitly enforced. This enables the use of formulations where the lower-level domain is relaxed, accelerating the exploration of the enumeration tree without compromising the quality of the solutions.

	Of course, the effectiveness of a B\&S algorithm largely depends on the chosen formulation, particularly, on the quality of the first-level bounds and solutions obtained at early stages of the search, which are the main drivers for the overall size of the enumeration tree. We adopt the proposed flow-based formulation described earlier, which uses a \textit{small} number of continuous variables, eliminating the need for projecting them out. Thanks to the inclusion of the aggregated demand constraints, tight lower bounds and high quality incumbent solutions are obtained at early stages of the search. This, in turn, results in small enumeration trees and improved overall performance.

	The main contributions of this paper can summarized below:
	\begin{itemize}
		\item We introduce a new formulation, based on 2-index flow variables that applies to a large class of HLPs and produces tight LP bounds. 
        It is rather versatile: with minor variations, it applies to both SA and MA policies, and  it allows for bounds on the flows through the arcs.
        \item We in introduce a family of feasibility Benders-type cuts, which guarantee the feasibility of the aggregated flows trough the arcs.
        \item For the special case where optimal routing paths use  at most one interhub arc, we provide an alternative formulation based on optimality inequalities.
		\item We design a B\&S solution framework based on the proposed formulation. We are not aware of any enumerative method for HLPs based solely on the solution of an (\textit{easy}) auxiliary subproblem to determine optimal flows at the nodes with binary values for the design variables.
        \item {In general, the proposed framework produces small enumeration trees. Its effectiveness is assessed through extensive computational tests on several HLP families under both SA and MA policies. The framework solves HLPs with up to 200 nodes within a time limit of 7200 seconds, outperforming a Branch\&Cut method based on the feasibility Benders-type cuts developed on the proposed formulation and the best-known results compiled in \cite{Wandelt22}.}
	\end{itemize}

	The remainder of this paper is structured as follows. The most relevant literature related to our work is reviewed in Section \ref{sect:literature}. Section \ref{sec:notation} introduces the notation we will use throughout the paper and formally defines the problems that we address. Section \ref{sec:formula} presents the different components of our modeling template and illustrates some of its characteristics.
    Sections \ref{sec:FeasBenders} and \ref{sec:1-inter} focus, respectively, on the feasibility Benders-type inequalities, and the formulation based on optimality constraints for the special case with at most one interhub arc. The computational tests that we have carried out to assess the effectiveness of our proposal are described in Section \ref{sec:compu}, where the obtained numerical results are summarized and analyzed. The paper closes in Section \ref{sec:conclu}  with some conclusions and avenues for future research.
    
    \section{Literature review}\label{sect:literature}
	There is a broad, and increasingly rich, literature on hub location. Seminal work \citep[see, e.g.][]{Okelly86,Okelly87,Aykin90,Okelly92,Campbell94,Campbell96} and subsequent developments \citep[see, e.g.][]{Ernst96,Skorin96,Ernst98,Ebery00,Boland04,Hamacher04,LabbeYaman2004,Marin05,Marin-et-al06} focused mainly on the development of tight formulations for \textit{fundamental} models. With few exceptions \citep[see, e.g.][]{Campbell94}, early models considered objective functions that minimized the costs for routing the commodities and, in some cases, the setup costs for activated hubs; still, they ignored the  setup costs for activated interhub links. As the field evolved, more sophisticated models were introduced progressively, with more general objective functions involving, for example, the setup costs of activated interhub links \citep[see, e.g.][]{AlumurKaraKarasan09,ContrerasFernandez14,deCamargo2017}. Prize-collecting models, entailing decisions for the commodities to serve  also emerged \citep[see, e.g.][]{Alibeyg16,Alibeyg18,TaherkhaniAlumur19,Taherkhanietal20}. Alternative objectives that have been considered in the literature include center \citep{ErnstCenter09}, cover \citep[][]{AlumurKaraCover09,AlumurKaraKarasan09,ErnstCover18}, or ordered median \citep{PuertoOM11,PuertoOM13}. The scope of research has also expanded to include other topics such as  hub-arc location \citep{CampbellHubArcII,CampbellHubArcI}, capacity issues \citep{ContrerasDiaz09,ContrerasDiaz11,Ebery00,ElhedhliWu10}, backbone networks with specific topologies \citep{LabbeRingStar,ContrerasTreeofHubs}, and the modeling of economies of scale \citep{OkellyBryan,Kimms06,LuerMarianov,Dominguez24}. The interested reader is referred to \cite{AlumurKara2008,CampbellOkelly25,ContrerasOkelly,Alumur2021} for comprehensive overviews of models, applications, and current research trends on the topic.
    
    Next, we focus on existing work on topics related to the main contributions of this paper. We review the key modeling techniques for uncapacitated hub location problems for both SA and MA,  along with the main exact solution approaches that have been developed for solving these models. Unless otherwise stated, we restrict our discussion to median-type minimization  problems. To distinguish among variations or extensions we adopt the following terminology:
    \begin{itemize}
        \item $p$-median: Models with a fixed number $p$ of hubs to activate, where the objective function accounts only for the total routing costs of the commodities.
        \item $H$-median: Models where the number of hubs to activate is not fixed, and the objective function includes setup costs for the activated hubs, in addition to the commodities routing costs.
        \item $G$-median: An extension of $H$-median that also includes setup costs for the activated interhub links in the objective function.
    \end{itemize}

	As noted above, most early studies focused on $p$-median or $H$-median variants. In these models, any arc connecting two activated hubs can be used for routing commodities  without incurring an activation cost. Then, under the assumption that routing costs satisfy the triangle inequality, it is easy to show that there exists an optimal solution in which each routing path uses at most one interhub link. This optimality property became the pillar for the first HLP formulations for both SA and MA. For SA, initial formulations were quadratic \citep[see,][]{Okelly86,Okelly87,Aykin90,Skorin96,LabbeYaman2004}. They used binary variables $z_{ik}$ indicating the (single) allocation of node $i$ to activated hub $k$ (with $z_{kk}$ determining whether node $k$ is an activated hub). Using these variables, the demand from $i$ to $j$ will be routed through path $i-k-m-j$ provided that $i$ and $j$ are allocated to activated hubs $k$ and $m$, respectively (i.e., $z_{ik}=z_{jm}=1$). Hence, this routing path can be modeled as the product $z_{ik}z_{mj}$. While this modeling \textit{logic} is elegant and simple, the resulting quadratic terms are algorithmically challenging to handle. This opened the way to formulations with 4-index variables where the quadratic terms are  linearized by defining \textit{path} variables $x_{ijkm}=z_{ik}z_{mj}$ \citep[see, e.g.][]{Campbell94,Campbell96,Skorin96,LabbeYaman2004}. Numerical results from computational experiments with these formulations showed that they  usually yield very tight LP bounds. Still, due to the their large number of variables, such formulations do not scale well with instance size, neither in terms of effectiveness nor memory requirements. Benders-type reformulations \citep{Benders62} offered an alternative for overcoming this difficulty by projecting out these variables \citep{LabbeYaman2004,Espejoetal23}. These methods, widely used in other HL variants, will be further discussed later in this section. More recently, \cite{Meier2018} proposed a novel approach for linearizing the quadratic terms in $p$-median models, referred to as \textit{Euclidean projection method}, which applies when distances between the locations of potential hubs are Euclidean. The method is based on a formulation with two sets of decision variables: binary SA and continuous 2-index variables that indicate the flows through the interhub links. An alternative approach to mitigate the difficulties caused by the large number of 4-index variables was proposed in \cite{Ernst96} where 3-index flow variables  are used in conjunction with the SA variables. In this formulation, the flow variables $f^i_{km}$ are defined for each triplet of nodes in the network and indicate the flow originating at node $i$ that traverses the interhub arc $(k,m)$. Flow balance constraints are imposed, for each pair of nodes $i, k$, to regulate the flow with origin at node $i$ passing through hub $k$.

	In the context of MA models, activated hubs  are identified with  binary decision variables $z_k$ (instead of $z_{kk}$ as in SA). Now, the routing paths of commodities sharing the same origin (resp. destination) may use different \textit{first} (resp. \textit{second}) hubs, although path variables $x_{ijkm}$ remain as a natural choice for modeling routing decisions \citep[][]{Campbell94,Campbell96,Hamacher04,Marin05,Marin-et-al06,Contrerasetal11}. As in SA models, these formulations yield tight LP bounds but face similar scalability limitations. Not surprisingly, the alternatives proposed  for overcoming them are also similar. \cite{Ernst98,Ebery00,Boland04} proposed formulations with 3-index flow variables $f^i_{km}$,  as defined above, along with additional variables for the flows through non-interhub arcs. \cite{Garcia2} use a different modeling approach to develop a formulation based on 2-index continuous \textit{radius} variables, representing the cost of the routing paths connecting OD pairs. 

	Unlike the models discussed so far, HLPs with $G$-median objectives, accounting for activation costs of interhub links, may have (unique) optimal solutions that induce incomplete backbone networks. Thus, optimal routing paths may involve multiple interhub links. This feature introduces additional modeling difficulties, as reflected in the related literature, which is significantly smaller than that with complete backbone networks, even if the trend is currently changing. Some studies with $G$-median objectives impose the simplifying assumption that service routes traverse at most one interhub arc. For this simplified model, \cite{ContrerasFernandez14} introduced a formulation  based on supermodular properties,  using 2-index variables only, which applies to several  HLP classes, including hub-arc location. Interestingly, for the particular case of the $p$-median, the \textit{supermodular} formulation coincides with the \textit{radius} formulation of \cite{Garcia2}.
	
	General $G$-median models imposing no restriction on the length of service routes have been studied as well. To the best of our knowledge, the first such model under a SA policy was introduced in \cite{AlumurKaraKarasan09}, who proposed a formulation based on 3-index flow variables $f^i_{km}$. The MA case has been addressed in \cite{Okellyetal15,deSa18}. Both works allow for direct OD shipments between non-hub nodes (\textit{bridge arcs}). \cite{Okellyetal15} fix the number of activated hubs and use 3-index flow variables analogous to those of \cite{AlumurKaraKarasan09}. In contrast, \cite{deSa18} use continuous variables, with values in $[0, 1]$, to represent the fractions of OD demand routed through the different types of arcs.
    
	The formulations that we propose use flow variables to represent the total flow through the different types of arcs (access, interhub, and distribution) and flow balance constraints to regulate flow circulation at the nodes. These variables and constraints can be interpreted as aggregated versions of the standard 3-index flow variables and  balance constraints found in the literature. This aggregation no longer ensures that the demand of each individual OD pair is correctly routed through the backbone network. With the purpose of reinforcing the flows routed through solution networks, we introduce a family of aggregated demand constraints, which, to the best of our knowledge, are novel for HLPs. They impose that the total flow sent from any subset of nodes $S\subset V$ to its complement $V\setminus S$ be at least equal to the total demand from the nodes in $S$ to the nodes in $V\setminus S$. While these constraints substantially strengthen the formulation, they are yet not sufficient to guarantee the feasibility of the flows through the solution network. For this reason, the formulation is completed with several new families of logic-based constraints.

	We next turn our attention to the main solution approaches developed for uncapacitated HLPs similar to the ones that we address in this work. For approaches were Lagrangean relaxation techniques are used to derive lower bounds or are integrated within branch \& price algorithms, we refer the  reader to \cite{PirkulShilling98,ContrerasDiaz09,ElhedhliWu10,ContrerasDiaz11}. Off-the-shelf solvers are used for the solution of early formulations with 4- and 3-index  variables. Typically, instances from the CAB and AP\footnote{See  \cite{Okelly87} \cite{Ernst96}, respectively} data sets are used or adapted. Exact methods are often compared with \textit{ad hoc} heuristics, which we do not analyze here. Outstanding results are reported by \cite{Skorin96} with their 4-index formulation for  SA and MA $p$-median instances with up to 25 nodes. The performance of a 3-index formulation for SA $p$-median is analyzed in \cite{Ernst96}  for instances with up to 50 nodes. For MA models using 3-index formulations, results have been reported for  instances with up to 50 nodes, first for $p$-median in \cite{Ernst98} and later in \cite{Boland04} also for $H$-median. Regarding MA $H$-median models, \cite{Marin-et-al06} present excellent results with their reinforced 4-index formulations, for  instances with up to 30 nodes.  For these models, \cite{Hamacher04,Marin05} analyze the impact of inequalities derived from polyhedral analyses. For $G$-median models with arbitrarily long routing paths, \cite{AlumurKaraKarasan09} test their 3-index formulation using CAB instances with up to 25 nodes as well as another set of instances with 81 demand points from the Turkish network used in \cite{Yaman07}. Their results show that small instances can be optimally solved without much difficulty, but highlight the need for more sophisticated solution techniques for larger instances.

	Shortly thereafter, Benders reformulations (\citeyear{Benders62}) became increasingly popular for addressing network optimization problems involving two sets of decision variables, one discrete and one continuous \citep[e.g.][]{CORDEAU2019,FischettiLjubicSinnl17}. The general principle behind these methods is to project out the continuous variables and express their domain as a family of inequalities (Benders cuts) that involve discrete variables only. This principle leads to two alternative algorithmic frameworks. One is to use an iterative cutting plane algorithm, where a master problem is solved at each iteration, and cuts violated by the current solution are generated and added to the master. The other one is to use a B\&Benders-Cut  (BBC) approach, where a single enumeration tree is explored and Benders cuts are separated as in a classical B\&C procedure. Both frameworks have been applied to HLPs. For instance, \cite{Contrerasetal11} developed an iterative cutting plane  algorithm based on a 4-index formulation of the $H$-median and solved instances with up to 500 nodes. Nonetheless, BBC approaches have become predominant in the field \citep[e.g.][]{Espejoetal23,deCamargo2009,deCamargo2017,deSa18}. Typical refinements applied to improve the performance of such methods include the use of aggregated Benders cuts \citep[e.g.][]{deCamargo2009}, the generation of tighter Pareto-optimal cuts \citep[e.g.][]{Taherkhanietal20}, and the incorporation of heuristic methods to identify optimal or near/optimal solutions at early stages of the solution process \citep[e.g.][]{Contrerasetal11}. An empirical comparison of the computational performance of these methods across different HLPs is presented in \cite{Wandelt22}.

	In this work we develop an enumeration method tailored to the formulation we propose, which leverages the nested structure of HLPs. The approach is motivated by the observation that the formulation typically produces good-quality solution networks early in the search process. Once a solution network is given, an associated optimal lower-level solution can be efficiently obtained by solving an auxiliary subproblem. As we will show, this simple strategy, without any further enhancement, is highly effective in solving several HLP families, and proves competitive against highly specialized, state-of-the-art solution algorithms for the tested families.

    \section{Notation and preliminaries}\label{sec:notation}
	Consider a network $N=(V, A)$, with node set $V=\{1, 2,\dots, n\}$ and arc set $A$ representing the existing connections between pairs of nodes.  Let $E=\{ij: (i, j)\in A \text{ or } (j, i)\in A, i< j\}$ denote the set of (undirected) edges underlying the arc set $A$. Associated with each arc $(i,j)\in A$, there is a unit routing cost, denoted by $c_{ij}$. We assume that potential locations for hubs are placed at nodes of the network. For each potential location $k$,  we denote by $F_k$ the setup cost for activating a hub at node $k$. We also assume that some edges can be activated as \textit{interhub edges} (or just \textit{hub edges}). Any potential hub edge $km$ can be activated incurring a cost $G_{km}$, provided that its two endnodes $k$ and $m$ are activated as hub nodes as well. We assume that: ($i$) $N$ is a complete network, i.e. for all $i, j\in V$, $i\ne j$, arcs $(i,j), (j,i)\in A$; ($ii$) the set of potential locations  for hub nodes coincides with $V$; and ($iii$) the set of potential interhub edges coincides with $E$. These assumptions can be made without loss of generality since arbitrarily large routing costs $c_{ij}$ can be assigned to non-existing arcs, and arbitrarily large setup costs $F_k$ and $G_{km}$, respectively, can be assigned to nodes that are not potential hub nodes and to edges that are not potential hub edges. To alleviate notation, in the remainder of this paper any edge $km\in E$, $k<m$ will be indistinctively denoted by $mk$.
	
	Let $H\subset V$  be a given set of activated hubs and $\overline E\subseteq \{ij: i,j\in H\}$, a given set of activated interhub edges connecting hubs of $H$, not necessarily complete. The network $N_{\mathcal B} = (H, A(\overline E))$, where  $A(\overline E)$ is the set of arcs associated with $\overline E$, will be referred to as \textit{backbone network induced by $\overline E$}.

	Service demand is given by a set of commodities defined over pairs of users, indexed by a set $R$. Let $\mathcal{D}=\{(o^r, d^r, w^r): r\in R\}$ denote the set of commodities, where each triplet $(o^r, d^r, w^r)$ indicates that an amount of flow $w^r\ge 0$ must be routed  from origin node $o^r\in V$ to destination node $d^r\in V$. When the context is clear, we will simply use $o$ and $d$ to refer to the origin and destination of commodity $r$, respectively. If necessary, to make the endnodes explicit, we will write $w_{od}$ instead of $w^r$. Without loss of generality, we will assume that the graph induced by the commodities with $w^r>0$ is connected; otherwise, each connected component could be treated as an independent subproblem. Let $R^+_i=\{r\in R: o^r=i\}$ and $R^-_i=\{r\in R: d^r=i\}$ denote the index sets of commodities with origin and destination at a given node $i\in V$, respectively. For a given node $i\in V$, $O_i=\sum_{r\in R^+_i}w^r$ and $D_i=\sum_{r\in R^-_i}w^r$ denote the total amount of demand with origin and destination at $i$, respectively. Let also $\overline W=\sum_{r\in R}w^r$ denote the total demand through the network.

	The following additional notation will be used. For any node set $S\subset V$, the (undirected) cutset $\delta(S)=\left\{km\in E: (k\in S \text{ and } m\notin S) \vee (k\notin S \text{ and } m\in S)\right\}$ is the set of edges connecting two nodes on opposite sides of the bi-partition $(S, S^c)$. The di-cuts $\delta^+(S)=\{(i,j)\in A: i\in S, j\notin S\}$ and $\delta^-(S)=\{(i,j)\in A: i\notin S,  j\in S\}$ are the sets of arcs leaving $S$ and entering $S$, respectively. The total demand with origin in $S$ and destination in $S^c$ is denoted by $W(S:S^c)=W(\delta^+(S))=\sum_{r\in R: o^r\in S, d^r\notin S}{w^{r}}$. Finally, for any given vector $\overline f$ with support set $S$, and any subset $T\subseteq S$, we denote by $\overline f(T)=\sum_{a\in T}\overline f_a$.

	\begin{defin}\label{defi:rpath}
		Let $N_{\mathcal B} = (H, A(\overline E))$ be a given backbone network.
		\begin{itemize}
			\item A path $P^r \equiv o^r - k_1 - \dots - k_t- d^r$  is \textit{consistent} for commodity $r\in R$ if: $(a)$ every intermediate node is activated as a hub node, i.e., $k_i\in H$, for all $1\le i\le t$;  and $(b)$ each pair of consecutive intermediate nodes is connected by an activated interhub edge; i.e., $k_ik_{i+1}\in\overline E$, for all $1\le i\le t-1$.
			
			If the origin node $o_r$ is itself activated as a hub, i.e., $o_r\in H$, then $k_1=o^r$ and the first leg $o^r-k_1$ is considered a fictitious (empty) arc. Similarly, if the destination node $d_r$ is an activated hub, i.e., $d_r\in H$, then $k_t=d^r$ and the last leg $k_t-d^r$ is also considered a fictitious (empty) arc. When $o^r$ is not activated as a hub, then $o^r\ne k_1$ and arc $(o^r, k_1)$ will be called \textit{access arc}. Similarly, when $d^r$ is not activated as a hub, then $k_t\ne d^r$, and arc $(k_t, d^r)$ will be called \textit{distribution arc}.
			
			Arcs $(k_i, k_{i+1})$ will be called \textit{interhub arcs} (or \textit{hub arcs}). Abusing slightly notation, we will also say that hub edges belong to the backbone network.
			
			A consistent path for commodity $r\in R$ will also called an \textit{$r$-path}.
			\item The unit routing cost through $r$-path $P^r$ is $\gamma c_{o^rk_1}+\alpha \sum_{i=1}^{t-1} c_{k_i k_{i+1}}+ \theta c_{k_sd^r}$, where $0\le \alpha\leq 1$ is a given interhub discount cost factor and $\gamma, \theta >\alpha$ are factors applied to routing costs through access and distribution arcs, respectively.
		\end{itemize}
	\end{defin}

	\subsection{Problem definition}
		We study a \textit{generic} HLP (GHLP), in which a feasible solution is composed of a backbone network
		and a set of $r$-paths, one for each commodity $r\in R$. The network consisting of the backbone network together with the access and distribution arcs used in the $r$-paths is referred to as  \textit{solution network} and is denoted by $N_{\mathcal S}=\left(V,\, \cup_{\in R}P^r\right)$, where $P^r$ is the $r$-path of commodity $r$. The objective of the GHLP is to minimize the total cost, defined as the sum of: $(i)$ the set up costs of the activated hub nodes; $(ii)$ the set up costs of the activated hub edges; $(iii)$ the total cost for routing the commodities demands along their $r$-paths. According to the terminology introduced in Section \ref{sect:literature}, the GHLP corresponds to a $G$-median objective, but includes both $p$-median and $H$-median as special cases.
		
		As mentioned, in both $p$-median and $H$-median models, under the assumption that the routing costs satisfy the triangle inequality, there is an optimal solution in which the backbone network is a complete graph, and each $r$-path includes at most one hub arc. On the contrary, for $G$-median problems, the optimal backbone network may not be a complete graph, so optimal $r$-paths may involve multiple hub arcs, unless  additional restrictions are imposed on path lengths. In the following, unless  stated otherwise, we focus on the general $G$-median objective.
		
		We consider two different GHLP variants, which differ in how access and distribution arcs are defined when the origin or destination of a commodity is not an activated hub. In the \textit{single allocation} GHLP (SA-GHLP) each non-hub node is assigned to a single hub for both sending and receiving flow. Specifically, for any non-hub node $o$, all  commodities $r\in R_o^-$ (with origin $o$)   share the same access arc $(o,\,k_1)$, and all  commodities $r\in R_o^+$ (with destination $o$) share the same distribution arc $(k_1,\, o)$.	In this case, we say that node $o$ is allocated to hub $k_1$, which is the only hub node to which node $o$ can be connected. On the contrary, the \textit{multiple allocation} GHLP (MA-GHLP) allows a non-hub node to connect to multiple hubs. Then, different commodities with the same origin (or destination) may be routed through different first (or last) hubs. Thus, for a non-hub node $o$, the $r$-paths of commodities with origin $o$ may involve different access arcs, and the $r$-paths of commodities with destination $o$ may involve different distribution arcs. In this case, we  say that node $o$ is allocated to all the hubs involved in the access or distribution arcs of the affected  $r$-paths.
		
	\subsection{Solutions representation}
		Any $r$-path $P^r$ associated with a commodity $r\in R$ induces a flow of value $w^r$ through its arcs. This flow can be represented by a vector $f^r\in \mathbb R^{|A|}$ where $f^r_a=w^r$ for all $a\in P^r$ and  $f^r_a=0$ for all $a\in A\setminus \cup_{r\in R}P^r$. The routing cost of $f^r$ is $w^r\left(\gamma c_{o^rk_1}+\alpha \sum_{i=1}^{t-1} c_{k_i k_{i+1}}+ \theta c_{k_td^r}\right)$. Thus, a solution to the GHLP can be represented by a pair $(s, f)$, where the \textit{design solution}  $s$  defines the structure of the solution network (see Section \ref{sec:formula}) and the \textit{flow solution} $f$ specifies the amount of demand routed through each arc. Next we state some feasibility and optimality conditions for such flows, given a fixed design solution $s$. Let ${\mathcal S}$ denote the domain for feasible design solutions.
        
		\begin{defin}\label{def:flows}
			Let $s\in \mathcal S$ be a given design solution and $f\in\mathbb R^{|A|}$ be a given \textit{flow vector}.
			\begin{itemize}
				\item $f$ is feasible for $s$ if $f=\sum_{r\in R}f^r$ where for each commodity $r\in R$, $f^r$ is a flow of value $w^r$ through an $r$-path in the solution network associated with $s$.
				\item An $r$-path of minimum unit routing cost in the solution network associated with $s$, is called a \textit{shortest $r$-path} ($s$-$r$-path).
				\item $f$ is $R$-feasible for $s$ if $f=\sum_{r\in R}f^r$ where, for every $r\in R$, $f^r$ is a flow of value $w^r$ through an $s$-$r$-path in the solution network determined by $s$.
			\end{itemize}
		\end{defin}
        
		\begin{rem}\label{Remark1}
			\begin{itemize}
				\item According to Definition \ref{def:flows}, a \textit{feasible} flow is one that can be disaggregated into individual flows induced by $r$-paths, one for each commodity. This is clearly a necessary condition for the feasibility of the considered flows.
				\item We can moreover restrict the search for optimal GHLP solutions to those where the  flows are $R$-feasible, i.e., they are routed through shortest $r$-paths in the underlying solution network.
				
				Suppose, on the contrary that there is an optimal GHLP solution $(s^*, f^*)$ such that $f^*$ is a feasible flow that is not $R$-feasible. Then, there must exist some commodity $\overline r\in R$ such that the demand $w^{\overline r}$ is routed through an $r$-path, $P^{\overline r}$ that is not of minimum unit cost. Hence, the solution network induced by $s^*$ contains an $s$-$r$-path $\overline P^{\overline r}$ with a (strictly) better routing cost than $P^{\overline r}$, contradicting the optimality  of $(s^*, f^*)$.
			\end{itemize} 
		\end{rem}
	
		Hence, in the following, feasible flows that are not $R$-feasible will not be considered and we will state the GHLP with the following high-level formulation:
		\begin{align}
			(GHLP) \qquad \min\quad & val(s, f) &&\nonumber\\
			& f \in\Omega(s), &&\nonumber\\
			& s\in {\mathcal S} &&\nonumber
		\end{align} 
        
		\noindent were $\Omega(s)$ denotes the set of $R$-feasible flows associated with a given design solution $s\in\mathcal S$, and $val(s, f)$ is the objective function value of solution $(s, f)$, which can be separated in two terms: one for the design solution (setup costs of activated hubs and interhub edges), $v_{act}(s)$, and another one for the routing costs of the flows, $v_{rout}(f)$. That is, $val(s, f)= v_{act}(s)+v_{rout}(f)$.
	
		For a given design solution $s\in\mathcal S$, we define $\Omega(s)$ as the intersection of two domains $\Omega(s)=\mathcal F(s)\cap\mathcal{L}(s)$. $\mathcal F(s)$  contains \textit{reinforced aggregated flows}, not necessarily $R$-feasible, that satisfy balance constraints as well as \textit{aggregated demand constraints}. $\mathcal{L}(s)$ is the domain of $R$-feasible flows for $s$. Instead of expressing this domain in terms of inequalities to potentially reinforce the domain $\mathcal F(s)$, we will handle  $\mathcal{L}(s)$ implicitly, and use it for eliminating nodes of the enumeration tree and, possibly, for updating the incumbent solution and its value. The main idea of the approach that we propose is to integrate the following strategies:
        
		\begin{itemize}
			\item[$(1)$:] to express $\mathcal F(s)$ with a \textit{small} number of decision variables, so the relaxed problem
			\begin{equation*}
				(R-GHLP)\, \min\,_{s\in \mathcal S,\, f \in\mathcal F(s)} \left\{\,v_{act}(s)+v_{rout}(f)\,\right\}
			\end{equation*}
			\noindent can be efficiently handled without having to project out any set of decision variables, and
			\item[$(2)$:] to \textit{postpone} exploring $\mathcal{L}(s)$ to \textit{promising} design solutions $s\in \mathcal S$ only.
		\end{itemize}
		
		\noindent The rationale behind this solution framework relies in the following two observations:
			\begin{itemize}
				\item An $R$-feasible flow for a given solution network $s\in \mathcal S$ is not feasible for a different solution network $s'\in \mathcal S$, $s'\ne s$, since at least one commodity will be routed through a different $s$-$r$-path.
				\item   All $R$-feasible flows corresponding to the same solution network $s\in \mathcal S$ yield the same objective function value. That is, for a given $s\in \mathcal S$,  $v_{rout}(f)$ is the same for all $f\in\mathcal L(s)$. Hence, for each promising solution network $s\in\mathcal S$, it suffices to find just one $R$-feasible flow.
			\end{itemize}

\section{A new formulation for the GHLP} \label{sec:formula}
	We now present the details of a new formulation for the GHLP that uses 1- and 2-index decision variables only. With minor modifications it applies to the SA and MA variants explained above. The formulation combines binary variables associated with the design decisions for the elements activated in the solution network, with continuous variables representing aggregated flows circulating through the different types of arcs. In particular, the formulation explicitly models the backbone network as well as access and distribution arcs, while it does not make explicit the individual $r$-paths for the commodities. Routing costs (discounted or not, depending on the case) are expressed in terms of the overall flows that circulate through the different types of arcs.
	
	\subsection{The domain $\mathcal S$}\label{Domain:S}
		Let $s=(z, y, x^1, x^2)$ denote the components of a design solution $s\in {\mathcal S}$, which are associated with the following sets of binary decision variables:
		\begin{itemize}
			\item $z_{k}\in\{0, 1\}$, $k\in V$.  $z_{k}=1$ if and only if a hub is activated at node $k$.
			\item $y_{km}\in\{0, 1\}$, $km\in E$. $y_{km}=1$ if and only if hub edge $km$ is activated.
			\item $x^1_{ij}\in\{0, 1\}$, $(i, j)\in A$. $x^1_{ij}=1$ if and only if access arc $(i, j)$ is activated.
			\item $x^2_{ij}\in\{0, 1\}$, $(i, j)\in A$. $x^2_{ij}=1$ if and only if distribution arc $(i, j)$ is activated.
		\end{itemize}
		The design cost of a solution $s=(z, y, x^1, x^2)$ is given by $$v_{des}(s)=F(z)+G(y)=\sum_{k\in V}f_kz_k+\sum_{km\in E}g_{km}y_{km}.$$
	
		\noindent For the MA policy, the formulation that we use to characterize $\mathcal S$ is:
		\begin{subequations}
			\begin{align}
				& \begin{cases}x^1_{ij}+x^2_{ij}+y_{ij}\leq 1\\
					x^1_{ji}+x^2_{ji}+y_{ij}\leq 1
				\end{cases}  && ij\in E \label{rel_X-Y-0}\\
				& \begin{cases}x^1_{ij}+y_{ij}\leq z_j\\
					x^1_{ji}+y_{ij}\leq z_i\end{cases}   && ij\in E \label{rel_z-Y-1_1-0}\\
				& \begin{cases}x^2_{ij}+y_{ij}\leq z_i\\
					x^2_{ji}+y_{ij}\leq z_j \end{cases}  && ij\in E \label{rel_z-Y-1_2-0}\\
                & \begin{cases} x^1_{ij}+ z_i\leq 1 \\
					x^2_{ij}+z_j\leq 1 \end{cases}  &&  (i, j)\in A \label{rel_z-X-2_2-bis}\\
				& \begin{cases} \sum_{j\in V\setminus\{i\}}x^1_{ij}+ z_i\geq 1 \\
					\sum_{j\in V\setminus\{i\}}x^2_{ji}+z_i\geq 1 \end{cases}  && i\in V \label{new-MA}\\
				& z_{i} \in \left\{0,1\right\} && i\in V \label{domain_z-0}\\
				& x^1_{ij}, x^2_{ij} \in \left\{0,1\right\} && (i, j)\in A \label{domainX-0}\\
				& y_{km} \in \left\{0,1\right\} && km\in E. \label{domainY-0}
			\end{align}
		\end{subequations}

		Constraints \eqref{rel_X-Y-0} establish that each edge and associated arc can be activated in at most one of the three possible classes (access/distribution directed arc or interhub undirected edge). Constraints \eqref{rel_z-Y-1_1-0}-\eqref{rel_z-Y-1_2-0} regulate the relation of activated interhub edges and access/distribution arcs with activated hub nodes. They impose, that $i)$ both endnodes of interhub edges are activated as hub nodes, $ii)$ the destination node of an access arc is a hub node, and $iii)$ the origin node of a distribution arc is a hub node. By \eqref{rel_z-X-2_2-bis} just one of the two endnodes of an access or distribution arc is a hub, whereas \eqref{new-MA} impose that non-hub nodes are the origin of at least one access and one distribution arc.
	
		For SA, \eqref{new-MA} must hold as equality so \eqref{rel_z-X-2_2-bis} are no longer needed. Moreover, in this case,  $x^2_{ij}=x^1_{ji}$ must also hold for all $(i,j)\in A$,  the second inequality in each of the above blocks can be removed.
	
		\subsection{The domain for reinforced aggregated flow vectors  $\mathcal F(s)$}\label{Domain:F}
			Throughout this section we assume that $s\in\mathcal S$ is a given design solution which, for clarity, will be denoted as $\overline s=(\overline z,\, \overline y,\, \overline x^1,\, \overline x^2)\in{\mathcal S}$. Then, the components of any flow vector $f\in {\mathcal F(\overline s)}$ are denoted by $f=(t, \, h^1, \, h^2)$, and are defined through the following sets of continuous decision variables:
			\begin{itemize}
				\item $h^1_{ij}$: Total flow circulating through access arc $(i, j)\in A$.
				\item $t_{ij}$: Total flow circulating through interhub arc $(i,j)\in A$.
				\item $h^2_{ij}$: Total flow circulating through distribution arc $(i, j)\in A$.
			\end{itemize}
	
		In terms of the above decision variables, the routing cost of $f=(t, \, h^1, \, h^2)$ is
		\begin{equation*}
			v_{rout}(f)=\sum_{(i,j) \in A} c_{ij}\left (\gamma h^1_{ij}+ \alpha t_{ij}+ \theta h^2_{ij}\right).
		\end{equation*}
			
		Then, the domain $\mathcal F(\overline s)$ is determined by the following set of constraints:\vspace{-0.5cm}
        
		\begin{subequations}
			\begin{align}
				& (1-\overline z_i)\,O_i=\sum_{(i,j)\in A}h^1_{ij} &&  i\in V \label{leaves_i_H-0}\\
				& (1-\overline z_i)\,D_i=\sum_{(j,i)\in A}h^2_{ji} &&  i\in V \label{enters_i_H-0}\\
				&  O_i\,\overline z_i+\sum_{j\ne i}h^1_{ji}+\sum_{j\ne i}t_{ji}\notag \\
				& \qquad \qquad =D_i\,\overline z_i+ \sum_{j\ne i}h^2_{ij}+\sum_{j\ne i}t_{ij} &&  i\in V \label{Flow_Balance-0}\\				& (t+h^1+h^2)(\delta^+(S))\geq W(S:S^c) &&  S\subset V:\,\exists r\in R, \text{ s.t. } (o^r, d^r)\in\delta^+(S)\label{Demand-0}\\
				& w_{ij}\,\overline x^1_{ij}\leq h^1_{ij}\leq O_i\,\overline x^1_{ij} &&  (i, j)\in A \label{H_X2_1-0}\\
				& w_{ij}\,\overline x^2_{ij}\leq h^2_{ij}\leq D_j\,\overline x^2_{ij} &&  (i, j)\in A \label{H_X2_2-0}\\
				& (w_{km}+w_{mk})\,\overline y_{km}\leq t_{km}+t_{mk}\leq \overline W y_{km} &&   km \in E \label{f_Y-0} \\
				& t_{ij}, h^1_{ij}, h^2_{ij}\geq 0 &&  (i,j)\in A.\label{domain_F-H-0-0}
			\end{align}
		\end{subequations}
		Equations \eqref{leaves_i_H-0}-\eqref{enters_i_H-0} are flow balance constraints at non-hub nodes, which impose that when $i$ is not a hub, the overall demand with origin and destination at $i$ must be routed with flows of types $h^1$ and $h^2$, respectively. Flow balance at hub nodes is enforced by Constraints \eqref{Flow_Balance-0}. These constraints are the aggregation, over all origins, of the balance constraints with 3-index flow variables $f^i_{km}$ discussed in Section \ref{sect:literature} \citep[see, e.g. ][]{Ernst96}. Hence, they may produce flows that cannot be disaggregated into $r$-paths.	 In order to strengthen the domain of feasible flows we introduce the \textit{aggregated demand constraints} \eqref{Demand-0}, which  can be explained as follows. For any subset of nodes $S\subset V$, all $r$-paths corresponding to commodities $r\in R$ such that $o^r\in S$ and $d^r\in S^c$ must \textit{cross} the di-cut $\delta^+(S)$. Hence the total flow exiting $S$, must be at least equal to the total demand from $S$ to $V\setminus S$, i.e. $W(S:S^c)$. Although this set of constraints still do not guarantee that flows can be disaggregated in individual $r$-paths, it significantly restricts the domain of the flows and improves the LP bounds. To the best of our knowledge, these constraints are novel for HLPs.
	
		Inequalities \eqref{H_X2_1-0}-\eqref{H_X2_2-0} impose that the flows of types $h^1$ and $h^2$ solely circulate through access and distribution arcs, respectively, and, in each case, establish lower and upper bounds on the flows through such arcs.	Finally, constraints \eqref{f_Y-0} regulate the relationship between flows through interhub arcs and activated interhub edges, where $\overline W$ serves as a conservative upper bound on the maximum total flow that can circulate in the two directions of each edge. Note that, by changing the values of the coefficients in \eqref{H_X2_1-0}-\eqref{f_Y-0}, we can immediately impose lower or upper bounds on the flows that circulate through the activated arcs.
 
		For the SA policy, \eqref{Demand-0}--\eqref{H_X2_2-0}, can be reinforced to
		\begin{align}
			& (t+h^1+h^2)(\delta^+(S))\geq W(S:S^c)\notag\\
			&\qquad+\sum_{ij\in S}w_{ij}\sum_{k\in S^c}x^1_{ik}+\sum_{ij\in S^c}w_{ij}\sum_{k\in S}x^1_{ik} && S\subset V:\,\exists r\in \delta^+(S)\tag{\ref{Demand-0}-SA}\label{Ag:dem_reinf}\\
			& h^1_{ij}= O_i\overline x^1_{ij} &&(i, j)\in A \tag{\ref{H_X2_1-0}-SA}\\
			& h^2_{ij}= D_j\overline x^2_{ij} &&(i, j)\in A, \tag{\ref{H_X2_2-0}-SA}
		\end{align}
         
        \noindent where the right hand side of \eqref{Ag:dem_reinf} is extended to consider the demand of all the commodities with origin and destination at the same side of the cutset, whose non-hub origin is allocated to a hub on the opposite side of the cutset.

		Integrating  the formulations for the domains $\mathcal S$ and $\mathcal F(s)$ we obtain a formulation for R-GHLP, which is given in  Appendix \ref{appendix:R-GHLP}, for the MA policy. As mentioned,  this formulation may produce optimal solutions $(s^*, f^*)$ with $s^*\in \mathcal S$, $f^*\in\mathcal F(s^*)$, in which $f^*$ is not $R$-feasible. This is illustrated in the following example.
	
		\begin{example}\label{example1}
			Consider an MA-GHLP instance from the CAB dataset with $n=10$ and $\alpha=0.8$. Unit routing costs ($c_{ij}$) and commodities demands ($w_{ij}$) are given in Tables \ref{tab:cab10_c} and \ref{tab:cab10_w}, respectively.
		
			The R-GHLP formulation of Appendix \ref{appendix:R-GHLP} without the aggregated demand constraints, gives the solution $(s^1, f^1)$ depicted in Figure \ref{fig:cab10_nodmnd_flows}, with an objective value of $1.616\times 10^7$. The activated hubs and backbone edge are $z^1_4 = z^1_5=1$, $y^1_{45}=1$. The values of the flows $f^1$ are shown over the arcs; due to the symmetry of the routing costs, they are the same in both directions. This solution violates the aggregated demand constraint \eqref{Demand-0} for the node set $S=\{3,\,5,\,6,\,7,\,8,\,10\}$. In particular, the overall flow through the cutset $\left(t+ h^1+h^2\right)(\delta^+(S))= t(\delta^+(S))=452$, whereas $W(S:S^c)=10586$.
		
			When the aggregated demand constraints are included, the R-GHLP formulation of Appendix \ref{appendix:R-GHLP}  produces the solution $(s^2, f^2)$ depicted in Figure \ref{fig:cab10_dmnd_flows}, with an objective value of $2.301\times 10^7$. Now, the activated hubs and backbone edge are $z^2_5=z^2_9=1$, $y^2_{59}=1$, and the values of the flows $f^2$ are shown over the arcs. Even if these flows satisfy all the aggregated demand constraints, they still fail in being $R$-feasible. Note that, in this solution nodes 7 and 10 are allocated only to hub 5, and nodes 2 and 4 are allocated only to hub 9. Hence, in any feasible solution with design solution $s^2$ the overall flow through the interhub arc (5, 9) must be at least $w_{52}+w_{54}+w_{59}+w_{72}+w_{74}+w_{79}+w_{10,2}+w_{10,4}+w_{10,9}=8067$, which is greater than the flow $t^2_{59}=8051$.
		
			The optimal GHLP solution for this instance, $(s^*, f^*)$, has an objective value of $2.302 \times 10^7$. It has the same design solution as R-GHLP, $s^*=s^2$, and $R$-flows as depicted in Figure \ref{fig:cab10_feas_flows}.			
		\end{example}
	
	\subsection{Connectivity inequalities for the backbone network}
		Below we discuss the family of inequalities
		\begin{align}
			&Y(\delta(S))\geq z_i+z_j-1 && \forall S\subset V, i,j\in V \text{ s.t. } i\in S, j\notin S,\label{connect-backbone-0}
		\end{align}
         
		\noindent imposing the connectivity of the backbone network: for any node set $S$ that contains one activated hub in both sides of its cutset, at least one interhub edge must be activated in the cutset $\delta(S)$.
		
		First, we observe that \eqref{connect-backbone-0} are not necessarily valid neither for the SA nor the MA policies.
		Suppose the inequality is violated for a given solution $(\bar s, \bar f)$ with $\bar s\in\mathcal S$, and $\bar f\in \mathcal F(s)$. Then there must exist $S\subset V$, $i\in S$,  $j\in V\setminus S$ such that $\bar y(\delta(S))=0$ and $\bar z_i=\bar z_j=1$. In fact, since $\bar y(\delta(S))=0$,  the inequality will also be violated by any pair of nodes, $k, m$ with $k\in S$,  $m\in V\setminus S$ such that $\bar z_k=\bar z_m=1$.
		Hence, 
        $$ (\bar t+\bar h^1+\bar h^2)(\delta^+(S)) = (\bar h^1+\bar h^2)(\delta^+(S))\leq W(S: S^c)-\sum_{\substack{k\in S,\, m\in S^c\\ \bar z_k=\bar z_m=1}} w_{km},$$
		where the inequality follows since the demand $w_{km}$ of any commodity with $k\in S$,  $m\in S^c$, such that $\bar z_k=\bar z_m=1$,  may \textit{cross} the dicut $\delta^+(S)$ through interhub arcs only, because both $k$ and $m$ are activated hubs. Therefore,
		$$w_{km}=0,\, \forall\, k\in S, m\in V\setminus S,\ \text{s.t. }\bar z_k=\bar z_m=1,$$
		\noindent since, otherwise, the aggregated demand constraint \eqref{Demand-0} would be violated.
		
		As shown in Figure \ref{fig:connect}, feasible solutions with the above characteristics may exist for  both SA and MA. That is: in principle, inequalities \eqref{connect-backbone-0} are not valid. Still, feasible solutions that violate these inequalities, are only possible for instances where some o/d pairs at different sides of a given cutset have no demand between them. While, in principle, this could be possible, the benchmark instances available in the literature have a strictly positive demand between every pair of nodes. For this reason, in the remainder of this paper we will include this family of inequalities in the definition of the domain $\mathcal S$.
		
	\subsection{The domain $\mathcal{L}(s)$ of $R$-feasible flows for a given solution $s\in\mathcal S$}\label{R-feas0}
		Let $G(\overline s)$ be the support graph induced by $\overline s=(\overline z,\, \overline y,\, \overline x^1,\, \overline x^2)\in{\mathcal S}$, which contains undirected edges induced by $\overline y$ and directed arcs induced by $\overline x^1$ and $\overline x^2$. To distinguish flows in $\mathcal F(\overline s)$ from $R$-feasible flows in $\mathcal L(\overline s)$, we keep the notation $\overline f=(\bar t, \bar h^1, \bar h^2)\in \mathcal F(\overline s)$ for the former  and denote the latter by $f(\overline s)=(t(\overline s),\,  h^1(\overline s), \,  h^2(\overline s))\in \mathcal L(\overline s)$.	Let also $P^r(\overline s)$ be the $s$-$r$-path for commodity $r$ in $G(\overline s)$.
	
		Below we detail two alternatives for finding $R$-feasible flows in $\mathcal{L}(\overline s)$ for a given $\overline s\in\mathcal S$ and for using them algorithmically:
		\begin{itemize}
			\item[$(i)$]  The first alternative is to find just one $R$-feasible flow $f(\overline s)\in \mathcal L(\overline s)$  by explicitly identifying an $s$-$r$-path, for every commodity $r\in R$, in the solution network induced by $\overline s$. Such paths can be easily found with shortest path algorithms (see,~e.g. \cite{Dijkstra}).
			
			After finding the $R$-feasible flow $f(\overline s)\in \mathcal L(\overline s)$, we can update the current incumbent solution to $(\overline s,\, f(\overline s))$ if $v_{act}(\overline s)+v_{rout}(f(\overline s))$ is smaller than the value of the current best-known solution, but we will not derive any optimality cuts from this information. The rationale for this alternative is that, if $\overline s$ is not an optimal design solution, then the information relative to $s$-$r$-paths for $\overline s$ may not be very relevant for finding $s$-$r$-paths relative to an optimal design solution (some $s$-$r$-path will necessarily change) so we may be overloading the formulation with unnecessary cuts.
			\item[$(ii)$] The second alternative is to proceed as above but, in addition, use information from $f(\overline s)\in \mathcal L(\overline s)$ to derive feasibility or optimality cuts for $R$-feasible flows. This alternative will be further developed in the next sections.
		\end{itemize}

\section{Feasibility Benders cuts for GHLP}\label{sec:FeasBenders}
    Benders cuts have been developed for various HLPs \cite[see, e.g.,][]{Contrerasetal11,deCamargo2009,Taherkhanietal20,Wandelt22}. These cuts have been derived for formulations with 3- or 4-index variables, with the purpose of  projecting out their large number of continuous variables. Our intention is not to project out the flow variables of our GHLP formulation, which have two indices only, but to derive feasibility Benders-type valid inequalities which can be used to enforce that the obtained flows are $R$-feasible for their underlying solution networks.
    
    Consider a feasible solution $(\bar s, \bar f)$ with  $\bar s=(\bar z,\, \bar y,\, \bar x^1,\,\bar x^2)$ and $\bar f\in\mathcal F(\bar s)$. To derive feasibility Benders cuts we define an auxiliary problem, in which we try to route jointly the individual commodities through the network $G(\bar s)$, using the values $\bar t$, $\bar h^1$, and $\bar h^2$ as capacities for the different types of arcs. For this, we define the following 3-index (flow) decision variables:
    \begin{itemize}
        \item $t^i_{km}$: flow with origin at node $i\in V$ routed through interhub arc $(k, m)\in A$.
        \item $g_{ik}$: flow routed through access arc $(i, k)\in A$.
        \item $h^i_{kj}$: flow with origin at node $i$ routed through distribution arc $(k, j)\in A$.
    \end{itemize}
    
    The auxiliary subproblem is:
    \begin{subequations}
        \begin{align}
        Aux_{(\bar s, \bar f)}\quad \min\ & \sum_{i\in V} \bigg[\sum_{k\in V\setminus\{i\}}\gamma c_{ik}g_{ik} +\notag \\
        & \qquad \sum_{(k, j)\in A} \left(\alpha c_{kj} t^i_{kj}+ \theta c_{kj} h^i_{kj}\right)\bigg] && \label{of_A2}\\
        \text{s.t.} & \sum_{k\in V: k\ne i}g_{ik}=O_i(1 - \bar z_i) &&  i\in V \label{i_access_A2}\\
        & \sum_{(i,j)\in A}\left(t^i_{im}+h^i_{ij}\right)=O_i\bar z_i &&  i\in V \label{Flow_Balance_A2}\\
        & \sum_{(k, j)\in A}h^i_{kj}=w_{ij}(1 - \bar z_j) &&  (i, j)\in R \label{leaves_j_H_A2}\\
        & \sum_{{(k,m)\in A: m\ne i}}\left(t^i_{km}+h^i_{km}\right)-\notag\\
        & \qquad \qquad g_{ik}-\sum_{(m,k)\in A}t^{i}_{mk}=-w_{ik}\bar z_k  &&  i, k\in V, k\ne i \label{Flow_Balance_A2_bis}\\
        & \sum_{i\in V: i\ne m}t^i_{km}\leq \bar t_{km} &&  (k, m)\in A \label{F-Delta_A2}\\
        & g_{ik}\leq \bar h^1_{ik} &&  (i, k)\in A  \label{bound_H_A2}\\
        & \sum_{i\in V: i\ne j}h^i_{kj}\leq \bar h^2_{kj}&&   (k, j)\in A  \label{leaves2_j_H_A2}\\
        & g_{ik},t^i_{kj}, h^i_{kj} \geq 0  &&  (k, j)\in A, i\in V.\label{domainA2}
        \end{align}
    \end{subequations}

    While \eqref{i_access_A2}-\eqref{Flow_Balance_A2_bis} are flow balance constraints that guarantee that all the commodities are routed through $r$-paths in the graph $G(\overline s)$, Constraints \eqref{F-Delta_A2}-\eqref{leaves2_j_H_A2} relate the value of the flows through the arcs with those of $(\overline t,\, \overline h^1,\, \overline h^2)$. Because of the objective function, when $Aux_{(\bar s, \bar f)}$ is feasible, the flows $t_{km}=\sum_{i\in V}t^i_{km}$, for all $(k,m)\in A$;  $h^1_{ik}=g_{ik}$ for all $(i,k)\in A$; $h^2_{kj}=\sum_{i\in V}h^i_{kj}$ determine an $R$-feasible flow for $\bar s$. Note that in this case the following feasibility conditions must hold:
    
    \begin{subequations}
        \begin{align}
            (i)\quad & \bar t_{km}\geq t_{km}(\bar s) && \forall (k, m)\in A \label{feas1}\\
            (ii)\quad & \bar h^1_{ik}\geq h^1_{ik}(\bar s) && \forall (i, k)\in A \label{feas2}\\
            (iii)\quad & \bar h^2_{kj}\geq h^2_{kj}(\bar s) && \forall (k, j)\in A \label{feas3}
        \end{align}
    \end{subequations}
    
    When some of the above bounding conditions does not hold, then $Aux_{(\bar s, \bar f)}$ will be unfeasible. In such a case, by analyzing its dual we may derive feasibility cuts. 	
    The dual of $(Aux)$ is:
    \begin{subequations}
        \begin{align}
            D-{Aux}_{(\bar s, \bar f)}  \hspace{8.pt} \max & \hspace{5.pt} \Phi(\Theta; \bar s, \bar f)\\
            &  \mu_{i}-\rho_{im}-\sigma_{im}\leq \alpha c_{im} &&  (i, m)\in A \label{D2_D3}\\
            & -\rho_{ik}+\lambda_i-\tau^1_{ik}\leq \gamma c_{ik} && (i, k)\in A \label{D2_D41}\\
            & \rho_{ik}+\pi_{ij}-\tau^2_{kj}\leq \theta c_{kj} &&  (k,j)\in A, i\in V\setminus\{k,j\}\label{D2_D4_1}\\
            & \mu_{i}+\pi_{ij}-\tau^2_{ij}\leq \theta c_{ij} &&  (i,j)\in A\label{D2_D4_2}\\
            & \sigma_{ij}, \tau^1_{ij}, \tau^2_{ij}\geq 0 &&  (i, j)\in A \\
            & \mu_{i}, \lambda_{i}\in\mathbb{R}&& i\in V \\
            & \pi_{ij}\in\mathbb{R} && (i, j)\in R.  \label{domainD}
        \end{align}
    \end{subequations}
    \noindent where $\Theta = (\lambda, \mu, \pi, \rho,\sigma,\tau^1,\tau^2)$. Then, the objective is
    $$\Phi(\Theta; \bar s, \bar f)=\sum_{i\in V}\left[O_i\left(1-\bar z_i\right)\lambda_i + O_i\bar z_i\mu_{i}\right] + \sum_{(i, j)\in R}w_{ij}(1-\bar z_j)\pi_{ij}-\sum_{(i, j)\in R}w_{ij}\bar z_j \rho_{ij}-$$
    $$\hspace{2.5 cm}-\sum_{(k, m)\in A}\left(\bar t_{km}\sigma_{km} + \bar h^1_{km}\tau^1_{km} +\bar h^2_{km}\tau^2_{km}\right).$$
    
    When $Aux_{(\bar s, \bar f)}$ is infeasible, $D-{Aux}_{(\bar s, \bar f)}$ will be unbounded. Let $\mathcal{E}=\{\Theta^\varepsilon=\left(\lambda^\varepsilon, \mu^\varepsilon, \pi^\varepsilon, \rho^\varepsilon,\sigma^\varepsilon, {\tau^1}^\varepsilon, {\tau^2}^\varepsilon\right): \varepsilon\in I^{\mathcal{E}}\}$ denote the set of extreme rays for $D-{Aux}_{(\bar s, \bar f)}$. Then, the following inequalities are valid for any feasible solution to GHLP:
    $$\Phi(\Theta^\varepsilon; \bar s, \bar f)\leq 0, \qquad \varepsilon\in I^{\mathcal{E}},$$
    \noindent where $\Phi(\Theta^\varepsilon; \bar s, \bar f)$ is the objective function value of $D-{Aux}_{(\bar  s, \bar  f)}$ for the extreme ray $\Theta^\varepsilon\in \mathcal{E}$.
    
    By substituting and rearranging terms, the above set of feasibility cuts can be rewritten as:
    \begin{align}
        &\sum_{(k, m)\in A}\left(\sigma_{km}^\varepsilon \bar t_{km}+ (\tau^1)^\varepsilon_{km} \bar h^1_{km} + (\tau^2)^\varepsilon_{km} \bar h^2_{km}\right) \nonumber\\
        &\qquad \geq \sum_{i\in V}\mu_{i}^\varepsilon O_i \bar z_i + \sum_{i\in V}\pi_i^\varepsilon O_i(1- \bar z_i)+\notag\\
        &\qquad \quad \sum_{i, j\in V: i\ne j}\rho_{ij}^\varepsilon w_{ij} (1-\bar z_j)- \sum_{i, k\in V: i\ne k}\lambda_{ik}^\varepsilon w_{ik} \bar z_k,\notag\\
        &\qquad \qquad \varepsilon\in I^{\mathcal{E}},\, \bar s\in\mathcal S,\, \bar f\in\mathcal F(\bar s). \label{Feas:D2}
    \end{align}

\section{The special case with at most one interhub arc: characterization of $\mathcal L(s)$}\label{sec:1-inter}
	In this section we  consider the special case of the GHLP when $r$-paths contain at most one interhub arc. We will focus on the $H$-median, although the developments of this section are also valid for the $G$-median with the additional constraint that service routes contain at most one interhub arc.
	
	Next, we derive optimality inequalities for feasible solutions $(s, f)$, $s\in\mathcal S$, $f\in\mathcal L(s)$, whose objective function value is better than that of a given solution $(\bar s, \,\bar f)$. Hence, we assume that a feasible solution $(\bar s, \,\bar f)$ is given, with $\bar s\in\mathcal S$, $\bar f\in\mathcal L(\bar s)$, and we will derive inequalities that must be satisfied by any feasible solution $(s,\, f)$ such that $v_{act}(s)+\,v_{rout}(f) < v_{act}(\bar s)+\,v_{rout}(\bar f)$.
	
	Let ${R}_{ij}(\overline s)$, ${R}^1_{oi}(\overline s)$, and ${R}^2_{jd}(\overline s)$ respectively denote the set of commodities such that $P^r(\overline s)$ uses interhub arc $(i, j)$, access arc $(o,i)$, and distribution arc $(j,d)$. For any $r\in R$, the unit routing cost through an $r$-path of the form $o^r - k - m - d^r$ will be denoted by $C_{km}^r= \gamma c_{o^rk}+\alpha c_{km}+ \theta c_{md^r}$.  We also define the following two sets:
	\begin{itemize}
		\item $Z^{r}(\overline s)=\{k\in V: C^r_{kk}< C^r(\overline s)\}$.
		\item $Y^r(\overline  s)=\{km\in E: \min\{C^r_{km}, C^r_{mk}\}< \min\{C^r(\overline s), C^r_{kk}, C^r_{mm}\}\}$.
	\end{itemize}
    
	$Z^{r}(\overline s)$ contains all potential hubs that would yield an $r$-path with one single intermediate hub node, whose routing cost is smaller than that of $P^r(\overline s)$, whereas $Y^r(\overline s)$ contains all potential interhub edges that would yield an $r$-path with one intermediate hub arc, whose routing cost is smaller than that of $P^r(\overline s)$ and also smaller than $C^r_{kk}$ and $C^r_{mm}$ (i.e., the $r$-path $o-k-m-d$ is better than paths $o-k-d$ and $o-m-d$). Since $P^r(\overline s)$ is an $s$-$r$-path in $G(\overline s)$, then $k\notin Z^{r}(\overline s)$ for all activated hubs (i.e., such that $\overline z_k=1$) and $km\notin Y^r(\overline s)$ for all activated hub edges (i.e., such that $\overline y_{km}=1$).
	
	In our analysis, without loss of generality, we make the assumption that when $C^r_{km}=C^r_{kk}$ the path with one single intermediate hub will be chosen.
	\begin{propo}\label{propo0}
		Let $\overline s\in {\mathcal S}$ be a given solution network and $f(\overline s)\in\mathcal L(\overline s)$ an $R$-feasible flow in $G(\overline s)$. Then, the following inequalities are valid for the $H$-median: 
		\begin{subequations}
			\begin{align}
			& t_{ij}+\sum_{r\in{R}_{ij}(\overline s)} w^r\left[\sum_{k\in Z^{r}(\overline s)} z_k+\sum_{\substack{km\in Y^{r}(\overline s):\\ k,m\notin Z(\overline s)^{r}}} y_{km} \right] \geq t_{ij}(\overline s)\, y_{ij} && ij\in E,\label{sp:1_a}\\
            & t_{ji}+\sum_{r\in{R}_{ji}(\overline s)} w^r\left[\sum_{k\in Z^{r}(\overline s)} z_k+\sum_{\substack{km\in Y^{r}(\overline s):\\ k,m\notin Z(\overline s)^{r}}} y_{km} \right] \geq t_{ji}(\overline s)\, y_{ij} && ij\in E.\label{sp:1_b}
			\end{align}
		\end{subequations}
	\end{propo}

    \begin{proof}
		The proof is given in Appendix \ref{appendix:proofs}.
	\end{proof}
	Broadly speaking, inequality \eqref{sp:1_a} states that, if  interhub edge $ij$ is activated in $G(\overline s)$, then, in any $R$-feasible flow for a solution network where $ij$ is also an activated hub edge,   the total flow through arc $(i, j)$ must be at least $t_{ij}(\overline s)$ unless, for some $r\in {R}_{ij}(\overline s)$, a hub is activated at some node of $Z^{r}(\overline s)$ or an interhub edge is activated for some edge of $Y^{r}(\overline s)$. A similar interpretation applies to \eqref{sp:1_b} for the flow circulating through $(j, i)$. The logic of the left hand side of these inequalities is to \textit{compensate} the demand of any commodity that is routed through $(i,j)$ (or through $(j, i)$) in  $G(\overline s)$, but would be rerouted if the solution network contained an $r$-path better than $P^r(\overline s)$. Note that, in the left-hand-side of \eqref{sp:1_a} (or  \eqref{sp:1_b}), the demand  $w^r$ of a given commodity $r\in R_{ij}(\overline s)$ is multiplied by $\sum_{k\in Z^{r}(\overline s)} z_k$, which can be greater than one. This means that, in some cases, the \textit{compensation} can be unnecessarily large. This can be avoided by subtracting the term $y_{km}$ for every pair $k\ne m$, when both $k, m\in Z^{r}(\overline s)$, $k\ne m$. In particular:
	\begin{propo}\label{propo:f}
		Let $\overline s\in {\mathcal S}$ be a given solution network and $f(\overline s)\in\mathcal L(\overline s)$ an $R$-feasible flow in $G(\overline s)$. Then, for all $ij\in E$, the following inequalities are valid for the $H$-median: 
		\begin{subequations}
			\begin{align}
				&  t_{ij}+\sum_{r\in{R}_{ij}(\overline s)}\!\!\!\! w^r\Bigg[\sum_{k\in Z^{r}(\overline s)}\!\!\! z_k+\sum_{\substack{km\in Y^{r}(\overline s):\\ k,m\notin Z^{r}(\overline s)}}\!\!\!\! y_{km} -\sum_{\substack{km\in E:\\ k,m\in Z^{r}(\overline s)}}\!\!\!\!y_{km}\Bigg]\geq t_{ij}(\overline s)\, y_{ij} &&\label{sp:2_a}\\
				&  t_{ji}+\sum_{r\in{R}_{ji}(\overline s)}\!\!\!\! w^r\Bigg[\sum_{k\in Z^{r}(\overline s)}\!\!\! z_k+\sum_{\substack{km\in Y^{r}(\overline s):\\ k,m\notin Z^{r}(\overline s)}}\!\!\!\! y_{km}  -\sum_{\substack{km\in E:\\ k,m\in Z^{r}(\overline s)}}\!\!\!\!y_{km}\Bigg]\geq t_{ji}(\overline s)\, y_{ij}.&& \label{sp:2_b}
			\end{align}
		\end{subequations}
 	\end{propo}
    
	We can derive inequalities  analogous to \eqref{sp:1_a}-\eqref{sp:1_b} for the values of the flows though access and distribution arcs of a given solution network $G(\overline s)$. Now, $h^1_{oi}(\overline s)=\sum_{r\in {R}^1_{oi}(\overline s)}{w^r}$ is the total flow circulating through access arc $(o, i)$ for the $R$-feasible flow $f(\overline s)\in\mathcal L(\overline s)$. If the solution network changed, then the $s$-$r$-path of any commodity $r\in {R}^1_{oi}(\overline s)$ would change not only when $i$ is deactivated as a hub or when a hub \textit{better} than $i$ is activated, but also when hub edge $ij$  used in $P^r(\overline s)$  is de-activated, even if the set of hubs remained unchanged. The following example illustrates this:
	
	\begin{example}\label{example2}
		Consider a four node graph with unit pairwise demands $w^r=1$, for all $r\in R$ and the symmetric unit cost matrix depicted in Figure \ref{ex:2}.(a) with $\alpha=0.2$. Consider also the solution network $\bar s$ with $\bar z_1=\bar z_2=1$, $\bar y_{12}=1$. As can be seen in  Figure \ref{ex:2}.(b), the flow through access arc $(3, 2)$ is $h^1_{32}(\bar s)=3$. However, if edge $12$ were deactivated, the flow through access arc   $(3, 2)$ will no longer be 3, even if node 2 remained activated as a hub. That is,  for the solution network $\bar s'$ with $\bar z'_2=1$, $\bar y'_{ij}=0$ for all $ij\in E$, the flow through $(3, 2)$ would be $h^1_{32}(\bar s')=w_{32}+w_{34}=2$. It is only when $\bar y_{12}=1$ that we can impose $h^1_{32}\ge 3$ unless a better alternative is activated for the involved commodities. In particular, as we will see next, the optimality inequality that must hold is:
		$$ h^1_{32} + w_{31}y_{31}+w_{32}y_{32}+w_{34}(z_{3}+z_{4}-y_{34}) \geq w_{31}y_{12}+ \left(w_{32}+w_{34}\right)z_2.$$
	\end{example}
    
	For taking into account the circumstance illustrated in the above example, we now partition the commodities $r\in {R}^1_{oi}(\overline s)$ according to the following two cases:
	
	\begin{itemize}
		\item ${R}^{1,1}_{oi}(\overline s)=\{r\in {R}^1_{oi}(\overline s):\, P^r(\overline s)=o-i-d\}$. That is, ${R}^{1,1}_{oi}(\overline s)$ contains the indices of the commodities with origin at node $o$,  whose $s$-$r$-path in $G(\overline s)$ uses $i$ as the only intermediate hub.
		\item ${R}^{1,2}_{oi}(\overline s)=\{r\in {R}^1_{oi}(\overline s): P^r(\overline s)=o-i-j^r(i)-d \}$.  That is, ${R}^{1,2}_{oi}(s)$ contains the indices of the commodities with origin at node $o$, whose $s$-$r$-path in $G(\overline s)$ uses two intermediate hubs, being $i$ the first one. The index of the second intermediate hub is denoted by $j^r(i)$. As we have shown in Example \ref{example2}, if $i\,j^r(i)$ is not activated as an interhub edge, then commodity $r$ could be routed through a different access arc, even if no element outside $G(\overline s)$ were activated.
	\end{itemize}
	\noindent Thus ${R}^1_{oi}(\overline s)=R^{1,1}_{oi}(\overline s)\cup {R}^{1,2}_{oi}(\overline s)$, and the total flow through $(o, i)$ will be  $\sum_{r\in {R}^{1,1}_{oi}(\overline s)}w^r z_{i}+\sum_{r\in {R}^{1,2}_{oi}(\overline s)}w^r y_{i,j^r(i)}$, unless, for some $r\in {R}^1_{ij}(\overline s)$, a hub is activated at some node of $Z^{r}(\overline s)$ or an interhub edge is activated for some edge of $Y^{r}(\overline s)$, as  stated in the following result:
	
	\begin{propo}\label{propo:access}
		Let $\overline s\in {\mathcal S}$ be a given solution network and $f(\overline s)\in\mathcal L(\overline s)$ an $R$-feasible flow in $G(\overline s)$. Then, for all $(o,i)\in A$, the following inequalities are valid for the $H$-median:
        \begin{align}
             h^1_{oi}&+\sum_{r\in{R}^1_{oi}(\overline s)}\!\!\!\! w^r\left[\sum_{k\in Z^{r}(\overline s)}\!\!\!\!z_k+\sum_{\substack{km\in Y^{r}(\overline s):\\ k\ne i,\, k,m\notin Z^{r}(\overline s)}}\!\!\!\!\!\!\!\!y_{km} \right] \notag\\
            & \geq \sum_{r\in {R}^{1,1}_{oi}(\overline s)}\!\!\!\!w^r z_{i}+\sum_{r\in {R}^{1,2}_{oi}(\overline s)}\!\!\!\!w^r y_{i,j^r(i)}, \label{sp:2_access_2-1}\\
            h^1_{oi}&+\sum_{r\in{R}^1_{oi}(\overline s)}\!\!\!\! w^r\Bigg[\sum_{k\in Z^{r}(\overline s)}\!\!\!\! z_k +\sum_{\substack{km\in Y^{r}(\overline s):\\ k\ne i,\, k,m\notin Z^{r}(\overline s)}} \!\!\!\!\!\!\!\!\!\!y_{km} -\sum_{\substack{km\in E:\\ k,m\in Z^{r}(\overline s)}}\!\!\!\!\!\!y_{km}\Bigg]\notag\\
            &\geq \sum_{r\in {R}^{1,1}_{oi}(\overline s)}\!\!\!\!\!w^r z_{i}+\sum_{r\in {R}^{1,2}_{oi}(\overline s)}\!\!\!\!w^r y_{i,j^r(i)}.&&\label{sp:2_access_2_2}
        \end{align}
	\end{propo}
    
    \begin{proof}
		The proof is given in Appendix \ref{appendix:proofs}.
	\end{proof}
	
	We finally derive valid inequalities for optimal flows through potential distribution arcs. The notation that we use now is the following:
	\begin{itemize}
		\item ${R}^{2,1}_{jd}(s)=\{r\in {R}^2_{jd}(\overline s): P^r(\overline s)=o-j-d\}$ contains the indices of the commodities with destination at node $d$ whose $s$-$r$-path in $G(\overline s)$ uses $j$ as the only intermediate hub.
		\item ${R}^{2,2}_{jd}(s)=\{r\in {R}^2_{jd}(\overline s): P^r(\overline s)=o-i^r(j)-j-d\}$ contains the indices of the commodities with destination at node $d$ whose $s$-$r$-path in $G(\overline s)$ uses two intermediate hubs, being $j$ the second one. The index of the first intermediate hub is denoted by $i^r(j)$. Similarly to the previous case, if $i^r(j)j$ is not an activated interhub edge, then commodity $r$ could be routed through a different distribution arc, even if no element outside $G(\overline s)$ were activated.
	\end{itemize}

	\begin{propo}\label{propo:distrib}
		Let $\overline s\in {\mathcal S}$ be a given solution network  and $f(\overline s)\in\mathcal L(\overline s)$ an $R$-feasible flow in $G(\overline s)$. Then,  for all $(j,d)\in A$, the following inequalities are valid for the $H$-median:
        
		\begin{align}
			h^2_{jd}&+\sum_{r\in{R}^2_{jd}(\overline s)}\!\!\!\! w^r\left[\sum_{k\in Z^{r}(\overline s)}\!\!\!\! z_k+\sum_{\substack{km\in Y^{r}(\overline s):\\ m\ne j,\, k,m\notin Z^{r}(\overline s)}}\!\!\!\!\!\!\!\!\!\! y_{km} \right]\notag\\
			&\geq \sum_{r\in {R}^{2,1}_{jd}(\overline s)}\!\!\!\!\!w^r z_{j}+\sum_{r\in {R}^{2,2}_{jd}(\overline s)}\!\!\!\!\!w^r y_{i^r(j),j} \label{sp:2_distrib_1}\\
			h^2_{jd}&+\sum_{r\in{R}^2_{jd}(\overline s)}\!\!\!\! w^r\Bigg[\sum_{k\in Z^{r}(\overline s)}\!\!\!\! z_k+ \sum_{\substack{km\in Y^{r}(\overline s):\\ m\ne j,\, k,m\notin Z^{r}(\overline s)}}\!\!\!\!\!\!\!\!\!\!\!\! y_{km}-\sum_{\substack{km\in E:\\ k,m\in Z^{r}(\overline s)}}\!\!\!\!\!\!y_{km} \Bigg]\notag\\
			&\geq \sum_{r\in {R}^{2,1}_{jd}(\overline s)}\!\!\!\!w^r z_{j}+\sum_{r\in {R}^{2,2}_{jd}(\overline s)}\!\!\!\!w^r y_{i^r(j),j} \label{sp:2_distrib_2}
		\end{align}
	\end{propo}
	\noindent The proof is similar to that of Proposition \ref{propo:access}. Details are omitted.
	
	\begin{teorema}\label{teo}
		A valid formulation for the 
        $H$-median is the following:
        \begin{subequations}
			\begin{align}
				&  (H-\text{med}) 
                \quad \min\, v_{act}(s)\,+v_{rout}(f) &&\notag\\
				&   t_{ij}+\sum_{r\in{R}_{ij}(\overline s)}\!\!\!\! w^r\Bigg[\sum_{k\in Z^{r}(\overline s)}\!\!\! z_k+\sum_{\substack{km\in Y^{r}(\overline s):\\ k,m\notin Z^{r}(\overline s)}}\!\!\!\! y_{km} -\sum_{\substack{km\in E:\\ k,m\in Z^{r}(\overline s)}}\!\!\!\!y_{km}\Bigg]\geq t_{ij}(\overline s)\, y_{ij} &&  \overline s\in {\mathcal S},\, ij\in E\label{eq1_1}\\
				&   t_{ji}+\sum_{r\in{R}_{ji}(\overline s)}\!\!\!\! w^r\Bigg[\sum_{k\in Z^{r}(\overline s)}\!\!\! z_k+\sum_{\substack{km\in Y^{r}(\overline s):\\ k,m\notin Z^{r}(\overline s)}}\!\!\!\! y_{km}  -\sum_{\substack{km\in E:\\ k,m\in Z^{r}(\overline s)}}\!\!\!\!y_{km}\Bigg]\geq t_{ji}(\overline s)\, y_{ij} && \overline s\in {\mathcal S},\, ij\in E\label{eq1_2}\\
				& h^1_{oi}+\sum_{r\in{R}^1_{oi}(\overline s)}\!\!\!\! w^r\Bigg[\sum_{k\in Z^{r}(\overline s)}\!\!\!\! z_k +\sum_{\substack{km\in Y^{r}(\overline s):\\ k\ne i,\, k,m\notin Z^{r}(\overline s)}} \!\!\!\!\!\!\!\!\!\!y_{km} -\sum_{\substack{km\in E:\\ k,m\in Z^{r}(\overline s)}}\!\!\!\!\!\!y_{km}\Bigg]\notag\\
				& \hspace{2.5cm}\geq \sum_{r\in {R}^{1,1}_{oi}(\overline s)}\!\!\!\!\!w^r z_{i}+\sum_{r\in {R}^{1,2}_{oi}(\overline s)}\!\!\!\!w^r y_{i,j^r(i)} &&  \overline s\in {\mathcal S},\, (o,i)\in A\label{eq2}\\
				& h^2_{jd}+\sum_{r\in{R}^2_{jd}(\overline s)}\!\!\!\! w^r\Bigg[\sum_{k\in Z^{r}(\overline s)}\!\!\!\! z_k+ \sum_{\substack{km\in Y^{r}(\overline s):\\ m\ne j,\, k,m\notin Z^{r}(\overline s)}}\!\!\!\!\!\!\!\!\!\!\!\! y_{km}-\sum_{\substack{km\in E:\\ k,m\in Z^{r}(\overline s)}}\!\!\!\!\!\!y_{km} \Bigg]\notag\\
                & \hspace{2.5cm}\geq \sum_{r\in {R}^{2,1}_{jd}(\overline s)}\!\!\!\!w^r z_{j}+\sum_{r\in {R}^{2,2}_{jd}(\overline s)}\!\!\!\!w^r y_{i^r(j),j}  &&  \overline s\in {\mathcal S},\, (j,d)\in A\label{eq3}\\
				& s=(z, y, x^1, x^2)\in {\mathcal S};\qquad f=(t, h^1,h^2)\in\mathcal F(s)\notag
			\end{align}
		\end{subequations}
	\end{teorema}
	
	\begin{proof}
		The proof is given in Appendix \ref{appendix:proofs}.
	\end{proof}

\section{Branch-and-solve for GHLP}\label{sec:B&S}
	In this section we detail the B\&S algorithm designed to optimally solve the GHLP. As explained, we model the GHLP in a mixture of two parts: a \textit{basic} MIP and a \textit{delayed} CLP. The \textit{basic} MIP is the relaxation of GHLP defined by the $R-GHLP$:
	$$	(R-GHLP) \qquad \min\,_{f \in\mathcal F(s), s\in \mathcal S\,} \quad\{ v_{act}(s)\,+v_{rout}(f) \}$$
	\noindent where $\mathcal S$ is the domain  of design solutions, described in Section  \ref{Domain:S}, and $\mathcal F(s)$ is the domain of reinforced aggregated flows for solution $s$, presented in Section \ref{Domain:F}. The \textit{delayed} part for a feasible design solution $s\in\mathcal S$ is denoted by $\mathcal L(s)$, and contains $R$-feasible flows for $s$ (see Section \ref{R-feas0}).
	
	The B\&S  algorithm explores an enumeration tree, branching on the design  variables $s$. At each node, it solves the LP relaxation of R-GHLP. The R-GHLP formulation is dynamically extended in a B\&C fashion by incorporating violated aggregated demand constraints \eqref{Demand-0}. Since the size of this family of constraints is exponential in the number of nodes of the input graph, they are separated for fractional solutions at the root node and as lazy constraints at the nodes of the enumeration tree where the LP relaxation of R-GHLP produces a solution $(\bar s, \bar f)$, with $\bar s$ binary. In addition, at the root node, the LP relaxation of R-GHLP is further reinforced by separating the connectivity constraints \eqref{connect-backbone-0}, also of exponential size in the number of nodes of the input graph.
	Initially, only constraints of both types associated with singletons  are included in the formulation. Each singleton $S=\{i\}$, with $i\in V$, defines one  constraint \eqref{Demand-0} and $|V|-1$ constraints \eqref{connect-backbone-0}, one for each $j\in V\setminus \{i\}$. The separation procedures for these two  constraint families are explained in Sections \ref{sec:separa} and \ref{sec:separa-connect}, respectively.
	
	Furthermore, at selected nodes of the enumeration tree, the GHLP is solved exactly, by integrating an \textit{ad-hoc} algorithm for the \textit{delayed} part (see Section \ref{R-feas}).  Since $\mathcal F(s)\subseteq \mathcal L(s)$, at the nodes where the LP relaxation of R-GHLP produces a solution $(\bar s, \bar f)$, with $\bar s$  binary, it is not guaranteed that the flow $\bar f$ is $R$-feasible. Hence, in order to guarantee a proper partition of the  GHLP solution space, throughout the tree exploration, at such nodes we explicitly compute an $R$-feasible flow $f(\bar s)\in\mathcal L(\bar s)$. This is done by finding, for each commodity $r\in R$, an $s$-$r$-path in $G(\bar s)$,  $P^r(\bar s)$, using a shortest path algorithm (see Section \ref{shortest}). This ensures that the flow $f(\bar s)$ dictated by the $s$-$r$-paths $P^r(\bar s)$, $r\in R$, is $R$-feasible for $\bar s$. Therefore,  $(\bar s, f(\bar s))$ is an optimal GHLP solution for the current node, which needs not be further explored.  If this solution improves the current incumbent, it is updated accordingly. Then, a nogood cut is added to force the solver eliminate the current node.
	
	A pseudo-code of the solution method is outlined in Algorithm \ref{algo:sol}. This basic version of the B\&S solution algorithm will be referred to as BS. It will benchmarked against published results from the literature and against two other B\&S variants: one in which feasibility Benders cuts \eqref{Feas:D2} are separated (referred to as BSF), and another one in which optimality cuts \eqref{eq1_1}-\eqref{eq3} are also separated (referred to as BSO). In both cases, the respective separation procedures will be applied at nodes with binary design solutions $\overline s$  and the identified violated inequalities incorporated to the R-GHLP formulation. Details on the separation of feasibility and optimality cuts are provided in Sections \ref{sec:separa-Benders} and \ref{sec:separa-opt}, respectively.
	\begin{algorithm}
        \small
		\caption{Exploration of tree node $\ell$}\label{algo:sol}
        \hspace*{\algorithmicindent} \textbf{Input:} Incumbent solution $(s^*, f^*)$; Incumbent value $v_{act}(s^*)\,+v_{rout}(f^*)$\vspace{1.pt}
		{\setlength{\baselineskip}{10pt}
			\begin{algorithmic}[1]
				\If{($\ell$ is the root node)}\vspace{1.pt}
				\State $repeat\gets true$\vspace{1.pt}
				\While{(repeat)}\vspace{1.pt}
				\State $repeat\gets false$\vspace{1.pt}
				\State $(\bar s, \bar f)\gets \text{Optimal solution to LP relaxation of } R-GHLP^{\ell}$\vspace{1.pt}
				\State Apply procedure of Section \ref{sec:separa-connect} to find a connectivity constraint \eqref{connect-backbone-0} violated by $\bar s$\vspace{1.pt}
				\State Apply procedure of \ref{sec:separa} to find an aggregated demand constraint \eqref{Demand-0} violated by $\bar f$\vspace{1.pt}
				\If{(violated connectivity constraint \eqref{connect-backbone-0} or violated aggregated demand constraint \eqref{Demand-0} found)}\vspace{1.pt}
				\State $repeat\gets true$
				\State Incorporate identified viol\vspace{1.pt}ated constraint \eqref{connect-backbone-0}\vspace{1.pt}
				\State Incorporate identified violated constraint \eqref{Demand-0}\vspace{1.pt}
				\EndIf
				\EndWhile
				\ElsIf{($\bar s$ is integer)}
				\State Apply procedure of \ref{sec:separa} to find an aggregated demand constraint \eqref{Demand-0} violated by $\bar f$\vspace{1.pt}
				\If{(violated aggregated demand constraint \eqref{Demand-0} found)}\vspace{1.pt}
				\State Incorporate identified violated constraint \eqref{Demand-0}\vspace{1.pt}
				\EndIf
				\State Apply Algorithm \ref{R-feas:alg} to find an $R$-feasible flow $f(\bar s)\in \mathcal L(\bar s)$\vspace{1.pt}
				\If{($v_{act}(\bar s)\,+v_{rout}(\bar f)<v^*$)}\vspace{1.pt}
				\State $(s^*, f^*)\gets (\bar s, f(\bar s))$\vspace{1.pt}
				\State $v^*\gets v_{act}(\bar s)\,+v_{rout}(\bar f)$\vspace{1.pt}
				\EndIf
				\State Add nogood cut $\sum_{ij\in E: \bar y_{ij}=0}y_{ij}+\sum_{ij\in E: \bar y_{ij}=1}\left(1-y_{ij}\right)\geq 1$\vspace{1.pt}
				\EndIf
		\end{algorithmic}}
		\hspace*{\algorithmicindent} \textbf{Output:} Incumbent solution $(s^*, f^*)$; Incumbent value $v_{act}(s^*)\,+v_{rout}(f^*)$
	\end{algorithm}
    
	\subsection{Finding feasible $R$-flows $f(\bar s)\in \mathcal L(\bar s)$: $s$-$r$-paths in $G(\bar s)$}\label{R-feas}\label{shortest}
		For finding $s$-$r$-paths in the solution network $G(\bar s)$ associated with a given solution $\bar s\in\mathcal S$, we simply solve an all-pairs shortest path problem on $G(\bar s)$ relative to the routing costs for the different types of arcs induced by $\bar s$, considering each node as the source node once. For this, Dijkstra's algorithm \citep{Dijkstra} is applied for each source node. Then, an $R$-feasible flow $f(\bar s)\in \mathcal L(\bar s)$ is obtained by computing for each arc of $G(\bar s)$ the sum of the demands of all the commodities whose $s$-$r$-path traverses that arc (see Algorithm \ref{R-feas:alg} in the Appendix \ref{shortest:app}).
	
		\subsection{Separation of inequalities}
			Below we describe the procedures we apply for the separation of the different families of inequalities that we consider in our solution algorithms.	
			\subsubsection{Separation of aggregated demand constraints \eqref{Demand-0}}\label{sec:separa}
				The constraint \eqref{Demand-0} associated with a given node set $S\subset V$ can be rewritten as $(t+h^1+h^2)(\delta^+(S))-W(\delta^+(S))\geq 0$, i.e.  $Q(\delta^+(S))\geq 0$ where, $Q_{ij}=({t}_{ij}+h^1_{ij}+h^2_{ij})-{w}_{ij}$ for each $(i,\, j)\in A$. For a given flow $\overline f= (\overline t,\, \overline h^1,\, \overline h^2)$ satisfying Constraints \eqref{leaves_i_H-0}-\eqref{Flow_Balance-0}, \eqref{H_X2_1-0}-\eqref{domain_F-H-0-0}, the sign of the coefficients $\overline{Q}_{ij}=({\bar t}_{ij}+\bar h^1_{ij}+\bar h^2_{ij})-{w}_{ij}$, with $(i,\, j)\in A$, can be both positive and negative. The separation of constraints \eqref{Demand-0} is thus to find a node set $S\subset V$ such that $\overline Q(\delta^+(S)) < 0$ or to prove that such a set does not exist.
				
				\paragraph{Exact separation} An exact solution to the above separation problem can be obtained by identifying a di-cut $S\subset V$ of minimum value, relative to the capacities vector $\overline{Q}$ (alternatively, of maximum value, relative to the capacities $-\overline{Q}$). Unlike the min-cut problem with non-negative capacities, which can be solved in polynomial time, min-cut (and max-cut) on a graph with arbitrary capacities is NP-hard. While formulations for the undirected variant of max-cut have been studied \citep[see, e.g. ][]{BGJR88,JuengerMallach}, we are not aware of any one for the directed variant, which is the one that arises here. A formulation is given in \cite{Ageev2001} that produces a bi-partition of the node set of a directed network for a special case of max-cut, which however does not ensure that the bi-partition is of maximum capacity, as it does not guarantee that all the arcs in the cutset are activated. In the EC we provide a formulation, based on that of \cite{Ageev2001}, which produces a cutset of minimum value relative to arc capacities $\overline Q$.
				
				\paragraph{Heuristic separation} While the above formulation allows to solve the separation problem exactly, it uses binary variables associated with both nodes 
                and arcs. Thus, solving it to optimality can be computationally rather time consuming, as was confirmed empirically in preliminary  testing. Therefore, we have considered two alternatives for the heuristic separation for inequalities \eqref{Demand-0}.
				\begin{itemize}
					\item The first one is to find a cutset of minimum capacity relative to $\widehat Q_{ij}=\max\{\overline Q_{ij},\, 0\}$, which can be obtained from the tree of min-cuts. Since $\widehat Q_{ij}\geq 0$ for all $(i, j)\in A$, such a tree can be found in polynomial time; for instance, with the algorithm of \cite{Gusfield}, which is is $\mathcal O (|V|^3)$.
					\item Repeatedly applying the above heuristic separation can be still too time consuming. A faster alternative is to establish a threshold value $\varepsilon$, and to identify the connected components in the subgraph $G_{\varepsilon}(\bar s)$, induced by the edges such that $\bar x^1_{ij} + \bar x^2_{ij}+\bar y_{ij}+\geq \varepsilon$, $ij\in E$. If there are several components, we check, for each of them, whether the corresponding  aggregated demand constraint \eqref{Demand-0} is violated by $\overline f$. Note that for binary design solutions, this alternative may  identify violated inequalities only if the solution network induces more than one connected component.
				\end{itemize}
				For the SA policy, adapting the separations above to the reinforced aggregated demand constraints \eqref{Ag:dem_reinf} is quite involved. For this reason,  only constraints \eqref{Demand-0} are separated also for the SA policy, although if a violated constraint is found, then, the inequality that is added is \eqref{Ag:dem_reinf}.

			\subsubsection{Separation of connectivity constraints \eqref{connect-backbone-0}}\label{sec:separa-connect}
                Constraints \eqref{connect-backbone-0} can be separated exactly by finding the tree of min-cuts in the subgraph induced by the edges such that $\bar y_{ij}>0$, using the values  $\bar y_{ij}$ as edge capacities. Again we apply the algorithm of \cite{Gusfield} to find such a tree.
                
                Alternatively, we may apply a heuristic separation, which is to consider the connected components in the graph $G(\bar y;\varepsilon)$, induced by the edges such that $\bar y_{ij}\geq \varepsilon$, and, if there are several components, to check, for each of them, whether the corresponding  constraint \eqref{connect-backbone-0} is violated by $\overline y$.
	
			\subsubsection{Separation of feasibility Benders cuts \eqref{Feas:D2}} \label{sec:separa-Benders}
				A feasibility Benders cut \eqref{Feas:D2} violated by a given solution $(\bar s,\, \bar f)$, with $\bar s\in \mathcal S$, $\bar f\in\mathcal F(\bar  s)$ can be found, if it exists, by solving the auxiliary problem  $D-{Aux}_{(\bar s, \bar f)}$. In case it is unbounded, the feasibility inequality \eqref{Feas:D2} associated with the extreme ray that allows to identify unboundedness is violated by $(\bar s, \,\bar f)$. In the BSF variant,  constraints \eqref{Feas:D2} are separated once for the final (fractional) solution of the root node, and at every node of the enumeration tree where the LP relaxation of R-GHLP produces a solution $(\bar s, \bar f)$, with $\bar s$  binary.
	
			\subsubsection{Separation of optimality inequalities \eqref{eq1_1}-\eqref{eq3}}\label{sec:separa-opt}
				For a given feasible solution $(\bar s, \,\bar f)$, with $\bar s\in\mathcal S$, $\bar f\in\mathcal L(\bar s)$, the separation of constraints \eqref{eq1_1}-\eqref{eq3} can be done by inspection, provided that the involved index sets are available.
				While the index sets $R_{ij}(\bar s)$, $R^1_{oi}(\bar s)$, and $R^2_{jd}(\bar s)$ can be identified by exploring each link of the graph $G(\bar s)$ (usually quite sparse) just once, computing the index sets $Z^r(\bar s)$ and $Y^r(\bar s)$ is significantly more time consuming as, essentially, it requires exploring the entire input graph for each commodity. For this reason, in the initialization of the solution algorithm, for each commodity $r\in R$, we enumerate all possible $r$-paths with at most one interhub arc and sort them in increasing order of their routing costs. This sorting produces two initial index lists for each commodity, $\bar Z^r$ and $\bar Y^r$, corresponding to $r$-paths with a single hub node and with two hub nodes,  respectively. Then, in the solution algorithm, for separating \eqref{eq1_1}-\eqref{eq3} for a given feasible solution $(\bar s, \,\bar f)$, the index sets $Z^r(\bar s)$ and $Y^r(\bar s)$ are determined, for all $r\in R$, by sequentially exploring  both lists, and identifying all  indices corresponding to routing costs smaller than $C^r(\bar s)$ and $\min\{C^r(\bar s), C^r_{kk}, C^r_{mm}\}$, respectively.
	
				Inequalities \eqref{eq1_1}-\eqref{eq3} are separated at every node of the enumeration tree where the LP relaxation of R-GHLP produces a solution $(\bar s, \bar f)$, with $\bar s$  binary.

\section{Computational experiments}\label{sec:compu}
    In this section we report results from the computational experiments we have run. All the experiments have been performed on a PC equipped with a Ryzen 7 5700G CPU and 32Gb of RAM. The formulations have been implemented in Python 3.11 and solved with Gurobi 10.0.1. To provide reproducible results, we set Gurobi's \textit{Threads} parameter to 1 and turned off the \textit{Presolve} option.

    For the computational experiments we have used benchmark instances from two very well-known datasets from the hub location literature:
    \begin{itemize}
        \item The Civil Aeronautics Board (CAB) dataset, contains data from 100 cities in the United States of America \citep[see][]{Okelly87}. It provides symmetric commodities demands, $w^{r}$ and unit routing costs $c_{ij}$. From this dataset we have generated test instances of  sizes $n \in \{25, 50, 75, 100\}$. As usual, for each value $n$, the instance is generated with the data corresponding to the first $n$ entries.
        \item The Australian Post (AP) dataset, first published by \cite{Ernst96}, contains data from 200 nodes with non-integer asymmetric demands for the commodities $w^{r}$ and unit routing costs $c_{ij}$. In these instances self flows $w_{ii}\ne 0$, $i\in V$. Still, we ignored them as we assume that  $o(r)\ne d(r)$, $r\in R$. From this dataset we have generated test instances for values of $n \in \{25, 50, 75, 100, 120, 140, 160, 180, 200\}$. Again, for each value $n$, the instance is generated with the data corresponding to the first $n$ entries
    \end{itemize}

    We have used interhub discount factors $\alpha\in\{0.2, 0.5, 0.8\}$ and, unless stated otherwise, access and distribution weights $\gamma=\theta=1$.

    Setup costs for hub activation and interhub link activation have been calculated as in \cite{Wandelt22}. In particular, the authors follow the methodology by \cite{Ebery00} for setup costs for hub activation, whereas the activation cost of edge $km\in E$ is computed as
    $$G_{km}=\frac{\sum_{r\in R} w^rc_{km}}{n^2}.$$

    The following variants of our B\&S algorithm have been compared:
    \begin{itemize}
        \item BS: This is the basic B\&S algorithm without any further feasibility or optimality cuts.
        \item BSF: This is the basic B\&S algorithm where feasibility Benders cuts \eqref{Feas:D2} are generated and incorporated to the formulation of R-GHLP as explained in Section \ref{sec:B&S}.
        \item BSO:  This is the basic B\&S algorithm where optimality cuts \eqref{eq1_1}-\eqref{eq3} are generated and incorporated to the formulation of R-GHLP as explained in Section \ref{sec:B&S}.
    \end{itemize}

    Since optimality cuts \eqref{eq1_1}-\eqref{eq3} are only valid for problems where optimal routing paths contain at most one interhub arc, the BSO variant has been tested for the $H$-median only.

    Our formulations and solution algorithms have been tested  for the different models, for both the SA and MA policies. The obtained results have been compared against: $(a)$ our own implementations of the 3- and 4-index formulations, denoted as 3I and 4I, respectively, up to the dimensions when this was possible ($n=50$ for 3I and $n=25$ for 4I); and $(b)$ the best  results obtained with BBC methods published in the literature. These results have been taken from \cite{Wandelt22} and from \cite{Espejoetal23}. The former present an extensive comparative analysis of BBC methods developed by different authors for solving different HLPs, and thus provide a unified reference for comparison. When using this reference, for each tested model, we have used as benchmark method the one  \textit{recommended}  by the authors  which is the one performing best among the ones they compare.  More recently, \cite{Espejoetal23} have developed a specialized BBC algorithm for the SA $H$-median that we also use for comparison. In particular, the tested GHLP models and benchmark results we have used are the following:
    \begin{itemize}
        \item $H$-median. Our results for these models  are compared against our 3I and 4I implementations, and the results reported in \cite{Wandelt22} corresponding to \cite{Ghaffarinasab2018} with  CAB and AP instances with up to $n=50$ nodes for SA, and to \cite{Camargo2008} with  CAB and AP instances with up to $n=75$ nodes for MA.\\
        For the $H$-median with SA policy, our results are also compared against those of \cite{Espejoetal23} whose BBC algorithm solves this variant of the problem for AP instances with up to 200 nodes. For this comparison we set the cost weights to those used in the referenced work, i.e.,  $\gamma=3$, $\theta=2$, and  $\alpha=0.75$ for access, distribution and interhub links, respectively.
        \item $G$-median: For these models, our results  are compared against those reported by \cite{Wandelt22}, with CAB and AP instances with $n=25$,  corresponding to their own BBC implementation for SA, and against their adaptation to the $GHLP$ of the BBC solution algorithm of \cite{deCamargo2017} for GHLP's with hop constraints for MA.
    \end{itemize}

    We do not compare our results with those obtained with cutting plane methods based on Benders decomposition (also referred to as \textit{row generation methods}) \citep[see, e.g.][]{Contrerasetal11,deCamargo2009}, since the performance of these methods  largely relies on the quality of initial solutions produced by heuristics and on refinements, tailored for each specific model and allocation policy. We highlight  the generality of  both $(i)$ the GHLP formulation that we use, {which applies to both allocation policies and with minor modifications can be adapted to a wide range of HLP models}, and $(ii)$ the B\&S solution algorithm, which includes no specialized heuristic (a rudimentary rounding heuristic is applied just once) or \textit{ad hoc} refinements of its basic ingredients.

    Several authors have observed \citep[see, e.g.,][]{CampbellOkelly25} that the hub arcs in optimal solutions to traditional HLPs, often carry much smaller flows than some of the access arcs do. This is a clear limitation of these models, which undermines the basic premise for economies of scale. As mentioned, our formulation already addresses this concern, by adapting the values of the coefficients in constraints \eqref{H_X2_1-0}-\eqref{f_Y-0}. In order to empirically illustrate this capability, in our last series of computational tests, we apply BS to an extension of the $GHLP$ where we impose lower bounds on the flows that must circulate through activated interhub links, which we refer to as  $GHLP$ with flow bounds ($G-FB$). In particular, the left-hand-side of Constraints \eqref{f_Y-0} become $\ell_{km}\,\overline y_{km}\leq t_{km}+t_{mk}$, where $\ell_{km}$ is the parameter that determines the minimum flow that must  circulate through $km$ if is is activated as an interhub edge.

    Table \ref{tab:compu} summarizes the tested models and solution algorithms. It also gives references to the results from the literature against which our results have been compared, as well as the tables where these comparative results are presented:

    \begin{table}[H]
		\begin{tabular}{c|c|c|c|c|c|c|c|c|c}
            \multicolumn{3}{c|}{}                   & \multicolumn{1}{c|}{3I} & \multicolumn{1}{c|}{4I} & \multicolumn{1}{c|}{BS}  & \multicolumn{1}{c|}{BSF} & \multicolumn{1}{c}{BSO}& \multicolumn{1}{c}{BBC}& \multicolumn{1}{c}{Tables}\\     \hline
           \multicolumn{2}{c}{\multirow{2}{*}{$H$-median}}           & SA   & \checkmark &\checkmark&           \checkmark              &           \checkmark              &      \checkmark       &  \checkmark [1] taken from  [4], and [5]  & \multirow{2}{*}{\ref{comp-BS-HG}, \ref{comp_best:GH}}\\
    		 \multicolumn{2}{c}{}	                                  & MA   & \checkmark &\checkmark&           \checkmark             &           \checkmark              &      \checkmark      &  \checkmark [2] taken from  [4] & \\ \hline
    			\multicolumn{2}{c}{\multirow{2}{*}{$G$-median}}       & SA   &&&           \checkmark&           \checkmark             &      & \checkmark [4] & \multirow{2}{*}{\ref{comp-BS-HG}, \ref{comp_best:GH}}\\
    		 \multicolumn{2}{c}{}	                                  & MA   &&&           \checkmark&           \checkmark             &      & \checkmark Adaptation of  [3] taken from [4]\\ \hline
          \multicolumn{2}{c}{\multirow{2}{*}{$G-FB$}} & SA & &&   \checkmark& & && \multirow{2}{*}{\ref{tab:capacitated}}\\
           \multicolumn{2}{c}{} & MA & &&   \checkmark& & &
		\end{tabular}
        \caption{Summary of computational experiments\label{tab:compu}}
        {\footnotesize
        \text{[1]:} \cite{Ghaffarinasab2018};
        \text{[2]:} \cite{Camargo2008};
        \text{[3]:} \cite{deCamargo2017};
        \text{[4]:} \cite{Wandelt22};
        \text{[5]:} \cite{Espejoetal23}}
	\end{table}
    
    Since the computer where we have ran our experiments is different from those used in \cite{Wandelt22} and \cite{Espejoetal23}, it is difficult to establish a precise equivalence between all computing times. In order to make the comparison as fair as possible, we have checked the specifications given by the Standard Performance Evaluation Corporation (SPEC). The referenced computers that resemble the most to ours and that of \cite{Wandelt22} and \cite{Espejoetal23} are an AMD\,R7\,5800X, \footnote{\url{https://www.spec.org/cpu2017/results/res2022q3/cpu2017-20220718-32225.html}} an INTEL\, E5-2650v4 \footnote{\url{https://www.spec.org/cpu2017/results/res2018q1/cpu2017-20180216-03626.html}}, and an INTEL XEON\, W-2295\footnote{ \url{https://www.spec.org/cpu2017/results/res2019q4/cpu2017-20191015-19201.html}}, respectively. The floating point speed scores they receive are 47.5, 68.4 and 73.2 respectively. Thus, for a fair comparison, in all our tables the computing  times taken from \cite{Wandelt22} are scaled by a factor $\frac{68.4}{47.5}\equiv 1.4$ and those from \cite{Espejoetal23} by a factor $\frac{73.2}{47.5}\equiv 1.5$. 

    \subsection{Implementation details}\label{imp:details}
        After some preliminary testing, in the final version of our solution algorithms we have used the choices detailed next for procedures and parameters.
        \begin{itemize}
            \item The heuristic option of the solver is turned off. Instead, we apply the following simple rounding heuristic immediately before leaving the root node to obtain a first incumbent. Let $(\bar s, \bar f)$ be the solution to the final LP relaxation of R-GHLP at the root node. We define the solution $(\hat s, \hat f)$, with $\hat f=f(\hat s)$, where the components of $\hat s$ take binary values: 1 if the corresponding component of $\bar s$ is greater than or equal to the  threshold $\varepsilon=0.6$, and 0 otherwise. In case no access/distribution arc is activated for some o/d, then the arc with the largest LP value is set to value 1.
            \item For the separation of inequalities of families of exponential size, among the alternatives discussed in Sections \ref{sec:separa} and \ref{sec:separa-connect} we apply the following. The aggregated demand constraints \eqref{Demand-0} are separated heuristically, by computing the tree of min cuts relative to the capacities $\widehat Q_{ij}=\max\{\overline Q_{ij},\, 0\}$. The tree is computed with the algorithm of \cite{Gusfield} although, instead of computing the entire tree, we stop the algorithm as soon as a violated constraint is found. The connectivity constraints \eqref{connect-backbone-0} are separated exactly by computing the tree of min cuts in the subgraph induced by the edges such that $\bar y_{ij}>0$. The algorithm of \cite{Gusfield} is used as just explained.
        \end{itemize}
        
        The remainder of this section is structured as follows. In Section \ref{sec:LP} we analyze the LP bounds produced by $R-GHLP$, by comparing them with those of the LP relaxations of 3I and 4I for each of the considered models. The comparison among BS, BSF, and BSO is discussed in Section \ref{comparison:BS-BSO-BSF}, whereas in Section \ref{results-p} we compare our results with those in the literature. Finally, Section \ref{results-large} summarizes our results for all the considered instances, including those with up to 200 nodes.
    
    \subsection{Comparison of LP bounds} \label{sec:LP}
        Figures 1-2 in the Electronic Companion of this paper (EC), compare the LP bounds produced by $R-GHLP$ with those of the LP relaxations of 3I and 4I for each of the considered models. As mentioned, due to memory limitations, the LP of 4I formulations could only be solved for instances with $n\leq 50$, whereas the LP  of 3I formulations could be solved for instances with $n\leq 100$. In our comparison, instances with $n\in\{25, 50\}$ are classified as ``small'' and instances with $n\in\{75, 100\}$ as ``medium''. For small-size instances, we can observe the (already known) advantage of 4I bounds over 3I ones. As expected, $R-GHLP$ produces LP bounds which never outperform those of the 4I formulation. Still, in general, the $R-GHLP$ bounds are very competitive with those provided by the 3I formulations. For the MA policy, the $R-GHLP$ bound often outperforms that of 3I across all models, particularly as the size of the instances increase. Still, for the SA policy, the 3I bound always outperforms that of $R-GHLP$, although the percentage deviation of the bound of $R-GHLP$ relative to that of the 3I formulation never exceeds {10}\%. We must however recall that $R-GHLP$ can handle instances with up to 200 nodes without memory limitations, which is not possible with the compared traditional formulations.
    
    \subsection{Comparison among BS, BSO and BSF} \label{comparison:BS-BSO-BSF}
        In order to compare the performance of the three alternative solution methods for the proposed GHLP models, we have used the CAB instances. The obtained results are summarized in Table \ref{comp-BS-HG}. Detailed results for the different allocation policies and data sets can be found in Tables 1-8 of the EC. Table \ref{comp-BS-HG} has three blocks of columns for each model, one for BS, one for BSF, and one for BSO, which has been used for $H$-median but not for $G$-median, since it is not valid for this model. Each block has one column for the number of explored nodes (\textit{\#Exp}) and another one for the computing time (\textit{cpu}). The third column of the blocks for BSF corresponds to the number of feasibility cuts \eqref{Feas:D2} generated. The blocks for BSO have one column for the number of optimality cuts of each type generated: \textit{Acc.} for the number of cuts \eqref{eq2}-\eqref{eq3} associated with access/distribution arcs and \textit{Int.} for the number of cuts  associated with interhub arcs \eqref{eq1_1}- \eqref{eq1_2}.
    
        As can be seen, all three algorithmic frameworks perform well, in general, in terms of both computing times and number of nodes explored in the search tree. Both BSF and BSO are effective in generating cuts of their respective families, although this comes at the expense of increasing not only the computing times but also the number of explored nodes.  In some cases, BSF fails in optimally solving the instances within the allowed time. This is clearly due to the computational burden involved in the solution of the auxiliary problem $D_{aux}$, which is required for the separation of feasibility cuts \eqref{Feas:D2}. Overall, BS clearly outperforms  BSF and BSO across all tested models and benchmark instances, both in terms of computing times and number of explored nodes. Because of that, this algorithm is used for all other computational experiments reported here.
    
    \subsection{Comparison with results from the literature} \label{results-p}
        Comparative results of BS against other methods in the literature for $H$-median and $G$-median instances are summarized in Table \ref{comp_best:GH}. The table has two main blocks of columns, one for SA and one for MA. In its turn, the SA block is divided in three smaller blocks: the first two ones are for comparison with the $H$-median results reported in \cite{Wandelt22} and \cite{Espejoetal23}, respectively, and the third one for comparison with the $G$-median results reported in \cite{Wandelt22}. The MA block is divided in two smaller blocks, both for comparison with the results reported in \cite{Wandelt22}, the first one for $H$-median, and the second one for $G$-median. Except for the block for comparison with \cite{Espejoetal23}, which is labeled as ``$H$-median [5] ($\alpha=0.75$)'', results are given for CAB and AP instances with up to 100 nodes for both SA and MA with values of $\alpha\in\{0.2, 0.5, 0.8\}$. On the contrary, the results of ``$H$-median [5] ($\alpha=0.75$)'' are given for instances with the same characteristics as in the referenced work:  larger AP instances with $n\in\{100, 125, 150, 150,200\}$, the SA policy, and  discount factor $\alpha=0.75$. In all blocks, 3I and 4I refer to results from 3- and 4-index formulations for the respective models, \textit{best} for results of the best approach for the corresponding model, and \textit{BS} for the results of our BS solution algorithm.
        
        As could be expected, BS outperforms both 4I and 3I for all tested benchmarks, not only because these formulations are only able to handle small size instances, but also because the computing times of BS are notably smaller.
        
        The results of Table \ref{comp_best:GH} show a clear advantage of BS  over the best results reported in \cite{Wandelt22} for the considered models.  On the contrary,  in general, the results of \cite{Espejoetal23} are remarkably better than our results. This can possibly explained by a combination of factors: on the one hand, similarly to our formulations the formulation of \cite{Espejoetal23} uses 2-index variables only. On the other hand, they apply a highly specialized algorithm, which, unfortunately cannot be extended to the MA policy or to the $G$-median, where routing paths can have more than one interhub arc. Still, note  that for the largest instance with 200 nodes BS clearly outperforms the method of \cite{Espejoetal23}.
        
        Overall, Table \ref{comp_best:GH} shows the effectiveness of BS. As we have shown, a generic combination of \textit{GHLP + B\&S} outperforms, with very few exceptions, existing highly specialized  BBC solution algorithms in the literature.
    
    \subsection{Analysis of results for larger instances} \label{results-large}
        Table \ref{tab:all}  summarizes the numerical results we have obtained in our computational experiments with BS with instances of up to 200 nodes. The table has two blocks of columns, one for each considered model, which, in turn, are divided in two smaller blocks of two columns each, one for SA and another one for MA. These two columns give information on the number of nodes explored in the search tree (\textit{\#Exp}) and the computing time (\textit{cpu}), respectively. The table has two blocks of rows, one for the CAB instances and another one for the AP instances. The first three columns of the table show the dataset, number of nodes ($n$), and value of the parameter $\alpha$, respectively.
        
        BS was able to solve to proven optimality all CAB and AP instances with up to $n=200$, with the exception of those with $n=200$ for $\alpha=0.8$. As can be seen, the computing times for these models are, in general, small for instances of this size, and the number of nodes explored in the search tree, tends to be remarkably small. Nevertheless, a clear \text{jump} can be observed in the number of explored nodes for the value $\alpha=0.8$. We think this can be explained because, for $\alpha=0.8$, the routing cost of a given arc changes very little when it is used as an interhub or an access/distribution arc.
        
        Altogether, we think that BS has a remarkable performance. Again we recall that use a formulation, which is the same \textit{sauf} minor modifications for all the considered models and allocation policies, and that the same solution algorithm, without enhancements or refinements, has been applied to all models.
    
    \subsection{Imposing minimum flows through interhub arcs}
        The main difference of BS for $G-FB$ than for the models previously analyzed is that, because of the lower bounds on arc flows, the subproblem at the nodes of the enumeration tree with integer design solutions can no longer  be solved with a shortest path algorithm.
        For our experiments, we solve an LP formulation, which is an adaptation of the feasibility problem $Aux_{(\bar s, \bar f)}$ to the current solution network, eliminating upper bounds on the flows, but adding the constraints $\ell_{km}\bar y_{km}\leq \sum_{i\in V}t^i_{km}$.
        
        To the best of our knowledge, there is no literature on computational tests on HLPs imposing lower bounds on  flows through activated interhub links. Thus, we have done some preliminary tests in order to find parameter values that produce solution networks that open more than just one or two hubs, and differ from the solutions obtained in our previous experiments. Based on the results of these tests,  we have set  $\ell_{km}=25\left(w_{km}+w_{mk}\right)$ for all $km\in E$. Moreover, in order to better visualize the impact of these bounds, we have reduced the setup costs for activating the hub nodes {by a factor of 10, i.e., $\bar f_k=\frac{f_k}{10}$}. All other parameters remained as before.
        
        Table \ref{tab:capacitated} summarizes the results we have obtained with small-size CAB and AP instances with up to 50 nodes. As can be seen, in general, CAB instances were easier to solve than those of AP. All instances were optimally solved although the computing times were notably larger than those without flow lower bounds. This is, of course, because the computing effort needed to solve the auxiliary problems is notably larger than that for finding shortest paths. This can be observed in the EC, by comparing columns $CPU_{SP} (\%)$ in Tables 3-4 and 7-8, with columns $CPU_{Aux} (\%)$ in Tables 13-14. We can also observe that, contrary to what usually happens with models without flow lower bounds, the difficulty for solving MA instances tends to be higher than that for SA instances.
        
        For further insight, the EC offers Figure 3 and Tables 9-12 illustrating the differences between the solutions of the $G$-median and the $G-FB$ for an specific CAB instance with 20 nodes and $\alpha=0.5$.

\section{Conclusions}\label{sec:conclu} 
    In this paper we have introduced a new formulation and solution framework for HLPs. The general principle behind the formulation is to reduce number of decision variables, which eliminates the need of projecting them out, so it is possible to design solution frameworks that do not rely on feasibility and optimality cuts generated during the solution process. The proposed formulation uses 2-index aggregated flow variables and includes a set of aggregated demand constraints, which are novel in hub location. It is very versatile as, with minor modifications, it applies to a large class of HLPs for both SA and MA policies. General purpose feasibility and optimality inequalities have also been developed. The formulation has been integrated within a B\&S solution framework, which leverages the nested structure of HLPs  by solving an auxiliary subproblem at selected nodes of the enumeration tree. The results from our computational experiments show the good performance of the proposal: instances with up to 200 nodes are solved to proven optimality with the basic framework for $H$-median and $G$-median. In particular, we have observed that, with few exceptions, our B\&S without any further enhancement outperforms existing solution algorithms, highly specialized for their respective specific models. The B\&S has also been applied to an extension of the $G$-median where lower bounds on the flows through activated interhub arcs are imposed. To the best of our knowledge such models have not been tested computationally before.
    
    Therefore, in our opinion, the proposed formulation and solution framework open a new, highly promising, avenue of research. From the modeling perspective, the main challenge is now to specialize the formulation template for specific families of HLPs. From the B\&S solution framework, the main  challenges are twofold.  On the one hand, to investigate whether effective strategies can be devised for the integration of the feasibility and optimality cuts within a B\&S. On the other hand, most of our experiments focused on uncapacitated HLPs for which the lower-level subproblems at the nodes of the search tree are easy to solve. Adding capacity (or other type of) constraints to the models, renders this approach more demanding. Thus, a relevant challenge is to derive efficient solution methods for more general second-level subproblems.

\section{Code and Data Disclosure}\label{sec:Code and Data Disclosure}
	Implemented code and data used in the current work will be made available shortly on a github repository after improving the readability, the rough version is available on request.

\newpage
\renewcommand{\thesection}{A-\arabic{section}}
\setcounter{section}{0}
    \section{Formulation for MA R-GHLP \label{appendix:R-GHLP}}
        $$(R-GHLP)\quad \min\, F(z)+G(y)+\sum_{(i,j) \in A}\!\!\!c_{ij}\left (\gamma h^1_{ij}+ \alpha f_{ij}+ \theta h^2_{ij}\right)$$ 
        \noindent Subject to:
        \begin{align}
            &  \eqref{rel_X-Y-0}-\eqref{rel_z-X-2_2-bis}\notag\\
        	&  (1-z_i)\,O_i=\sum_{(i,j)\in A}h^1_{ij}  &&  i\in V & \eqref{leaves_i_H-0}\notag\\
        	&  (1-z_i)\,D_i=\sum_{(j,i)\in A}h^2_{ji}  &&  i\in V & \eqref{enters_i_H-0}\notag\\
        	&  O_i\, z_i+\sum_{j\ne i}h^1_{ji}+\sum_{j\ne i}t_{ji}= \notag\\
            &  D_i\, z_i+ \sum_{j\ne i}h^2_{ij}+\sum_{j\ne i}t_{ij} &&  i\in V & \eqref{Flow_Balance-0}\notag\\
        	&  (t+h^1+h^2)(\delta^+(S))\geq W(S:S^c) &&  S\subset V:\,\exists r\in R, \text{ s.t. } (o^r, d^r)\in\delta^+(S) &\eqref{Demand-0}\notag\\
        	&  w_{ij} x^1_{ij}\leq h^1_{ij}\leq O_i x^1_{ij} &&  (i, j)\in A & \eqref{H_X2_1-0}\notag\\
        	&  w_{ij} x^2_{ij}\leq h^2_{ij}\leq D_j x^2_{ij} &&  (i, j)\in A & \eqref{H_X2_2-0}\notag\\
        	&  (w_{km}+w_{mk}) y_{km}\leq t_{km}+t_{mk}\leq \overline W y_{km} &&  km \in E & \eqref{f_Y-0} \notag\\
            &  \eqref{domain_z-0}-\eqref{domainY-0},\, \eqref{domain_F-H-0-0}.\notag
    	\end{align}

    \section{Proofs \label{appendix:proofs}}	
    	\noindent\textbf{Proof of Proposition \ref{propo0}:}
    	Every edge $ij\in E$ produces two different inequalities: \eqref{sp:1_a}, which involves the flow through arc $(i, j)$, and \eqref{sp:1_b}, which involves the flow through the opposite arc $(j, i)$. We prove the validity of  \eqref{sp:1_a}, as the proof for the opposite direction is the same  interchanging $i$ and $j$.
    	
    	Inequality \eqref{sp:1_a} trivially holds when $y_{ij}=0$, so let $s=(z,\, y,\, x^1,\, x^2)\in\mathcal S$ with $y_{ij}=1$, be a feasible solution network, possibly  $s\ne \overline s$, and $f\in \mathcal L(s)$ an $R$-feasible flow in $G(s)$.
    	Since $y_{ij}=1$, the value of the right-hand-side of the inequality is $t_{ij}(\overline s)$. Furthermore, the flow $t_{ij}$ will be at least $t_{ij}(\overline s)$ unless $G(s)$ contains some element that is not activated in $G(\overline s)$  that would allow for a better $r$-path for some commodity routed through $(i,j)$ in $\overline s$ (i.e., for some $r\in R_{ij}(\overline s)$).
    	
    	Consider the following cases:
    	\begin{enumerate}
        	\item[$(i)$:] For all $r\in R_{ij}(\overline s)$, it holds that $z_k=0$ for all $k\in Z^{r}(\overline s)$  and $y_{km}=0$ for all $km\in Y^{r}(\overline s)$.\\
        	Then, $P^r(s)=P^r(\overline s)$, for all $r\in R_{ij}(\overline s)$, since $G(s)$  contains no element that could produce an $r$-path better than $P^r(\overline s)$ for some $r \in R_{ij}(\overline s)$. Hence, $t_{ij}$ should be at least $t_{ij}(\overline s)$ and the inequality holds.
        	\item[$(ii)$:] There exists $r\in R_{ij}(\overline s)$ such that $z_k=1$ for some $k\in Z^{r}(\overline s)$.\\
        	In this case, commodity $r$ will no longer be routed through $P^r(\overline s)$, since the support graph $G(s)$ would contain the $r$-path $o-k-d$, which has a lower routing cost than $P^r(\overline s)$. Thus, the value of $t_{ij}$ would no longer account for the demand $w^r$. This is compensated by the second term of left-hand-side, which accounts for $w^r$ since $z_k=1$.
        	\item[$(iii)$:] There is some $r\in R_{ij}(\overline s)$ such that $y_{km}=1$ for some $km\in Y^{r}(\overline s)$ with $k, m\notin Z^{r}(\overline s)$.\\
        	As in the previous case, commodity $r$ will no longer be routed through $P^r(\overline s)$, since the support graph $G(s)$ would contain the $r$-path $o-k-m-d$, which has a lower routing cost than $P^r(\overline s)$. Again, the value $t_{ij}$ would not account for the demand $w^r$. Nevertheless, this would again be compensated by the second term of left-hand-side since $y_{km}=1$. \hfill $\blacksquare$
    	\end{enumerate}
        
    	\noindent\textbf{Proof of Proposition \ref{propo:access}:}
    	\begin{itemize}
        	\item[$(a)$] We first prove the validity of \eqref{sp:2_access_2-1} for a given arc $(0,\, i)\in A$. Since $R^1_{oi}(\overline s)=R^{1,1}_{oi}(\overline s)\cup R^{1,2}_{oi}(\overline s)$, we observe that its right-hand-side takes the value 0 when $R^1_{oi}(\overline s)=\emptyset$, i.e., $(o,\, i)$ is not activated as an access arc. Hence, let $s=(z,\, y,\, x^1,\, x^2)\in\mathcal S$ with $x^1_{oi}=1$, be a feasible solution network, possibly  $s\ne \overline s$, and $f\in \mathcal L(s)$ an $R$-feasible flow in $G(s)$.
        	We now consider the following cases:
        	\begin{enumerate}
        		\item[$(i)$:] For all $r\in {R}^1_{oi}(\overline s)$, it holds that $z_k=0$ for all $k\in Z^{r}(\overline s)$  and $y_{km}=0$ for all $km\in Y^{r}(\overline s)$.\\
        		As explained, if $G(s)$ contains no element that may produce an $s$-$r$-path better than $P^r(\overline s)$, for all $r\in {R}$, then the flow through $(o, i)$, $h^1_{oi}$ should be at least the one indicated by the right-hand-side of the inequality.
        		Therefore, the inequality holds in this case.
        		\item[$(ii)$:] There is some $r\in {R}^1_{oi}(\overline s)$ such that $z_k=1$ for some $k\in Z^{r}(\overline s)$.\\
        		Commodity $r$ would no longer be routed through $P^r(\overline s)$, since $G(s)$ would contain the $r$-path $o-k-d$, with
        		a lower cost than $P^r(\overline s)$. Thus, $h^1_{oi}$ would no longer account for the demand $w^r$. This would be compensated by the second term of left-hand-side, which accounts for $w^r$ since $z_k=1$.
        		\item[$(iii)$:] There is some $r\in {R}^1_{oi}(\overline s)$ such that $y_{km}=1$ for some $km\in Y^{r}(\overline s)$, with $k\ne i$, $k,m\notin Z^{r}(\overline s)$.\\
        		Commodity $r$ would no longer be routed through $P^r(\overline s)$, since $G(s)$ would contain the $r$-path $o-k-m-d$, with a lower routing cost than $P^r(\overline s)$. Again, $h^1_{oi}$ would no longer account for$w^r$, although this would be compensated by the second term of left-hand-side since $y_{km}=1$. 
        	\end{enumerate}
        	\item[$(b)$]  The validity of \eqref{sp:2_access_2_2} follows from that of \eqref{sp:2_access_2-1}, using arguments similar to those used for the reinforcement of \eqref{sp:1_a} and \eqref{sp:1_b}. \hfill $\blacksquare$
    	\end{itemize}
    	
    	\noindent\textbf{Proof of Theorem \ref{teo}:}
    	Since we have proven that inequalities \eqref{eq1_1}-\eqref{eq3} are valid for the $H$-median, 
    	it only remains to see that for any  solution $(s, f)$ with $s\in\mathcal S$, $f\in\mathcal F(s)$, that satisfies \eqref{eq1_1}-\eqref{eq3}, then $f$ is $R$-feasible for $s$. Let $(\bar s, \bar f)$ be a solution in the domain of the formulation $H$-med 
        and suppose that $\bar f\in\mathcal F(\bar s)$ is not $R$-feasible for $\bar s$. That is, the flow $\bar f$ can not be decomposed in individual $s$-$r$-paths. Then, the auxiliary problem $Aux_{(\bar s, \bar f)}$ will be unfeasible so at least one of the feasibility conditions \eqref{feas1}-\eqref{feas3} will not hold. Thus, at least one of the following cases will hold:
    	\begin{itemize}
        	\item  There exists $(\bar k,\, \bar m)\in A$ such that $\bar t_{\bar k \bar m}< t_{\bar k \bar m}(\bar s)$. In this  case, the inequality \eqref{eq1_1}  would be violated for arc $(\bar k,\, \bar m)$, contradicting that $\bar s$ satisfies all inequalities \eqref{eq1_1}-\eqref{eq3}.
        	\item  There exists $(o,\, \bar k)\in A$ such that  $\overline h^1_{o\bar k}\geq h^1_{o\bar k}(\bar s)$.  In this  case, the inequality \eqref{eq2} would be violated for arc $(o,\, \bar k)$, contradicting that $\bar s$ satisfies all inequalities \eqref{eq1_1}-\eqref{eq3}.
        	\item There exists $(\bar m,\, d)\in A$ such that $\bar h^2_{\bar md}\geq h^2_{\bar jd}(\bar s)$. In this  case, the inequality \eqref{eq3} would be violated for arc $(\bar m,\, d)$, contradicting that $\bar s$ satisfies all inequalities \eqref{eq1_1}-\eqref{eq3}. \hfill $\blacksquare$
    	\end{itemize}

    \section{Procedure for finding feasible $R$-flows $f(\bar s)\in \mathcal L(\bar s)$}\label{shortest:app}
    	\begin{algorithm}
    		\caption{Procedure for finding $f(\bar s) \in \mathcal L(\bar s)$}\label{R-feas:alg}
            \small
    		\hspace*{\algorithmicindent} \textbf{Input:} Design solution $\bar s\in\mathcal S$
    		{\setlength{\baselineskip}{10pt}
    			\begin{algorithmic}[1]
    				\For{$r\in R$}
    				\State $P^r(\bar s)\gets \text{ shortest path from $o^r$ to $d^r$ in } G(\bar s)$
    				\State $A^r(\bar s)\gets \{(i, j)\in G(\bar s): (i, j)\in P^r(\bar s)\}$
    				\EndFor
    				\For{$(i,j)\in G(\bar s)$}
    				\State $flow\gets\sum_{r\in R: (i, j)\in A^r(\bar s)}w^r$
    				\If{$\bar y_{ij}+\bar y_{ji}=1$}
    				\State $t_{ij}(\bar s)\gets flow$
    				\ElsIf{$\bar x^1_{ij}=1$}
    				\State $h^1_{ij}(\bar s)\gets flow$
    				\ElsIf{$\bar x^2_{ij}=1$}
    				\State $h^2_{ij}(\bar s)\gets flow$
    				\EndIf
    				\EndFor
    		\end{algorithmic}}
    		\hspace*{\algorithmicindent} \textbf{Output:} {$R$-feasible flow $f(\bar s)\in \mathcal L(\bar s)$}
    	\end{algorithm}

    \section{Figures and tables with numerical results}\label{appendix:1}
		\begin{minipage}[t]{.48\textwidth}
			\begin{table}[H]
                \resizebox{\textwidth}{!}
                {
                    \begin{tabular}{c|rrrrrrrrrr}
                        {$c_{ij}$} & {1} & {2} & {3} & {4} & {5} & {6} & {7} & {8} & {9} & {10} \\ \hline
                        1 & 0 & 414 & 1426 & 315 & 528 & 522 & 1169 & 2066 & 280 & 899 \\
                        2 & 414 & 0 & 1833 & 167 & 852 & 784 & 1574 & 2454 & 288 & 1275 \\
                        3 & 1426 & 1833 & 0 & 1738 & 1269 & 1412 & 613 & 704 & 1670 & 950 \\
                        4 & 315 & 167 & 1738 & 0 & 692 & 618 & 1438 & 2381 & 121 & 1124 \\
                        5 & 528 & 852 & 1269 & 692 & 0 & 143 & 812 & 1969 & 576 & 449 \\
                        6 & 522 & 784 & 1412 & 618 & 143 & 0 & 948 & 2111 & 498 & 573 \\
                        7 & 1169 & 1574 & 613 & 1438 & 812 & 948 & 0 & 1288 & 1340 & 391 \\
                        8 & 2066 & 2454 & 704 & 2381 & 1969 & 2111 & 1288 & 0 & 2328 & 1649 \\
                        9 & 280 & 288 & 1670 & 121 & 576 & 498 & 1340 & 2328 & 0 & 1014 \\
                        10 & 899 & 1275 & 950 & 1124 & 449 & 573 & 391 & 1649 & 1014 & 0
                    \end{tabular}
                }
                \caption{Unit routing costs. \label{tab:cab10_c}}
			\end{table}
		\end{minipage}
		\hfill
		\begin{minipage}[t]{.48\textwidth}
			\begin{table}[H]
                \resizebox{\textwidth}{!}{
                    \begin{tabular}{c|rrrrrrrrrr}
                        {$w_{ij}$} & 1 & 2 & 3 & 4 & 5 & 6 & 7 & 8 & 9 & 10 \\ \hline
                        1 & 0 & 34 & 45 & 84 & 426 & 51 & 11 & 8 & 453 & 5 \\
                        2 & 34 & 0 & 131 & 6 & 693 & 80 & 24 & 5 & 1262 & 22 \\
                        3 & 45 & 131 & 0 & 36 & 398 & 37 & 390 & 40 & 1215 & 54 \\
                        4 & 84 & 6 & 36 & 0 & 452 & 26 & 9 & 3 & 5 & 7 \\
                        5 & 426 & 693 & 398 & 452 & 0 & 3746 & 533 & 45 & 6469 & 1223 \\
                        6 & 51 & 80 & 37 & 26 & 3746 & 0 & 31 & 11 & 448 & 70 \\
                        7 & 11 & 24 & 390 & 9 & 533 & 31 & 0 & 6 & 230 & 198 \\
                        8 & 8 & 5 & 40 & 3 & 45 & 11 & 6 & 0 & 29 & 1 \\
                        9 & 453 & 1262 & 1215 & 5 & 6469 & 448 & 230 & 29 & 0 & 161 \\
                        10 & 5 & 22 & 54 & 7 & 1223 & 70 & 198 & 1 & 161 & 0
                    \end{tabular}
                }
                \caption{Commodities demands. \label{tab:cab10_w}}
			\end{table}
		\end{minipage}

        \begin{figure}[H]
            \centering
    		\includegraphics[width=.35\textwidth, keepaspectratio]{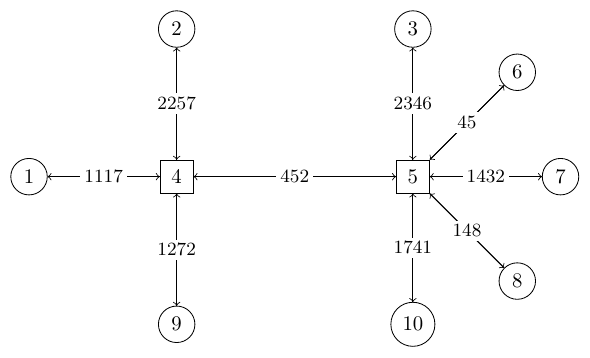}
    		\caption{Optimal solution to R-GHLP without aggregated demand constraints, $(s^1, f^1)$.\label{fig:cab10_nodmnd_flows}}
        \end{figure}
	
    	\begin{figure}[H]
            \centering
    		{\subcaptionbox{\small Optimal R-GHLP solution, $(s^2, f^2)$. \label{fig:cab10_dmnd_flows}}{\includegraphics[width=.35\textwidth, keepaspectratio]{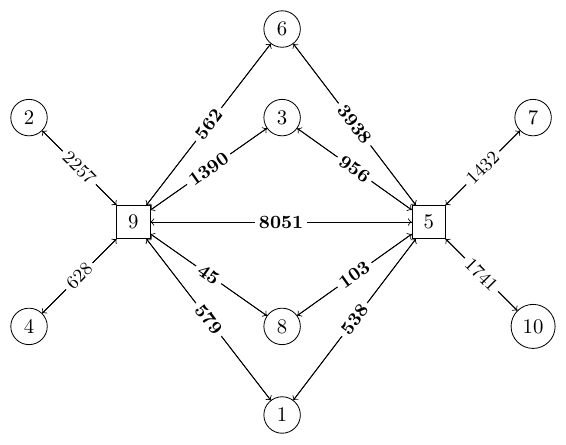}}
            \hfill
            \subcaptionbox{\small Optimal GHLP solution, $(s^*, f^*)$. \label{fig:cab10_feas_flows}}{\includegraphics[width=.35\textwidth, keepaspectratio]{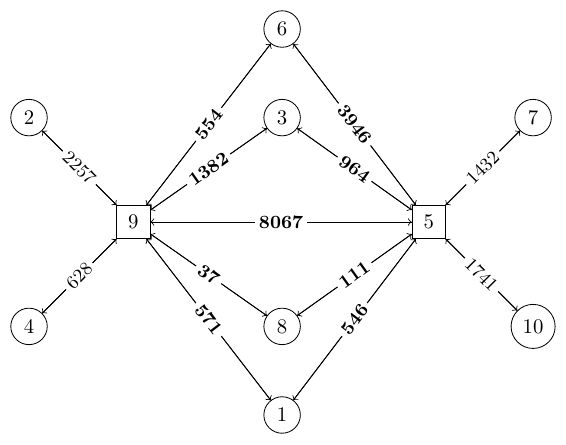}}
    		}
    		\caption{Optimal solutions for R-GHLP ($(s^2, f^2)$) \& GHLP $(s^*, f^*)$}
    	\end{figure}
		
    	\begin{figure}[H]
            \centering
    		{
    			\subcaptionbox{\small SA policy. \label{SA}}{\includegraphics[width=0.35\textwidth]{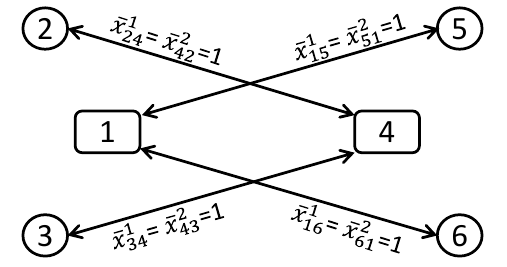}}
    			\hfill
    			\subcaptionbox{\small MA policy. \label{MA}}{\includegraphics[width=0.35\textwidth]{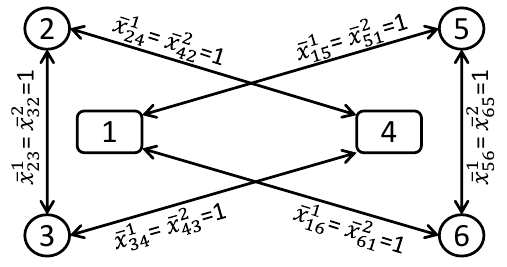}}
    		}
    		\caption{Feasible solutions with violated connectivity inequalities for a network. \label{fig:connect}}
    		{\footnotesize Demand: $w_{15}=w_{16}=w_{56}=w_{51}=w_{61}=w_{65}=w_{23}=w_{24}=w_{34}=w_{32}=w_{42}=w_{43}=1$, $w_{ij}=0$, otherwise.}
    	\end{figure}
    	
    	\begin{figure}[H]
            \centering
    		{\includegraphics[width = .75\textwidth, keepaspectratio]{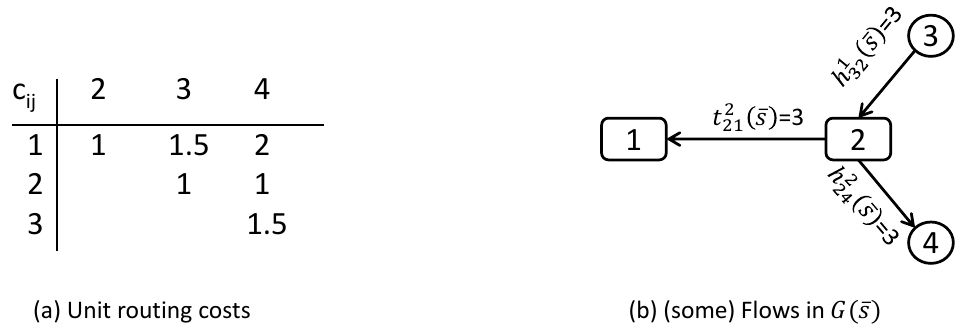}}
    		\caption{Illustration of flows through access arcs \label{ex:2}}
    	\end{figure}

        \begin{table}[htbp]
            \scriptsize
            \centering
            \resizebox{.95\textwidth}{!}{
                 \begin{tabular}{c|c|c|rr@{.}l|rrr@{.}l|rrrr@{.}l||rr@{.}l|rrr@{.}l}
                    &		&		&	\multicolumn{12}{c}{$H$-median}	&	\multicolumn{7}{c}{$G$-median}\\ [5pt]
                    &	 	&		&	\multicolumn{3}{c|}{BS}	&	\multicolumn{4}{c|}{BSF}	&	\multicolumn{5}{c||}{BSO}	&	\multicolumn{3}{c|}{BS}	&	\multicolumn{4}{c}{BSF}	\\
            [2pt]
            \hline
                &	$n$	&	$\alpha$	&	\multicolumn{1}{c}{\#Exp}	&	\multicolumn{2}{c|}{cpu}	&	\multicolumn{1}{c}{\#Exp}	&		 \multicolumn{1}{c}{Feas}	 &\multicolumn{2}{c|}{cpu} &	\multicolumn{1}{c}{\#Exp}	&	\multicolumn{1}{c}{Acc.}&	 \multicolumn{1}{c}{Int.} & \multicolumn{2}{c||}{cpu}	 &	\multicolumn{1}{c}{\#Exp}	&	\multicolumn{2}{c|}{cpu} &		 \multicolumn{1}{c}{\#Exp}	&		 \multicolumn{1}{c}{Feas}	 &\multicolumn{2}{c}{cpu}\\
            [2pt]
            \hline
            {SA}                    &	                25 	&	0.2	&	   1	&	   0&08	&	   1	&	   1	&    0&13 		&	   1  &		0	&	   0   	&	   0&23  &	1  &	0&08    &   1   &   1   &   0&14\\
                                    &	                    &	0.5	&	   1	&	   0&10	&	   1	&	   1	&    0&14 		&	   1  &		0	&	   0   	&	   0&20  &	1  &	0&09    &   1   &   1   &   0&13\\
                                    &	                    &	0.8	&	   1	&	   0&10	&	   1	&	   1	&    0&16 		&	   1  &		0	&	   0   	&	   0&19  &	1  &	0&14    &   1   &   1   &   0&22\\
            \cdashline{2-22}[0.5pt/3pt]
                                    &	                50 	&	0.2	&	   1	&	   0&67	&	   1	&	   1	&    0&98 		&	   1  &		0	&	   0   	&	   2&45  &	1  &	0&96    &   1   &   1   &   100&81\\
                                    &	                    &	0.5	&	   1	&	   1&02	&	   1	&	   1	&    1&60 		&	   1  &		0	&	   0   	&	   2&63  &	1  &	1&27    &   1   &   1   &   101&14\\
                                    &	                    &	0.8	&	   1	&	   1&40	&	  11	&	   3	&  309&59 		&	   1  &		0	&	 118   	&	  10&81  &	1  &	7&45    &   1   &   1   &   101&42\\
            \cdashline{2-22}[0.5pt/3pt]
                                    &	                75 	&	0.2	&	   7	&	 145&18	&	   9	&	   1	&  282&16 		&	   7  &		0	&	  26   	&	  99&04  &	  9 &	79&07   &  29   &   1   &   141&46\\	
                                    &	                    &	0.5	&	 439	&	 177&49	&	 474	&	   6	&  834&23 		&	 414  &		0	&	 418   	&	 254&69  &	499 &	205&65  &  1463 &   8   &   941&96\\
                                    &	                    &	0.8	&	 647	&	 342&91	&	 723	&	  13	& 1695&19 		&	 621  &		0	&	1106   	&	 401&64  &	250 &	297&63  &  1093 &   9   &  1111&66\\
            \cdashline{2-22}[0.5pt/3pt]
                                    & 	               100  &	0.2	&	  19	&	 200&51 &	  17	&	   2	&  479&43 		&	  17  &		0	&	  34   	&	3296&51 &   17  &	500&79  &   35  &   1   &   569&53\\	
                                    &	                    &	0.5	&	1370	&	 912&07 &	2309	&	  10	& 2759&03 		&	2567  &		0	&	1088   	&\multicolumn{2}{c||}{TL}   & 1704  &  1914&37  &   2573   &   13   &   2534&96\\
                                    &	                    &	0.8	&	5333	&	1717&02 &	6281	&	  29	& 6873&56 		&	3867  &		0	&	2588   	&\multicolumn{2}{c||}{TL}   & 4574  &  2955&68  &   3448   &   40   &   5955&06\\
            \hline
            MA                      &	                25 	&	0.2	&	   1	&	  0&18	&	    1	&	   1	&  24&84  		&	  31  &		  0	&	   6   &	   2&65  &	 1    &	 0&07  &   1   &   1   &   0&13\\
                                    &	                    &	0.5	&	   3	&	  0&92	&	    1	&	   1	&  26&28  		&	 178  &		  0	&	 170   &	   5&93  &	 3    &	 0&62  &   3   &   6   &   3&83\\
                                    &	                    &	0.8	&	  73	&	  2&52	&	   54	&	   7	&  72&30  		&	  96  &		124	&	1368   &	  15&82  &	64    &	 1&82  & 138   &  15   &  24&48\\
            \cdashline{2-22}[0.5pt/3pt]
                                    &	                50 	&	0.2	&	   1	&	  1&17	&	    1	&      1	&  611&15 		&	  16  &		  0	&	 126  &		  38&34 &	 1	  &	 0&78  &   1   &   1   &   100&83\\
                                    &	                    &	0.5	&	 214	&	 11&91	&	   28	&      7	& 1318&65 		&	 232  &		  5	&	 706  &		  46&07 &	 30	  &	 9&02  &  41   &   9   &   423&92\\
                                    &	                    &	0.8	&	 213	&	 16&87	&	   23	&      8	& 2032&00 		&	 446  &		 56	&	1270  &		  46&93 &	185	  &	21&44  & 368   &  20   &   677&97\\
            \cdashline{2-22}[0.5pt/3pt]
                                    &	                75 	&	0.2	&	   7	&	 28&64  &	 119	&     11	& 1852&94 		&	 401  &		  1	&	  92   &	 278&10   &	  17	&	   33&40  &   39   &   4   &   148&92\\
                                    &	                    &	0.5	&	 247	&	 69&07  &	 393	&     23	& 3679&57 		&	3100  &		203	&	2936   &	1799&87   &	  595	&	  133&96  &  910   &  15   &  1249&50\\
                                    &	                    &	0.8	&	1917	&	339&55  &	2505	&     65	& \multicolumn{2}{c|}{TL}  &	9503  &		337	&	7540   &\multicolumn{2}{c||}{TL}   &  2646   &	845&76  &  1892   &  129   &  \multicolumn{2}{c}{TL}    \\
            \cdashline{2-22}[0.5pt/3pt]
                                    &	               100 	&	0.2	&	  30	&	185&63  &	  41	&      4	&  515&31 		&	  40  &		  0 &	  38   &	 205&75   &	   51    &	153&64 &  102   &  24   &  1105&83   \\	
                                    &	                    &	0.5	&	 877	&	411&55  &	1335	&     19	& 2398&54 		&	1371  &		 19 &	 732   &	 803&36   &	 3017     &1699&19 & 5282   &  111   &  6997&98	      \\
                                    &	                    &	0.8	&	1100	&	690&23  &	2842	&     28	&\multicolumn{2}{c|}{TL}   &	4621  &		257 &	3980   &	3932&23   &	 4054     &6774&36 & 1419   &  87   &  \multicolumn{2}{c}{TL} 
            \end{tabular}}             
            \caption{Comparison among BS, BSO and BSF for $H$-median and $G$-median with CAB instances\label{comp-BS-HG}}
            {\footnotesize TL: Time limit}
        \end{table}
        
        \begin{table}[htbp]
            \centering
             \resizebox{1.\textwidth}{!}
            {
            \scriptsize
            \begin{tabular}{c|c|c|r@{.}lr@{.}lr@{.}lr@{.}l|cr@{.}lr@{.}l||r@{.}lr@{.}lr@{.}l||r@{.}lr@{.}lr@{.}lr@{.}l|r@{.}lr@{.}lr@{.}l}
            &&&\multicolumn{19}{c||}{SA} & \multicolumn{14}{c}{MA}\\[5pt]
               	&		&		&	\multicolumn{8}{c|}{$H$-median}	&	\multicolumn{5}{c||}{$H$-median [5] ($\alpha=0.75$)}	& \multicolumn{6}{c||}{$G$-median} &	 \multicolumn{8}{c|}{$H$-median}	&	 \multicolumn{6}{c}{$G$-median}\\[5pt]
        &$n$&$\alpha$&    \multicolumn{2}{c}{4I}	&    \multicolumn{2}{c}{3I}&    \multicolumn{2}{c}{best}&    \multicolumn{2}{c|}{BS}&   \multicolumn{1}{c}{$n$}&    \multicolumn{2}{c}{[5]}&    \multicolumn{2}{c||}{BS} & \multicolumn{2}{c}{3I}&    \multicolumn{2}{c}{best}&    \multicolumn{2}{c||}{BS}  	&   \multicolumn{2}{c}{4I} &    \multicolumn{2}{c}{3I}&    \multicolumn{2}{c}{best}&    \multicolumn{2}{c|}{BS}&    \multicolumn{2}{c}{3I}&    \multicolumn{2}{c}{best}&    \multicolumn{2}{c}{BS} \\[3pt]
        \hline  
        CAB & 25 & 0.2 &	2617&15  & 	 3&10     &	  3&37     &	   0&08     & \multicolumn{1}{c}{}&\multicolumn{2}{c}{}&\multicolumn{2}{c||}{}&	8&04    &	30&67     &	   0&08   	             &	  72&40	  &   3&98 	    &    7&94       &	   0&18      &	23&68    &	188&58   	&	   0&07 	\\
        &    & 0.5 &	2516&27  & 	 3&32     &	  3&46     &	   0&10     & \multicolumn{1}{c}{}&\multicolumn{2}{c}{}&\multicolumn{2}{c||}{}&	6&18    &	34&27     &	   0&09   	             &	  63&52	  &   1&83 	    &    5&78       &	   0&92      &	22&80    &	169&85   	&	   0&62 	 \\
        &    & 0.8 &	2622&05  & 	 2&03     &	  3&21     &	   0&10     &	\multicolumn{1}{c}{}&\multicolumn{2}{c}{}&\multicolumn{2}{c||}{}&  3&21    &	62&06     &	   0&14   	             &	  83&68	  &   3&53 	    &    9&87       &	   2&52      &	23&35    &	145&14   	&	   1&82 	 \\
        \cdashline{2-11}[0.5pt/3pt]\cdashline{17-36}[0.5pt/3pt]
        & 50 & 0.2 &	   \multicolumn{2}{c}{}    & 	 4&69     &	 79&14     &	   0&67     & \multicolumn{1}{c}{}&\multicolumn{2}{c}{}&\multicolumn{2}{c||}{}&  \multicolumn{2}{c}{}      &	 \multicolumn{2}{c}{}       &	   0&96   	             &	7142&70	  &  93&72 	    &   89&31       &	   1&17      &	 \multicolumn{2}{c}{}      &	\multicolumn{2}{c}{}     	 &0&78 	 \\
        &    & 0.5 &	   \multicolumn{2}{c}{}    & 	 5&28     &	 68&75     &	   1&02     &	\multicolumn{1}{c}{}&\multicolumn{2}{c}{}&\multicolumn{2}{c||}{}& \multicolumn{2}{c}{}      &	 \multicolumn{2}{c}{}       &	   1&27   	             &	6747&83	  &  74&77 	    &  121&46       &	  11&91      &	 \multicolumn{2}{c}{}      &	 \multicolumn{2}{c}{}     	 &9&02 	 \\
        &    & 0.8 &	   \multicolumn{2}{c}{}    & 	10&23     &	 96&66     &	   1&40     &	\multicolumn{1}{c}{}&\multicolumn{2}{c}{}&\multicolumn{2}{c||}{}&  \multicolumn{2}{c}{}      &	 \multicolumn{2}{c}{}       &	   7&45   	             &	5904&41	  &  89&23 	    &  128&37       &	  16&87      &	 \multicolumn{2}{c}{}      &	 \multicolumn{2}{c}{}     	&21&44 	 \\
        \cdashline{2-11}[0.5pt/3pt]\cdashline{17-36}[0.5pt/3pt]
        & 75 & 0.2 &	   \multicolumn{2}{c}{}    &   \multicolumn{2}{c}{}       &	  \multicolumn{2}{c}{}       &	 145&18     &	 \multicolumn{1}{c}{}&\multicolumn{2}{c}{}&\multicolumn{2}{c||}{}&  \multicolumn{2}{c}{}      &	 \multicolumn{2}{c}{}       &	  79&07   	             &    \multicolumn{2}{c}{}     &   \multicolumn{2}{c}{}   & 965&40       &	  28&64      &	 \multicolumn{2}{c}{}      &	  \multicolumn{2}{c}{}     	&	  33&40 	 \\
        &    & 0.5 &	   \multicolumn{2}{c}{}    &   \multicolumn{2}{c}{}       &	  \multicolumn{2}{c}{}       &	 177&49     &  \multicolumn{1}{c}{}&\multicolumn{2}{c}{}&\multicolumn{2}{c||}{}& 	\multicolumn{2}{c}{}      &	 \multicolumn{2}{c}{}       &	 205&65   	             &    \multicolumn{2}{c}{}     &   \multicolumn{2}{c}{}   	& 27&66       &	  69&07      &	 \multicolumn{2}{c}{}      &	  \multicolumn{2}{c}{}     	&	 133&96 	 \\
        &    & 0.8 &	   \multicolumn{2}{c}{}    &   \multicolumn{2}{c}{}       &	  \multicolumn{2}{c}{}       &	 342&91     &	 \multicolumn{1}{c}{}&\multicolumn{2}{c}{}&\multicolumn{2}{c||}{}&  \multicolumn{2}{c}{}      &	 \multicolumn{2}{c}{}       &	 297&63   	             &    \multicolumn{2}{c}{}     &   \multicolumn{2}{c}{}   & 435&22       &	 339&55      &	 \multicolumn{2}{c}{}      &	  \multicolumn{2}{c}{}     	&	 845&76 	 \\
        \cdashline{2-11}[0.5pt/3pt]\cdashline{17-36}[0.5pt/3pt]
        &100 & 0.2 &	   \multicolumn{2}{c}{}    &   \multicolumn{2}{c}{}       &	  \multicolumn{2}{c}{}       &	 200&51     &  \multicolumn{1}{c}{}&\multicolumn{2}{c}{}&\multicolumn{2}{c||}{}
        	 &	\multicolumn{2}{c}{}      &	 \multicolumn{2}{c}{}       &	 500&79   	             &    \multicolumn{2}{c}{}     &   \multicolumn{2}{c}{}   &    \multicolumn{2}{c}{}     &	 185&63      &	 \multicolumn{2}{c}{}      &	  \multicolumn{2}{c}{}     	&	 153&64 	 \\
        &    & 0.5 &	   \multicolumn{2}{c}{}    &   \multicolumn{2}{c}{}       &	  \multicolumn{2}{c}{}       &	 912&07     &    \multicolumn{1}{c}{}&\multicolumn{2}{c}{}&\multicolumn{2}{c||}{}  &	\multicolumn{2}{c}{}      &	 \multicolumn{2}{c}{}       &	1914&37   	             &    \multicolumn{2}{c}{}     &   \multicolumn{2}{c}{}   &    \multicolumn{2}{c}{}      &	 411&55      &	 \multicolumn{2}{c}{}      &	  \multicolumn{2}{c}{}     	&	1699&19 	 \\
        &    & 0.8 &	   \multicolumn{2}{c}{}    &   \multicolumn{2}{c}{}       &	  \multicolumn{2}{c}{}       &	1717&02     &	 \multicolumn{1}{c}{}&\multicolumn{2}{c}{}&\multicolumn{2}{c||}{}&  \multicolumn{2}{c}{}      &	 \multicolumn{2}{c}{}       &	2955&68   	             &    \multicolumn{2}{c}{}     &   \multicolumn{2}{c}{}   	   	&    \multicolumn{2}{c}{} &	 690&23      &	 \multicolumn{2}{c}{}      &	  \multicolumn{2}{c}{}     	&	 6774&36 	 \\[5pt] 
        \hline  
        AP& 25 & 0.2 &	3288&82  & 	 4&35     &	  8&39     &	  0&56      &	\multicolumn{1}{c}{$n$} & \multicolumn{2}{c}{[5]} & \multicolumn{2}{c||}{$BS$}  &  5&87 	   	
             & 31&63     &	   0&74   			      &	  80&49	  &   4&19 	    &    9&21       &	   1&57      &	13&20    &	132&82   	& 	   0&47		\\  
        \cline{12-16}     
        &    & 0.5 &	2647&65  & 	 9&68     &	 18&98     &	  0&14      &	100 & 12&70 & 219&13  & 4&37    &	  	41&37     &	   0&17   			      &	  62&32	  &   4&01 	    &    6&50       &	   2&36      &	12&22    &	102&68   	& 	   1&37		\\
        &    & 0.8 &	3006&73  & 	10&63     &	 13&96     &	  0&77      &	125	& 49&10 &	2322&32 & 5&18    &	52&64     &	   0&92   			      &	  67&74	  &   3&91 	    &    5&77       &	   3&44      &	13&30    &	160&40   	& 	   1&77		\\
        \cdashline{2-11}[0.5pt/3pt]\cdashline{17-36}[0.5pt/3pt]
        & 50 & 0.2 &	   \multicolumn{2}{c}{}    &   \multicolumn{2}{c}{}       &	165&73     &	  0&61   & 150 & 132&50	& 2626&91   &	\multicolumn{2}{c}{}      &	 \multicolumn{2}{c}{}       &	   0&93   			      &	   \multicolumn{2}{c}{TL} 	  & 103&07 	    &   85&90       &	  11&91      &	 \multicolumn{2}{c}{}      &\multicolumn{2}{c}{}     	& 	   0&68		\\
        &    & 0.5 &	   \multicolumn{2}{c}{}    &   \multicolumn{2}{c}{}       &	433&87     &	  1&09 & 175 &  554&50	&  5215&33     &	\multicolumn{2}{c}{}      &	 \multicolumn{2}{c}{}       &	   1&70   			      &	   \multicolumn{2}{c}{TL} 	  &  77&91 	    &   77&74       &	  11&02      &	 \multicolumn{2}{c}{}      &\multicolumn{2}{c}{}     	& 	  17&32	 	\\
        &    & 0.8 &	   \multicolumn{2}{c}{}    &   \multicolumn{2}{c}{}       &	243&77     &	  1&22   &200 & 16489&10	&   5395&12    &	 \multicolumn{2}{c}{}      &	 \multicolumn{2}{c}{}       &	  10&78   			      &	   \multicolumn{2}{c}{TL} 	  &  92&55 	    &  128&49       &	  21&23      &	 \multicolumn{2}{c}{}      &\multicolumn{2}{c}{}     	& 	  37&00  	\\
        \cdashline{2-36}[0.4pt/3pt]
        & 75 & 0.2 &	   \multicolumn{2}{c}{}    &   \multicolumn{2}{c}{}       &	  \multicolumn{2}{c}{}       &	  4&63   &\multicolumn{1}{c}{}&\multicolumn{2}{c}{}&\multicolumn{2}{c||}{}   &	\multicolumn{2}{c}{}      &	 \multicolumn{2}{c}{}       &	  10&17   			      &    \multicolumn{2}{c}{}     &   \multicolumn{2}{c}{}   	    &  474&66       &	  27&54      &	    \multicolumn{2}{c}{}      &	  \multicolumn{2}{c}{}     	& 	   4&97		\\
        &    & 0.5 &	   \multicolumn{2}{c}{}    &   \multicolumn{2}{c}{}       &	  \multicolumn{2}{c}{}       &	  7&63  &\multicolumn{1}{c}{}&\multicolumn{2}{c}{}&\multicolumn{2}{c||}{}    &	\multicolumn{2}{c}{}      &	 \multicolumn{2}{c}{}       &	  16&07   			      &    \multicolumn{2}{c}{}     &   \multicolumn{2}{c}{}   	    &  443&20       &	  38&40      &	    \multicolumn{2}{c}{}      &	  \multicolumn{2}{c}{}     	& 	  14&95		\\
        &    & 0.8 &	   \multicolumn{2}{c}{}    &   \multicolumn{2}{c}{}       &	  \multicolumn{2}{c}{}       &	 10&80   &\multicolumn{1}{c}{}&\multicolumn{2}{c}{}&\multicolumn{2}{c||}{}   &	\multicolumn{2}{c}{}      &	 \multicolumn{2}{c}{}       &	  28&93   			      &    \multicolumn{2}{c}{}     &   \multicolumn{2}{c}{}   	    &  566&16       &	  68&46      &	     \multicolumn{2}{c}{}      &	  \multicolumn{2}{c}{}     	& 	 300&17	    \\
        \cdashline{2-11}[0.5pt/3pt]\cdashline{17-36}[0.5pt/3pt]
        &100 & 0.2 &	   \multicolumn{2}{c}{}    &   \multicolumn{2}{c}{}       &	  \multicolumn{2}{c}{}       &	 29&58      & \multicolumn{1}{c}{}&\multicolumn{2}{c}{}&\multicolumn{2}{c||}{} 	& \multicolumn{2}{c}{}      &	 \multicolumn{2}{c}{}       &	  30&75   			      &    \multicolumn{2}{c}{}     &   \multicolumn{2}{c}{}   	    &    \multicolumn{2}{c}{}         &	     121&72      &	 \multicolumn{2}{c}{}      &	  \multicolumn{2}{c}{}     	& 	  18&67 		\\
        &    & 0.5 &	   \multicolumn{2}{c}{}    &   \multicolumn{2}{c}{}       &	  \multicolumn{2}{c}{}       &	 31&48      &	 \multicolumn{1}{c}{}&\multicolumn{2}{c}{}&\multicolumn{2}{c||}{}     &\multicolumn{2}{c}{}      &	 \multicolumn{2}{c}{}       &	  29&93   			      &    \multicolumn{2}{c}{}     &   \multicolumn{2}{c}{}   	    &    \multicolumn{2}{c}{}         &	     165&28      &	 \multicolumn{2}{c}{}      &	  \multicolumn{2}{c}{}     	& 	  55&86 		\\
        &    & 0.8 &	   \multicolumn{2}{c}{}    &   \multicolumn{2}{c}{}       &	  \multicolumn{2}{c}{}       &	163&34      &	 \multicolumn{1}{c}{}&\multicolumn{2}{c}{}&\multicolumn{2}{c||}{}   & \multicolumn{2}{c}{}      &	 \multicolumn{2}{c}{}       &	 134&93   			      &    \multicolumn{2}{c}{}     &   \multicolumn{2}{c}{}   	    &    \multicolumn{2}{c}{}         &	     359&88      &	 \multicolumn{2}{c}{}      &	  \multicolumn{2}{c}{}     	& 	 982&69 		\\
        \hline
            \end{tabular}}
            \caption{Comparison with best for $G$-median and $H$-median\label{comp_best:GH}}
            {\footnotesize TL: Time limit}
        \end{table}
        
        \begin{table}[htbp]
            \resizebox{.9\textwidth}{!}{
            \begin{tabular}{c|c|c|rr@{.}l|rr@{.}l||rr@{.}l|rr@{.}l}
            	&		&		&	\multicolumn{6}{c}{$H$-median} & \multicolumn{6}{c}{$G$-median}\\
                	&		&		&	\multicolumn{3}{c|}{SA}	&	 \multicolumn{3}{c||}{MA}&	\multicolumn{3}{c|}{SA}	&	 \multicolumn{3}{c}{MA}\\ [5pt]
        &$n$&$\alpha$&   \multicolumn{1}{c}{\#Exp.}	&	\multicolumn{2}{c|}{cpu}&	\multicolumn{1}{c}{\#Exp.}	&	 \multicolumn{2}{c||}{cpu}&	\multicolumn{1}{c}{\#Exp.}	&	\multicolumn{2}{c|}{cpu}&	 \multicolumn{1}{c}{\#Exp.}	&	 \multicolumn{2}{c}{cpu}\\
        [2pt]
        \hline
        CAB & 25 & 0.2 &	 1    &	   0&08  &	   1    &	  0&18	&	   1		&	   0&08	 &	   1		&	   0&07	\\
        &        & 0.5 &     1    &	   0&10  &	   3    &	  0&92	&	   1		&	   0&09	 &	   3		&	   0&62	\\
        &        & 0.8 &	 1    &	   0&10  &	  73    &	  2&52	&	   1		&	   0&14	 &	  64		&	   1&82	\\
        \cdashline{2-15}[0.5pt/3pt]
        &    50  & 0.2 &	 1    &	   0&67  &	   1    &	  1&17	&	   1		&	   0&96	 &	   1		&	   0&78	\\
        &        & 0.5 &	 1    &	   1&02  &	 214    &	 11&91	&	   1		&	   1&27	 &	  30		&	   9&02	\\
        &        & 0.8 &     1    &	   1&40  &	 213    &	 16&87	&	   1		&	   7&45	 &	 185		&	  21&44	\\
        \cdashline{2-15}[0.5pt/3pt]
        &    75  & 0.2 &	 7    &	 145&18  &	   7    &	 28&64	&	   9		&	  79&07  &	  17		&	  33&4	\\
        &        & 0.5 &   439    &	 177&49  &	 247    &	 69&07	&	 499		&	 205&65  &	 595		&	 133&96	\\
        &        & 0.8 &   647    &	 342&91  &	1917    &	339&55	&	 250		&	 297&63	 &	2646		&	 845&76	\\
        \cdashline{2-15}[0.5pt/3pt]
        &100     & 0.2 &   19    &	 200&51  &	  30    &	185&63	&	  17		&	 500&79	 &	  51		&	 153&64	\\
        &        & 0.5 & 1370    &	 912&07  &	 877    &	411&55	&	1704		&	1914&37	 &	3017		&	1699&19	\\
        &        & 0.8 & 5333    &	1717&02  &	1100    &	690&23	&	4574		&	2955&68	 &	4054		&	6774&36	\\
        \hline
        AP& 25   & 0.2 &	3    &	   0&56  &	   3    &	  1&57	&      3		&	   0&74   &     8		&	   0&47  	\\
        &        & 0.5 &	1    &	   0&14  &	   8    &	  2&36	&      1		&	   0&17   &    58		&	   1&37  	\\
        &        & 0.8 &	3    &	   0&77  &	  14    &	  3&44	&      3		&	   0&92   &    71		&	   1&77  	\\
        \cdashline{2-15}[0.5pt/3pt]
        & 50     & 0.2 &	1    &	   0&61  &	   5    &	 11&91	&      1		&	   0&93   &     1		&	   0&68  	\\
        &        & 0.5 &	1    &	   1&09  &	  25    &	 11&02	&      1		&	   1&70   &   154		&	  17&32  	\\
        &        & 0.8 &	1    &	   1&22  &	  55    &	 21&23	&      3		&	  10&78   &   488		&	  37&00  	\\
        \cdashline{2-15}[0.5pt/3pt]
        & 75     & 0.2 &	1    &	   4&63  &	  12    &	 27&54	&      1		&	  10&17   &     1		&	   4&97  	\\
        &        & 0.5 &	1    &	   7&63  &	  22    &	 38&40	&      1		&	  16&07   &    15		&	  14&95  	\\
        &        & 0.8 &	1    &	  10&80  &	  57    &	 68&46	&      1		&	  28&93   &   858		&	 300&17  	\\
        \cdashline{2-15}[0.5pt/3pt]
        &100     & 0.2 &	1    &	  29&58  &	  24    &	121&72	&      1		&	  30&75   &     1		&	  18&67  	\\
        &        & 0.5 &	1    &	  31&48  &	 165    &	165&28	&      1		&	  29&93   &    38		&	  55&86  	\\
        &        & 0.8 &   15    &	 163&34  &	 454    &	314&52	&     11		&	 134&93   &  1779		&	 982&69  	\\
        \cdashline{2-15}[0.5pt/3pt]
        &125     & 0.2 &	1    &	  60&99  &	  41    &	177&60  & 	   1		&	  63&22   &	    1		&	  74&31	\\
        &        & 0.5 &	1    &	 146&10  &	 174    &	387&32  & 	   1		&	 172&88   &	   43		&	 178&15	\\
        &        & 0.8 &  504    &	 771&24  &	 388    &	597&73  & 	 523		&	 775&31   &	  839		&	1148&44	\\
        \cdashline{2-15}[0.5pt/3pt]
        &150     & 0.2 &	1    &	 195&91  &	  25    &	749&87  & 	   1		&	 158&12   &	    1		&	 222&03	\\
        &        & 0.5 &	1    &	 356&88  &	 205    &	862&87  & 	   1		&	 350&95   &	   14		&	 317&37	\\
        &        & 0.8 &	593    &	2314&57  &	 238    &	813&61  & 	1031		&	3719&71   &	 1017		&	3082&97	\\
        \cdashline{2-15}[0.5pt/3pt]
        &175     & 0.2 &	1    &	 435&75  &	   9    &	1918&21 & 	   1		&	 373&86   &	    1		&	 389&00	\\
        &        & 0.5 &	1    &	 622&35  &	  50    &	1475&44 & 	   1		&	 635&52   &	   55		&	 686&21	\\
        &        & 0.8 & 1288    &	5986&32  &	 156    &	1788&90 & 	1141		&	6958&72   &	  688		&	2365&66	\\
        \cdashline{2-15}[0.5pt/3pt]
        &200     & 0.2 &	1    &	 606&14  &	  21    &	3114&74 & 	   1		&	 840&34   &	    1		&	 869&64	\\
        &        & 0.5 &	1    &	1887&09  &	 154    &	4025&22 & 	   1		&	1387&11   &	   24		&	1540&95	\\
        &        & 0.8 &	21    &	6536&77  &	 402    &	TL\,(1&96\%)&  555		&   TL\,(2&15\%)&	  131		&   TL\,(6&33\%)\\
        \hline                                                                              	        	     		    		
            \end{tabular}}                           
            \caption{Summary of results with BS for CAB and AP instances with up to 200 nodes \label{tab:all}}
            {\footnotesize For the instances reaching the time limit (TL), the value in parenthesis is the percent optimality gap at termination.}
        \end{table}
        
        \begin{table}[htbp]
            \centering
            \scriptsize
            {
            \begin{tabular}{c|c|rr@{.}l|rr@{.}l||rr@{.}l|rr@{.}l}
                &		&	\multicolumn{6}{c}{CAB} & \multicolumn{6}{c}{AP}\\
                &		&	\multicolumn{3}{c|}{SA}	&	 \multicolumn{3}{c||}{MA}&	\multicolumn{3}{c|}{SA}	&	 \multicolumn{3}{c}{MA}\\ [5pt]
                $n$&$\alpha$&   \multicolumn{1}{c}{\#Exp.}	&	\multicolumn{2}{c|}{cpu}&	\multicolumn{1}{c}{\#Exp.}	&	 \multicolumn{2}{c||}{cpu}&	\multicolumn{1}{c}{\#Exp.}	&	\multicolumn{2}{c|}{cpu}&	 \multicolumn{1}{c}{\#Exp.}	&	 \multicolumn{2}{c}{cpu}\\
                [2pt]
                \hline
                25 & 0.2 &    64&      3&05  &   276&      2&32     &   2504&     11&26     &   2742&     13&83\\
                & 0.5 &   934&      5&78  &  1197&      6&17     &  23242&     89&30     &   9502&    131&87\\
                & 0.8 &  9866&     32&43  &  3863&     49&39     &  89224&    392&38     &  41177&   1722&12\\
                50 & 0.2 &  6656&    639&89  &  7144&    403&98     &   5414&    265&92     &   9015&    397&25\\
                & 0.5 & 35947&   1884&64  & 30775&   3308&30     &  57096&   3885&71     &  11249&   4698&33\\
                & 0.8 & 41962&   2863&64  & 23191&   6598&24     &  56575&   7101&86     &  27521&   7166&48\\
                \hline
            \end{tabular}}                           
            \caption{Summary of results for the $G-FB$ for CAB and AP instances with up to 50 nodes \label{tab:capacitated}}
        \end{table}
\newpage
\bibliographystyle{abbrvnat}.
\bibliography{Thesis_bib}

\begin{thebibliography}{64}
\providecommand{\natexlab}[1]{#1}
\providecommand{\url}[1]{\texttt{#1}}
\expandafter\ifx\csname urlstyle\endcsname\relax
  \providecommand{\doi}[1]{doi: #1}\else
  \providecommand{\doi}{doi: \begingroup \urlstyle{rm}\Url}\fi

\bibitem[Ageev et~al.(2001)Ageev, Hassin, and Sviridenko]{Ageev2001}
A.~Ageev, R.~Hassin, and M.~Sviridenko.
\newblock A 0.5-approximation algorithm for max dicut with given sizes of parts.
\newblock \emph{SIAM Journal on Discrete Mathematics}, 14\penalty0 (2):\penalty0 246--255, Jan. 2001.

\bibitem[Alibeyg et~al.(2016)Alibeyg, Contreras, and Fernández]{Alibeyg16}
A.~Alibeyg, I.~Contreras, and E.~Fernández.
\newblock Hub network design problems with profits.
\newblock \emph{Transportation Research Part E: Logistics and Transportation Review}, 96:\penalty0 40--59, Dec. 2016.

\bibitem[Alibeyg et~al.(2018)Alibeyg, Contreras, and Fernández]{Alibeyg18}
A.~Alibeyg, I.~Contreras, and E.~Fernández.
\newblock Exact solution of hub network design problems with profits.
\newblock \emph{European Journal of Operational Research}, 266\penalty0 (1):\penalty0 57--71, Apr. 2018.

\bibitem[Alumur and Kara(2008)]{AlumurKara2008}
S.~A. Alumur and B.~Y. Kara.
\newblock Network hub location problems: The state of the art.
\newblock \emph{European Journal of Operational Research}, 190\penalty0 (1):\penalty0 1--21, Oct. 2008.

\bibitem[Alumur and Kara(2009)]{AlumurKaraCover09}
S.~A. Alumur and B.~Y. Kara.
\newblock A hub covering network design problem for cargo applications in turkey.
\newblock \emph{Journal of the Operational Research Society}, 60\penalty0 (10):\penalty0 1349--1359, Oct. 2009.

\bibitem[Alumur et~al.(2009)Alumur, Kara, and Karasan]{AlumurKaraKarasan09}
S.~A. Alumur, B.~Y. Kara, and O.~E. Karasan.
\newblock The design of single allocation incomplete hub networks.
\newblock \emph{Transportation Research Part B: Methodological}, 43\penalty0 (10):\penalty0 936--951, Dec. 2009.

\bibitem[Alumur et~al.(2021)Alumur, Campbell, Contreras, Kara, Marianov, and O’Kelly]{Alumur2021}
S.~A. Alumur, J.~F. Campbell, I.~Contreras, B.~Y. Kara, V.~Marianov, and M.~E. O’Kelly.
\newblock Perspectives on modeling hub location problems.
\newblock \emph{European Journal of Operational Research}, 291\penalty0 (1):\penalty0 1--17, May 2021.

\bibitem[Aykin(1990)]{Aykin90}
T.~Aykin.
\newblock On “a quadratic integer program for the location of interacting hub facilities”.
\newblock \emph{European Journal of Operational Research}, 46\penalty0 (3):\penalty0 409--411, June 1990.

\bibitem[Barahona et~al.(1988)Barahona, Grötschel, Jünger, and Reinelt]{BGJR88}
F.~Barahona, M.~Grötschel, M.~Jünger, and G.~Reinelt.
\newblock An application of combinatorial optimization to statistical physics and circuit layout design.
\newblock \emph{Operations Research}, 36\penalty0 (3):\penalty0 493--513, June 1988.

\bibitem[Benders(1962)]{Benders62}
J.~F. Benders.
\newblock Partitioning procedures for solving mixed-variables programming problems.
\newblock \emph{Numerische Mathematik}, 4\penalty0 (1):\penalty0 238--252, Dec. 1962.

\bibitem[Boland et~al.(2004)Boland, Krishnamoorthy, Ernst, and Ebery]{Boland04}
N.~Boland, M.~Krishnamoorthy, A.~T. Ernst, and J.~Ebery.
\newblock Preprocessing and cutting for multiple allocation hub location problems.
\newblock \emph{European Journal of Operational Research}, 155\penalty0 (3):\penalty0 638--653, June 2004.
\newblock ISSN 0377-2217.

\bibitem[Campbell(1994)]{Campbell94}
J.~F. Campbell.
\newblock Integer programming formulations of discrete hub location problems.
\newblock \emph{European Journal of Operational Research}, 72\penalty0 (2):\penalty0 387--405, Jan. 1994.

\bibitem[Campbell(1996)]{Campbell96}
J.~F. Campbell.
\newblock Hub location and the p-hub median problem.
\newblock \emph{Operations Research}, 44\penalty0 (6):\penalty0 923--935, Dec. 1996.

\bibitem[Campbell and O’Kelly(2012)]{CampbellOkelly25}
J.~F. Campbell and M.~E. O’Kelly.
\newblock Twenty-five years of hub location research.
\newblock \emph{Transportation Science}, 46\penalty0 (2):\penalty0 153--169, May 2012.

\bibitem[Campbell et~al.(2005{\natexlab{a}})Campbell, Ernst, and Krishnamoorthy]{CampbellHubArcI}
J.~F. Campbell, A.~T. Ernst, and M.~Krishnamoorthy.
\newblock Hub arc location problems: Part i—introduction and results.
\newblock \emph{Management Science}, 51\penalty0 (10):\penalty0 1540--1555, Oct. 2005{\natexlab{a}}.

\bibitem[Campbell et~al.(2005{\natexlab{b}})Campbell, Ernst, and Krishnamoorthy]{CampbellHubArcII}
J.~F. Campbell, A.~T. Ernst, and M.~Krishnamoorthy.
\newblock Hub arc location problems: Part ii—formulations and optimal algorithms.
\newblock \emph{Management Science}, 51\penalty0 (10):\penalty0 1556--1571, Oct. 2005{\natexlab{b}}.

\bibitem[Contreras and Fernández(2012)]{ContrerasFernandez12}
I.~Contreras and E.~Fernández.
\newblock General network design: A unified view of combined location and network design problems.
\newblock \emph{European Journal of Operational Research}, 219\penalty0 (3):\penalty0 680--697, June 2012.

\bibitem[Contreras and Fernández(2014)]{ContrerasFernandez14}
I.~Contreras and E.~Fernández.
\newblock Hub location as the minimization of a supermodular set function.
\newblock \emph{Operations Research}, 62\penalty0 (3):\penalty0 557--570, June 2014.

\bibitem[Contreras and O’Kelly(2019)]{ContrerasOkelly}
I.~Contreras and M.~O’Kelly.
\newblock Hub location problems.
\newblock In G.~Laporte, S.~Nickel, and e.~Saldanha~da Gama, F., editors, \emph{Location Science}, pages 327--363. Springer, Berlin Heidelberg, 2019.

\bibitem[Contreras et~al.(2009)Contreras, Díaz, and Fernández]{ContrerasDiaz09}
I.~Contreras, J.~A. Díaz, and E.~Fernández.
\newblock Lagrangean relaxation for the capacitated hub location problem with single assignment.
\newblock \emph{OR Spectrum}, 31\penalty0 (3):\penalty0 483--505, Jan. 2009.

\bibitem[Contreras et~al.(2010)Contreras, Fernández, and Marín]{ContrerasTreeofHubs}
I.~Contreras, E.~Fernández, and A.~Marín.
\newblock The tree of hubs location problem.
\newblock \emph{European Journal of Operational Research}, 202\penalty0 (2):\penalty0 390--400, Apr. 2010.

\bibitem[Contreras et~al.(2011{\natexlab{a}})Contreras, Cordeau, and Laporte]{Contrerasetal11}
I.~Contreras, J.-F. Cordeau, and G.~Laporte.
\newblock Benders decomposition for large-scale uncapacitated hub location.
\newblock \emph{Operations Research}, 59\penalty0 (6):\penalty0 1477--1490, Dec. 2011{\natexlab{a}}.

\bibitem[Contreras et~al.(2011{\natexlab{b}})Contreras, Díaz, and Fernández]{ContrerasDiaz11}
I.~Contreras, J.~A. Díaz, and E.~Fernández.
\newblock Branch and price for large-scale capacitated hub location problems with single assignment.
\newblock \emph{INFORMS Journal on Computing}, 23\penalty0 (1):\penalty0 41--55, Feb. 2011{\natexlab{b}}.

\bibitem[Cordeau et~al.(2019)Cordeau, Furini, and Ljubić]{CORDEAU2019}
J.-F. Cordeau, F.~Furini, and I.~Ljubić.
\newblock Benders decomposition for very large scale partial set covering and maximal covering location problems.
\newblock \emph{European Journal of Operational Research}, 275\penalty0 (3):\penalty0 882--896, June 2019.

\bibitem[de~Camargo et~al.(2008)de~Camargo, Miranda, and Luna]{Camargo2008}
R.~S. de~Camargo, G.~Miranda, and H.~Luna.
\newblock Benders decomposition for the uncapacitated multiple allocation hub location problem.
\newblock \emph{Computers \& Operations Research}, 35\penalty0 (4):\penalty0 1047--1064, Apr. 2008.

\bibitem[de~Camargo et~al.(2009)de~Camargo, de~Miranda, and Luna]{deCamargo2009}
R.~S. de~Camargo, G.~de~Miranda, and H.~P.~L. Luna.
\newblock Benders decomposition for hub location problems with economies of scale.
\newblock \emph{Transportation Science}, 43\penalty0 (1):\penalty0 86--97, Feb. 2009.
\newblock ISSN 1526-5447.

\bibitem[de~Camargo et~al.(2017)de~Camargo, de~Miranda, O’Kelly, and Campbell]{deCamargo2017}
R.~S. de~Camargo, G.~de~Miranda, M.~E. O’Kelly, and J.~F. Campbell.
\newblock Formulations and decomposition methods for the incomplete hub location network design problem with and without hop-constraints.
\newblock \emph{Applied Mathematical Modelling}, 51:\penalty0 274--301, Nov. 2017.

\bibitem[Dijkstra(1959)]{Dijkstra}
E.~W. Dijkstra.
\newblock A note on two problems in connexion with graphs.
\newblock \emph{Numerische Mathematik}, 1\penalty0 (1):\penalty0 269--271, Dec. 1959.

\bibitem[Domínguez-Bravo et~al.(2024)Domínguez-Bravo, Fernández, and Lüer-Villagra]{Dominguez24}
C.-A. Domínguez-Bravo, E.~Fernández, and A.~Lüer-Villagra.
\newblock Hub location with congestion and time-sensitive demand.
\newblock \emph{European Journal of Operational Research}, 316\penalty0 (3):\penalty0 828--844, Aug. 2024.

\bibitem[Ebery et~al.(2000)Ebery, Krishnamoorthy, Ernst, and Boland]{Ebery00}
J.~Ebery, M.~Krishnamoorthy, A.~Ernst, and N.~Boland.
\newblock The capacitated multiple allocation hub location problem: Formulations and algorithms.
\newblock \emph{European Journal of Operational Research}, 120\penalty0 (3):\penalty0 614--631, Feb. 2000.

\bibitem[Elhedhli and Wu(2010)]{ElhedhliWu10}
S.~Elhedhli and H.~Wu.
\newblock A lagrangean heuristic for hub-and-spoke system design with capacity selection and congestion.
\newblock \emph{INFORMS Journal on Computing}, 22\penalty0 (2):\penalty0 282--296, May 2010.

\bibitem[Ernst and Krishnamoorthy(1996)]{Ernst96}
A.~T. Ernst and M.~Krishnamoorthy.
\newblock Efficient algorithms for the uncapacitated single allocation p-hub median problem.
\newblock \emph{Location Science}, 4\penalty0 (3):\penalty0 139--154, Oct. 1996.

\bibitem[Ernst and Krishnamoorthy(1998)]{Ernst98}
A.~T. Ernst and M.~Krishnamoorthy.
\newblock Exact and heuristic algorithms for the uncapacitated multiple allocation p-hub median problem.
\newblock \emph{European Journal of Operational Research}, 104\penalty0 (1):\penalty0 100--112, Jan. 1998.

\bibitem[Ernst et~al.(2009)Ernst, Hamacher, Jiang, Krishnamoorthy, and Woeginger]{ErnstCenter09}
A.~T. Ernst, H.~Hamacher, H.~Jiang, M.~Krishnamoorthy, and G.~Woeginger.
\newblock Uncapacitated single and multiple allocation p-hub center problems.
\newblock \emph{Computers \& Operations Research}, 36\penalty0 (7):\penalty0 2230--2241, July 2009.

\bibitem[Ernst et~al.(2017)Ernst, Jiang, Krishanmoorthy, and Baatar]{ErnstCover18}
A.~T. Ernst, H.~Jiang, M.~Krishanmoorthy, and D.~Baatar.
\newblock \emph{Reformulations and Computational Results for the Uncapacitated Single Allocation Hub Covering Problem}, pages 133--148.
\newblock Springer International Publishing, July 2017.

\bibitem[Espejo et~al.(2023)Espejo, Marín, Muñoz-Ocaña, and Rodríguez-Chía]{Espejoetal23}
I.~Espejo, A.~Marín, J.~M. Muñoz-Ocaña, and A.~M. Rodríguez-Chía.
\newblock A new formulation and branch-and-cut method for single-allocation hub location problems.
\newblock \emph{Computers \& Operations Research}, 155:\penalty0 106241, July 2023.

\bibitem[Fischetti et~al.(2017)Fischetti, Ljubić, and Sinnl]{FischettiLjubicSinnl17}
M.~Fischetti, I.~Ljubić, and M.~Sinnl.
\newblock Redesigning benders decomposition for large-scale facility location.
\newblock \emph{Management Science}, 63\penalty0 (7):\penalty0 2146--2162, July 2017.

\bibitem[García et~al.(2012)García, Landete, and Marín]{Garcia2}
S.~García, M.~Landete, and A.~Marín.
\newblock New formulation and a branch-and-cut algorithm for the multiple allocation p-hub median problem.
\newblock \emph{European Journal of Operational Research}, 220\penalty0 (1):\penalty0 48--57, July 2012.

\bibitem[Ghaffarinasab and Kara(2018)]{Ghaffarinasab2018}
N.~Ghaffarinasab and B.~Y. Kara.
\newblock Benders decomposition algorithms for two variants of the single allocation hub location problem.
\newblock \emph{Networks and Spatial Economics}, 19\penalty0 (1):\penalty0 83--108, Oct. 2018.

\bibitem[Gusfield(1990)]{Gusfield}
D.~Gusfield.
\newblock Very simple methods for all pairs network flow analysis.
\newblock \emph{SIAM Journal on Computing}, 19\penalty0 (1):\penalty0 143--155, Feb. 1990.

\bibitem[Hamacher et~al.(2004)Hamacher, Labbé, Nickel, and Sonneborn]{Hamacher04}
H.~W. Hamacher, M.~Labbé, S.~Nickel, and T.~Sonneborn.
\newblock Adapting polyhedral properties from facility to hub location problems.
\newblock \emph{Discrete Applied Mathematics}, 145\penalty0 (1):\penalty0 104--116, Dec. 2004.

\bibitem[Jünger and Mallach(2021)]{JuengerMallach}
M.~Jünger and S.~Mallach.
\newblock Exact facetial odd-cycle separation for maximum cut and binary quadratic optimization.
\newblock \emph{INFORMS Journal on Computing}, Feb. 2021.

\bibitem[Kimms(2006)]{Kimms06}
A.~Kimms.
\newblock \emph{Economies of Scale in Hub \& Spoke Network Design Models: We Have It All Wrong}, pages 293--317.
\newblock DUV, 2006.

\bibitem[Labbé and Yaman(2004)]{LabbeYaman2004}
M.~Labbé and H.~Yaman.
\newblock Projecting the flow variables for hub location problems.
\newblock \emph{Networks}, 44\penalty0 (2):\penalty0 84--93, July 2004.
\newblock ISSN 1097-0037.

\bibitem[Labbé et~al.(2004)Labbé, Laporte, Martín, and González]{LabbeRingStar}
M.~Labbé, G.~Laporte, I.~R. Martín, and J.~J.~S. González.
\newblock The ring star problem: Polyhedral analysis and exact algorithm.
\newblock \emph{Networks}, 43\penalty0 (3):\penalty0 177--189, Mar. 2004.

\bibitem[Lüer-Villagra et~al.(2019)Lüer-Villagra, Eiselt, and Marianov]{LuerMarianov}
A.~Lüer-Villagra, H.~Eiselt, and V.~Marianov.
\newblock A single allocation p-hub median problem with general piecewise-linear costs in arcs.
\newblock \emph{Computers \& Industrial Engineering}, 128:\penalty0 477--491, Feb. 2019.

\bibitem[Martins~de Sá et~al.(2018)Martins~de Sá, Morabito, and de~Camargo]{deSa18}
E.~Martins~de Sá, R.~Morabito, and R.~S. de~Camargo.
\newblock Benders decomposition applied to a robust multiple allocation incomplete hub location problem.
\newblock \emph{Computers \& Operations Research}, 89:\penalty0 31--50, Jan. 2018.

\bibitem[Marín(2005)]{Marin05}
A.~Marín.
\newblock Uncapacitated euclidean hub location: Strengthened formulation, new facets and a relax-and-cut algorithm.
\newblock \emph{Journal of Global Optimization}, 33\penalty0 (3):\penalty0 393--422, Nov. 2005.

\bibitem[Marín et~al.(2006)Marín, Cánovas, and Landete]{Marin-et-al06}
A.~Marín, L.~Cánovas, and M.~Landete.
\newblock New formulations for the uncapacitated multiple allocation hub location problem.
\newblock \emph{European Journal of Operational Research}, 172\penalty0 (1):\penalty0 274--292, July 2006.

\bibitem[Meier and Clausen(2018)]{Meier2018}
J.~F. Meier and U.~Clausen.
\newblock Solving single allocation hub location problems on euclidean data.
\newblock \emph{Transportation Science}, 52\penalty0 (5):\penalty0 1141--1155, Oct. 2018.

\bibitem[O’Kelly and Bryan(1998)]{OkellyBryan}
M.~O’Kelly and D.~Bryan.
\newblock Hub location with flow economies of scale.
\newblock \emph{Transportation Research Part B: Methodological}, 32\penalty0 (8):\penalty0 605--616, Nov. 1998.

\bibitem[O’kelly(1986)]{Okelly86}
M.~E. O’kelly.
\newblock Activity levels at hub facilities in interacting networks.
\newblock \emph{Geographical Analysis}, 18\penalty0 (4):\penalty0 343--356, Oct. 1986.

\bibitem[O’kelly(1987)]{Okelly87}
M.~E. O’kelly.
\newblock A quadratic integer program for the location of interacting hub facilities.
\newblock \emph{European Journal of Operational Research}, 32\penalty0 (3):\penalty0 393--404, Dec. 1987.

\bibitem[O’Kelly(1992)]{Okelly92}
M.~E. O’Kelly.
\newblock Hub facility location with fixed costs.
\newblock \emph{Papers in Regional Science}, 71\penalty0 (3):\penalty0 293--306, July 1992.

\bibitem[O’Kelly et~al.(2014)O’Kelly, Campbell, de~Camargo, and de~Miranda]{Okellyetal15}
M.~E. O’Kelly, J.~F. Campbell, R.~S. de~Camargo, and G.~de~Miranda.
\newblock Multiple allocation hub location model with fixed arc costs.
\newblock \emph{Geographical Analysis}, 47\penalty0 (1):\penalty0 73--96, Sept. 2014.

\bibitem[Pirkul and Schilling(1998)]{PirkulShilling98}
H.~Pirkul and D.~A. Schilling.
\newblock An efficient procedure for designing single allocation hub and spoke systems.
\newblock \emph{Management Science}, 44\penalty0 (12-part-2):\penalty0 S235--S242, Dec. 1998.

\bibitem[Puerto et~al.(2011)Puerto, Ramos, and Rodríguez-Chía]{PuertoOM11}
J.~Puerto, A.~Ramos, and A.~Rodríguez-Chía.
\newblock Single-allocation ordered median hub location problems.
\newblock \emph{Computers \& Operations Research}, 38\penalty0 (2):\penalty0 559--570, Feb. 2011.

\bibitem[Puerto et~al.(2013)Puerto, Ramos, and Rodríguez-Chía]{PuertoOM13}
J.~Puerto, A.~Ramos, and A.~Rodríguez-Chía.
\newblock A specialized branch \& bound \& cut for single-allocation ordered median hub location problems.
\newblock \emph{Discrete Applied Mathematics}, 161\penalty0 (16–17):\penalty0 2624--2646, Nov. 2013.

\bibitem[Skorin-Kapov et~al.(1996)Skorin-Kapov, Skorin-Kapov, and O’Kelly]{Skorin96}
D.~Skorin-Kapov, J.~Skorin-Kapov, and M.~O’Kelly.
\newblock Tight linear programming relaxations of uncapacitated p-hub median problems.
\newblock \emph{European Journal of Operational Research}, 94\penalty0 (3):\penalty0 582--593, Nov. 1996.

\bibitem[Taherkhani and Alumur(2019)]{TaherkhaniAlumur19}
G.~Taherkhani and S.~A. Alumur.
\newblock Profit maximizing hub location problems.
\newblock \emph{Omega}, 86:\penalty0 1--15, July 2019.

\bibitem[Taherkhani et~al.(2020)Taherkhani, Alumur, and Hosseini]{Taherkhanietal20}
G.~Taherkhani, S.~A. Alumur, and M.~Hosseini.
\newblock Benders decomposition for the profit maximizing capacitated hub location problem with multiple demand classes.
\newblock \emph{Transportation Science}, 54\penalty0 (6):\penalty0 1446--1470, Nov. 2020.

\bibitem[Thorsteinsson(2001)]{Thorsteinsson21}
E.~S. Thorsteinsson.
\newblock \emph{Branch-and-Check: A Hybrid Framework Integrating Mixed Integer Programming and Constraint Logic Programming}, pages 16--30.
\newblock Springer Berlin Heidelberg, 2001.

\bibitem[Wandelt et~al.(2022)Wandelt, Dai, Zhang, and Sun]{Wandelt22}
S.~Wandelt, W.~Dai, J.~Zhang, and X.~Sun.
\newblock Toward a reference experimental benchmark for solving hub location problems.
\newblock \emph{Transportation Science}, 56\penalty0 (2):\penalty0 543--564, Mar. 2022.

\bibitem[Yaman et~al.(2007)Yaman, Kara, and Tansel]{Yaman07}
H.~Yaman, B.~Y. Kara, and B.~c. Tansel.
\newblock The latest arrival hub location problem for cargo delivery systems with stopovers.
\newblock \emph{Transportation Research Part B: Methodological}, 41\penalty0 (8):\penalty0 906--919, Oct. 2007.

\end{thebibliography}

\newpage
\renewcommand{\thesection}{EC-\arabic{section}}
\setcounter{section}{0}
\setcounter{figure}{0}
\setcounter{table}{0}
\setcounter{equation}{0}
\renewcommand{\thefigure}{EC~\arabic{figure}}
\renewcommand{\thetable}{EC~\arabic{table}}
\renewcommand{\theequation}{EC~\arabic{equation}}

\section*{Electronic companion}
    This electronic companion provides supplementary material to support and complement the main manuscript. It includes additional figures and detailed tables that document the optimization process used in the computational experiments. While these materials are not essential to the core understanding of the main paper, they offer valuable insights and further context that may be of interest to the reader. Their inclusion here is intended to enhance transparency and reproducibility without overloading the main text.

\section{Formulation for the exact separation of constraints (2d)} \label{appendix:aggdmnd_form}
	Below we provide a formulation, based on that of Ageev, et.al. (2001), which produces a cutset of maximum value relative to arc capacities $-\overline Q$. We consider the following sets of decision variables:
	\begin{itemize}
		\item $\alpha_{i}\in\{0, 1\}$, $i\in V$. $\alpha_{i}=1$ if and only if node $i\in S$.
		\item $\beta_{ij}\in\{0, 1\}$, $(i,j)\in A$.  $\beta_{ij}=1$ if and only if $i\in S$ and $j\notin S$ (i.e., arc $(i,j)$ is in $\delta^+(S)$).
	\end{itemize}
	
	Then the formulation is:
	\begin{subequations}
		\begin{align}
			F_{SEP}\qquad  \max & \sum_{(i,j) \in A} (-\overline Q_{ij}) \beta_{ij} && \label{of}\\
			\text{s.t.}\qquad &  \beta_{ij}\leq \alpha_i \qquad && (i, j)\in A \label{C1}\\
			& \beta_{ij}\leq 1-\alpha_j \qquad && (i, j)\in A \label{C2}\\
			& \alpha_i \leq \alpha_j + \beta_{ij} && (i, j)\in A \label{C3}\\
			&  \alpha_{i} \in \left\{0,1\right\} \quad i\in V&&\label{domain_x}\\
			& \beta_{ij} \in \left\{0,1\right\}, \, (i, j)\in A.\label{domainz}
		\end{align}
	\end{subequations}
	
	Constraints \eqref{C1}-\eqref{C2} determine a bi-partition of the node set, $(S, S^c)$, induced by $\alpha$, together with a subset of arcs in $\delta^+(S)$, induced by $\beta$. Still, Constraints \eqref{C1}-\eqref{C2} alone, do not guarantee that the subset of arcs induced by $\beta$ coincides with $\delta^+(S)$. Specifically, those arcs with an objective function coefficient $\overline Q_{ij}<0$ will not be activated. Hence, the set of constraints  \eqref{C3} is needed in order to guarantee that all arcs in $\delta^+(S)$ are activated.

\section{Complementary figures and tables}\label{sec:Figs}
    In this section we present complementary figures and tables that may be of interest to the reader.
    \subsection{Complementary material for Section 8.2.}
        Figures \ref{fig-lp-H} and \ref{fig-lp-G} represent the comparison of linear relaxations between the classical 4- and 3-index formulations (green and blue bars, respectively) and our formulation (orange bar) on instances of up to 100 nodes. Instances are considered small for $n \in \{25, 50\}$ and medium for $n \in \{75, 100\}$.
    
        \begin{figure}[H]
        	\centering
        	{\includegraphics[width=.95\textwidth]{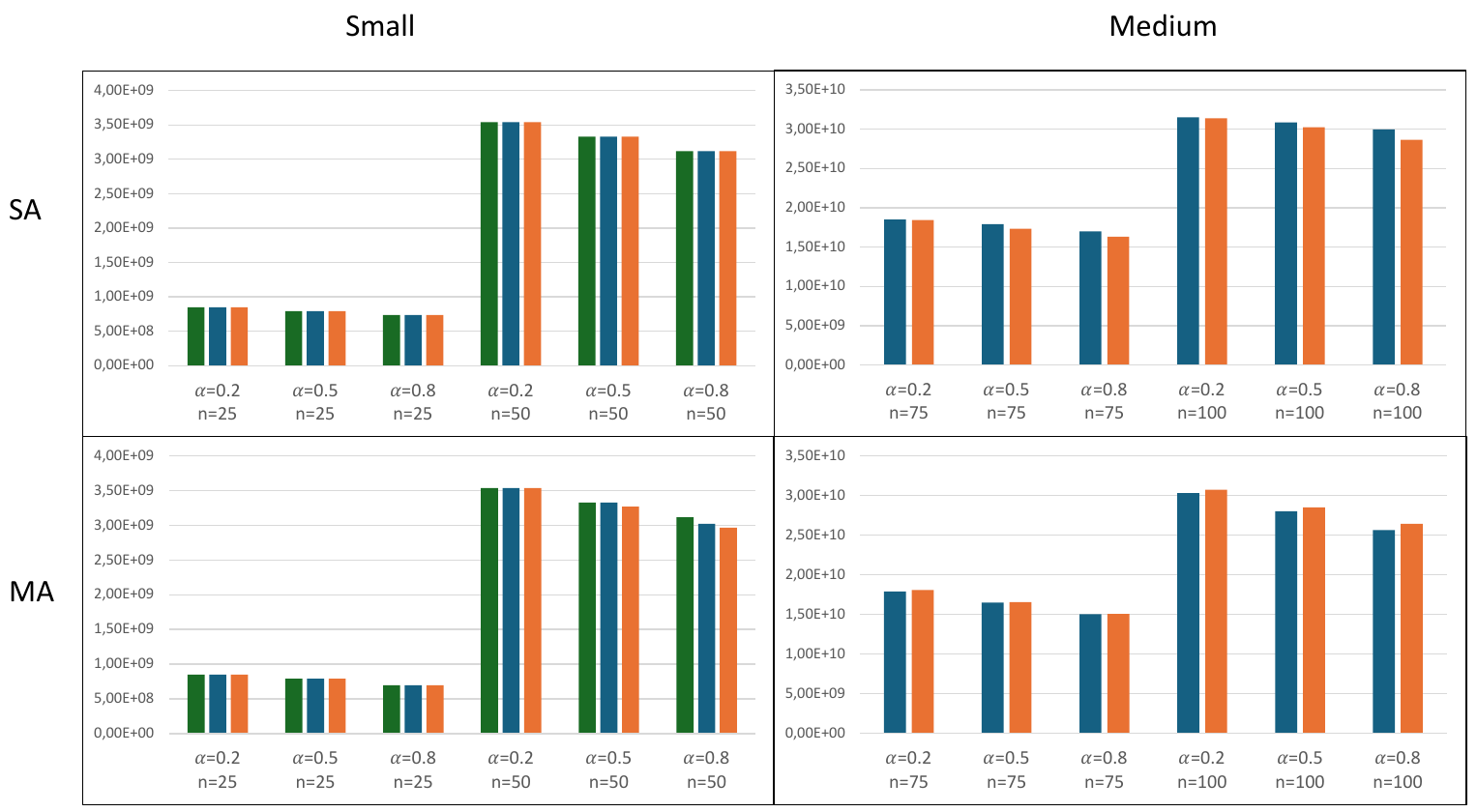}}
        	\caption{Comparison of LP bounds for the $H$-median. \label{fig-lp-H}}
        	{\footnotesize Green: 4-index; blue: 3-index; orange: 2-index}
    	\end{figure}
    
    	\begin{figure}[H]
            \centering
        	{\includegraphics[width=.95\textwidth]{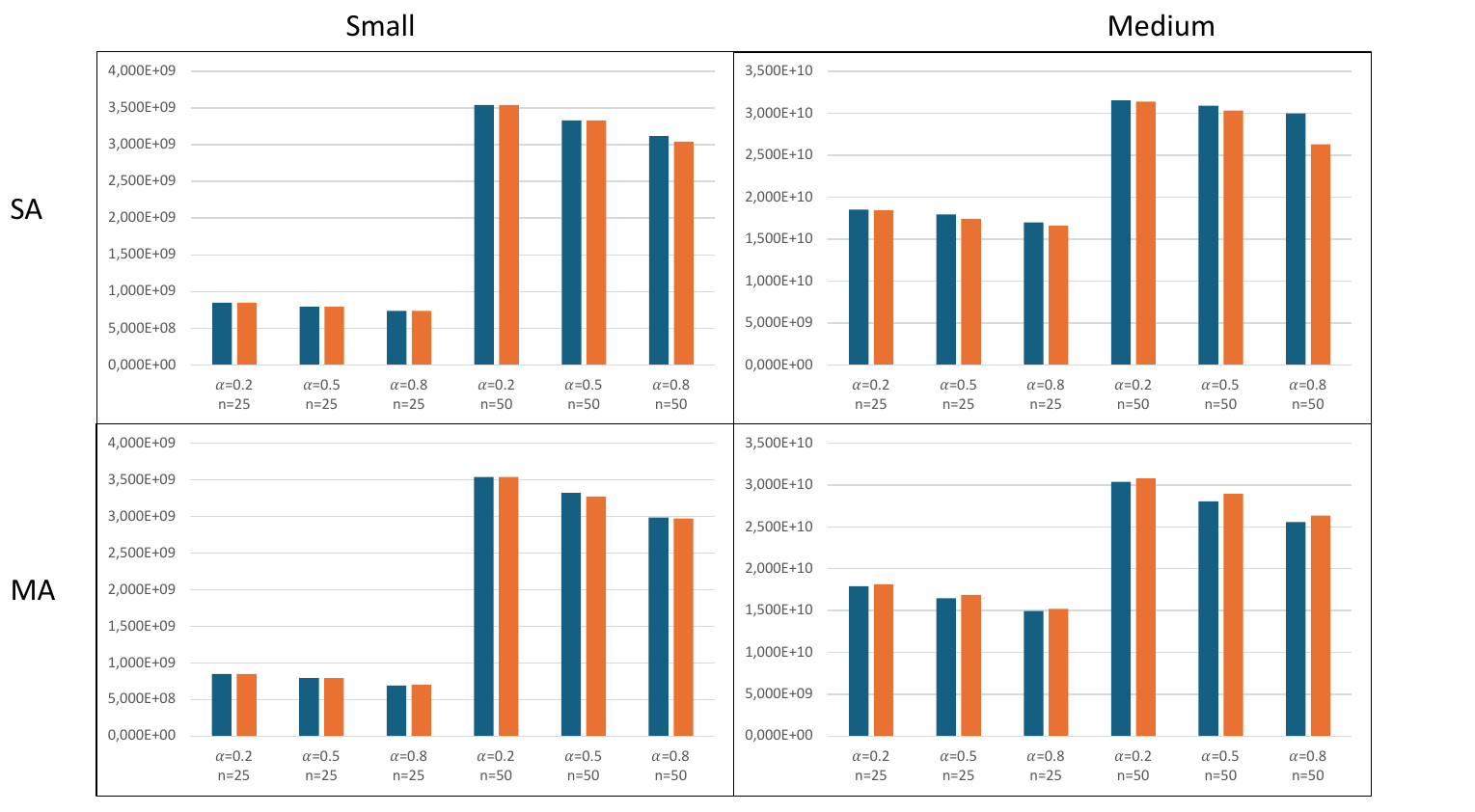}}
        	\caption{Comparison of LP bounds for $G$-median. Blue: 3-index; orange: 2-index\label{fig-lp-G}}
    	\end{figure}

    \subsection{Complementary material for Section 8.3.}
        Tables \ref{tab:SA_Hmed_CAB} to \ref{tab:MA_Gmed_AP} report optimization details for the H-median and G-median problems under their respective allocation policies. Each table corresponds to a specific dataset (``CAB'' or ``AP'') and allocation policy (``SA'' or ``MA''). The structure is the same across tables: the first two columns report instance size and discount factor, respectively. Column ``CPU (s)'' gives the total runtime, either until optimality was reached or the time limit (set to two hours). Column ``$\#Exp$'' shows the number of branching nodes explored; a value of one indicates that only the root node was explored. The column ``GAP (\%)'' reports the final optimality gap. Columns ``UB'' and ``LB'' provide the best upper and lower bounds found, respectively. With respect to the root node, column ``LB$_{root}$'' reports the LP relaxation value at root, which serves as an indicator of its quality relative to the optimal solution, and column ``CPU$_{root}$'' shows the runtime spent at root (in seconds). The last column, ``CPU$_{SP}(\%)$,'' reports the percentage of computing time, relative to the total computing time devoted to solving the shortest path algorithm, expressed as a percentage to enhance readability.

        \begin{table}[H]
            \centering    
            {
        \begin{tabular}{cc|rrrrrrrr}
        $n$ & $\alpha$ & \multicolumn{1}{c}{CPU (s)} & \multicolumn{1}{c}{$\#Exp$} & \multicolumn{1}{c}{GAP (\%)} & \multicolumn{1}{c}{UB} & \multicolumn{1}{c}{LB} & \multicolumn{1}{c}{LB$_{root}$} & \multicolumn{1}{c}{CPU$_{root}$} &  \multicolumn{1}{c}{CPU$_{SP} (\%)$} \\ \hline\hline 
        25 & $0.2$ & 0.08 & 1 & 0.00 & 8.4902$\times 10^{8}$ & 8.4902$\times 10^{8}$ & 8.4902$\times 10^{8}$ & 0.07 & 50.00 \\
         & $0.5$ & 0.10 & 1 & 0.00 & 7.9312$\times 10^{8}$ & 7.9312$\times 10^{8}$ & 7.9312$\times 10^{8}$ & 0.10 & 40.00 \\
         & $0.8$ & 0.10 & 1 & 0.00 & 7.3722$\times 10^{8}$ & 7.3722$\times 10^{8}$ & 7.3722$\times 10^{8}$ & 0.09 & 70.00 \\
        50 & $0.2$ & 0.67 & 1 & 0.00 & 3.5404$\times 10^{9}$ & 3.5404$\times 10^{9}$ & 3.5404$\times 10^{9}$ & 0.66 & 31.34 \\
         & $0.5$ & 1.02 & 1 & 0.00 & 3.4456$\times 10^{9}$ & 3.4456$\times 10^{9}$ & 3.4456$\times 10^{9}$ & 1.02 & 24.51 \\
         & $0.8$ & 1.40 & 1 & 0.00 & 3.1321$\times 10^{9}$ & 3.1321$\times 10^{9}$ & 3.1321$\times 10^{9}$ & 1.40 & 34.29 \\
        75 & $0.2$ & 145.18 & 7 & 0.00 & 1.8542$\times 10^{10}$ & 1.8542$\times 10^{10}$ & 1.8461$\times 10^{10}$ & 130.25 & 26.17 \\
         & $0.5$ & 177.49 & 439 & 0.00 & 1.8090$\times 10^{10}$ & 1.8090$\times 10^{10}$ & 1.7348$\times 10^{10}$ & 49.75 & 32.51 \\
         & $0.8$ & 342.91 & 647 & 0.00 & 1.7019$\times 10^{10}$ & 1.7019$\times 10^{10}$ & 1.6335$\times 10^{10}$ & 67.92 & 21.41 \\
        100 & $0.2$ & 200.51 & 19 & 0.00 & 3.1552$\times 10^{10}$ & 3.1552$\times 10^{10}$ & 3.1408$\times 10^{10}$ & 154.29 & 18.48 \\
         & $0.5$ & 912.07 & 1370 & 0.00 & 3.1106$\times 10^{10}$ & 3.1106$\times 10^{10}$ & 3.0260$\times 10^{10}$ & 286.57 & 17.95 \\
         & $0.8$ & 1717.02 & 5333 & 0.00 & 3.0408$\times 10^{10}$ & 3.0408$\times 10^{10}$ & 2.8641$\times 10^{10}$ & 264.91 & 26.67
        \end{tabular}}
            \caption{Results for the SA $H$-median problem for CAB dataset.\label{tab:SA_Hmed_CAB}}
        \end{table}

        \begin{table}[H]
            \centering
            {\begin{tabular}{cc|rrrrrrrr}
            $n$ & $\alpha$ & \multicolumn{1}{c}{CPU (s)} & \multicolumn{1}{c}{$\#Exp$} & \multicolumn{1}{l}{GAP (\%)} & \multicolumn{1}{c}{UB} & \multicolumn{1}{c}{LB} & \multicolumn{1}{c}{LB$_{root}$} & \multicolumn{1}{c}{CPU$_{root}$} &  \multicolumn{1}{c}{CPU$_{SP} (\%)$} \\ \hline\hline 
            25 & $0.2$ & 0.18 & 1 & 0.00 & 8.4902$\times 10^{8}$ & 8.4902$\times 10^{8}$ & 8.4902$\times 10^{8}$ & 0.18 & 16.67 \\
             & $0.5$ & 0.92 & 3 & 0.00 & 7.9312$\times 10^{8}$ & 7.9312$\times 10^{8}$ & 7.9179$\times 10^{8}$ & 0.45 & 3.26 \\
             & $0.8$ & 2.52 & 73 & 0.00 & 7.3722$\times 10^{8}$ & 7.3722$\times 10^{8}$ & 6.9768$\times 10^{8}$ & 0.90 & 2.78 \\
            50 & $0.2$ & 1.17 & 1 & 0.00 & 3.5404$\times 10^{9}$ & 3.5404$\times 10^{9}$ & 3.5404$\times 10^{9}$ & 1.17 & 14.53 \\
             & $0.5$ & 24.41 & 214 & 0.00 & 3.3297$\times 10^{9}$ & 3.3297$\times 10^{9}$ & 3.2693$\times 10^{9}$ & 3.37 & 13.27 \\
             & $0.8$ & 43.76 & 213 & 0.00 & 3.1190$\times 10^{9}$ & 3.1190$\times 10^{9}$ & 2.9661$\times 10^{9}$ & 4.72 & 11.52 \\
            75 & $0.2$ & 28.64 & 7 & 0.00 & 1.8468$\times 10^{10}$ & 1.8468$\times 10^{10}$ & 1.8094$\times 10^{10}$ & 26.31 & 12.05 \\
             & $0.5$ & 69.07 & 247 & 0.00 & 1.7603$\times 10^{10}$ & 1.7603$\times 10^{10}$ & 1.6574$\times 10^{10}$ & 17.71 & 10.05 \\
             & $0.8$ & 527.34 & 1917 & 0.00 & 1.6473$\times 10^{10}$ & 1.6473$\times 10^{10}$ & 1.5074$\times 10^{10}$ & 28.73 & 15.23 \\
            100 & $0.2$ & 185.63 & 30 & 0.00 & 3.1459$\times 10^{10}$ & 3.1459$\times 10^{10}$ & 3.0747$\times 10^{10}$ & 135.41 & 19.61 \\
             & $0.5$ & 411.55 & 877 & 0.00 & 3.0261$\times 10^{10}$ & 3.0261$\times 10^{10}$ & 2.8515$\times 10^{10}$ & 82.15 & 22.69 \\
             & $0.8$ & 690.23 & 2336 & 0.00 & 2.8462$\times 10^{10}$ & 2.8462$\times 10^{10}$ & 2.6460$\times 10^{10}$ & 115.51 & 30.60
            \end{tabular}}
            \caption{Results for the MA $H$-median problem for CAB dataset. \label{tab:MA_Hmed_CAB}}
        \end{table}

        \begin{table}[H]
            \centering
            {\begin{tabular}{cc|rrrrrrrr}
            $n$ & $\alpha$ & \multicolumn{1}{c}{CPU (s)} & \multicolumn{1}{c}{$\#Exp$} & \multicolumn{1}{c}{GAP (\%)} & \multicolumn{1}{c}{UB} & \multicolumn{1}{c}{LB} & \multicolumn{1}{c}{LB$_{root}$} & \multicolumn{1}{c}{CPU$_{root}$} & \multicolumn{1}{c}{CPU$_{SP} (\%)$} \\ \hline\hline 
            25 & $0.2$ & 0.08 & 1 & 0.00 & 8.4912$\times 10^{8}$ & 8.4912$\times 10^{8}$ & 8.4912$\times 10^{8}$ & 0.08 & 37.50 \\
             & $0.5$ & 0.09 & 1 & 0.00 & 7.9315$\times 10^{8}$ & 7.9315$\times 10^{8}$ & 7.9315$\times 10^{8}$ & 0.09 & 33.33 \\
             & $0.8$ & 0.14 & 1 & 0.00 & 7.3723$\times 10^{8}$ & 7.3723$\times 10^{8}$ & 7.3723$\times 10^{8}$ & 0.14 & 50.00 \\
            50 & $0.2$ & 0.96 & 1 & 0.00 & 3.5404$\times 10^{9}$ & 3.5404$\times 10^{9}$ & 3.5404$\times 10^{9}$ & 0.95 & 17.71 \\
             & $0.5$ & 1.27 & 1 & 0.00 & 3.3297$\times 10^{9}$ & 3.3299$\times 10^{9}$ & 3.3299$\times 10^{9}$ & 1.27 & 21.26 \\
             & $0.8$ & 7.45 & 1 & 0.00 & 3.1194$\times 10^{9}$ & 3.1194$\times 10^{9}$ & 3.1194$\times 10^{9}$ & 6.78 & 5.64 \\
            75 & $0.2$ & 79.07 & 9 & 0.00 & 1.8546$\times 10^{10}$ & 1.8546$\times 10^{10}$ & 1.8462$\times 10^{10}$ & 71.18 & 4.36 \\
             & $0.5$ & 205.65 & 499 & 0.00 & 1.8094$\times 10^{10}$ & 1.8094$\times 10^{10}$ & 1.7443$\times 10^{10}$ & 72.27 & 13.50 \\
             & $0.8$ & 297.63 & 250 & 0.00 & 1.7023$\times 10^{10}$ & 1.7023$\times 10^{10}$ & 1.7023$\times 10^{10}$ & 82.10 & 26.98 \\
            100 & $0.2$ & 500.79 & 17 & 0.00 & 3.1556$\times 10^{10}$ & 3.1556$\times 10^{10}$ & 3.1411$\times 10^{10}$ & 472.12 & 7.27 \\
             & $0.5$ & 1914.37 & 1704 & 0.00 & 3.1109$\times 10^{10}$ & 3.1109$\times 10^{10}$ & 3.0347$\times 10^{10}$ & 479.60 & 39.03 \\
             & $0.8$ & 2955.68 & 4574 & 0.00 & 3.0411$\times 10^{10}$ & 3.0411$\times 10^{10}$ & 2.6415$\times 10^{10}$ & 382.74 & 26.35
            \end{tabular}}
            \caption{Results for the SA $G$-median problem for CAB dataset. \label{tab:SA_Gmed_CAB}}
        \end{table}

        \begin{table}[H]
            \centering
            {\begin{tabular}{cc|rrrrrrrr}
            $n$ & $\alpha$ & \multicolumn{1}{c}{CPU (s)} & \multicolumn{1}{c}{$\#Exp$} & \multicolumn{1}{c}{GAP (\%)} & \multicolumn{1}{c}{UB} & \multicolumn{1}{c}{LB} & \multicolumn{1}{c}{LB$_{root}$} & \multicolumn{1}{c}{CPU$_{root}$} & \multicolumn{1}{c}{CPU$_{SP} (\%)$} \\ \hline\hline 
            25 & $0.2$ & 0.07 & 1 & 0.00 & 8.4907$\times 10^{8}$ & 8.4907$\times 10^{8}$ & 8.4902$\times 10^{8}$ & 0.07 & 57.14 \\
             & $0.5$ & 0.62 & 3 & 0.00 & 7.9316$\times 10^{8}$ & 7.9316$\times 10^{8}$ & 7.9216$\times 10^{8}$ & 0.39 & 22.58 \\
             & $0.8$ & 1.82 & 64 & 0.00 & 7.3728$\times 10^{8}$ & 7.3728$\times 10^{8}$ & 6.9891$\times 10^{8}$ & 0.65 & 18.13 \\
            50 & $0.2$ & 0.78 & 1 & 0.00 & 3.5409$\times 10^{9}$ & 3.5409$\times 10^{9}$ & 3.5404$\times 10^{9}$ & 0.78 & 21.79 \\
             & $0.5$ & 9.02 & 30 & 0.00 & 3.3301$\times 10^{9}$ & 3.3301$\times 10^{9}$ & 3.2698$\times 10^{9}$ & 2.96 & 8.87 \\
             & $0.8$ & 21.44 & 185 & 0.00 & 3.1195$\times 10^{9}$ & 3.1195$\times 10^{9}$ & 2.9668$\times 10^{9}$ & 3.93 & 5.55 \\
            75 & $0.2$ & 33.40 & 17 & 0.00 & 1.8472$\times 10^{10}$ & 1.8472$\times 10^{10}$ & 1.8308$\times 10^{10}$ & 26.41 & 11.17 \\
             & $0.5$ & 133.96 & 595 & 0.00 & 1.7607$\times 10^{10}$ & 1.7607$\times 10^{10}$ & 1.6862$\times 10^{10}$ & 31.28 & 17.11 \\
             & $0.8$ & 845.76 & 2646 & 0.00 & 1.6477$\times 10^{10}$ & 1.6477$\times 10^{10}$ & 1.5178$\times 10^{10}$ & 20.61 & 11.82 \\
            100 & $0.2$ & 153.64 & 51 & 0.00 & 3.1463$\times 10^{10}$ & 3.1463$\times 10^{10}$ & 3.1136$\times 10^{10}$ & 88.25 & 26.74 \\
             & $0.5$ & 1699.19 & 3017 & 0.00 & 3.0265$\times 10^{10}$ & 3.0265$\times 10^{10}$ & 2.9021$\times 10^{10}$ & 95.84 & 19.92 \\
             & $0.8$ & 6774.36 & 4054 & 0.00 & 2.8468$\times 10^{10}$ & 2.8468$\times 10^{10}$ & 2.6366$\times 10^{10}$ & 112.84 & 25.99
            \end{tabular}}
            \caption{Results for the MA $G$-median problem for CAB dataset. \label{tab:MA_Gmed_CAB}}
        \end{table}

        \begin{table}[H]
            \centering
            {\begin{tabular}{cc|rrrrrrrr}
            $n$ & $\alpha$ & \multicolumn{1}{c}{CPU (s)} & \multicolumn{1}{c}{$\#Exp$} & \multicolumn{1}{c}{GAP (\%)} & \multicolumn{1}{c}{UB} & \multicolumn{1}{c}{LB} & \multicolumn{1}{c}{LB$_{root}$} & \multicolumn{1}{c}{CPU$_{root}$} & \multicolumn{1}{c}{CPU$_{SP} (\%)$} \\ \hline\hline 
            25 & $0.2$ & 0.56 & 3 & 0.00 & 1.2529$\times 10^{8}$ & 1.2529$\times 10^{8}$ & 1.2517$\times 10^{8}$ & 0.37 & 34.95 \\
             & $0.5$ & 0.14 & 1 & 0.00 & 1.1809$\times 10^{8}$ & 1.1809$\times 10^{8}$ & 1.1809$\times 10^{8}$ & 0.12 & 41.43 \\
             & $0.8$ & 0.77 & 3 & 0.00 & 1.1088$\times 10^{8}$ & 1.1088$\times 10^{8}$ & 1.1064$\times 10^{8}$ & 0.52 & 37.34 \\
            50 & $0.2$ & 0.61 & 1 & 0.00 & 1.3102$\times 10^{8}$ & 1.3102$\times 10^{8}$ & 1.3102$\times 10^{8}$ & 0.60 & 31.84 \\
             & $0.5$ & 1.09 & 1 & 0.00 & 1.2305$\times 10^{8}$ & 1.2305$\times 10^{8}$ & 1.2305$\times 10^{8}$ & 1.08 & 29.11 \\
             & $0.8$ & 1.22 & 1 & 0.00 & 1.1508$\times 10^{8}$ & 1.1508$\times 10^{8}$ & 1.1508$\times 10^{8}$ & 1.21 & 26.71 \\
            75 & $0.2$ & 4.63 & 1 & 0.00 & 1.3335$\times 10^{8}$ & 1.3335$\times 10^{8}$ & 1.3335$\times 10^{8}$ & 4.61 & 12.69 \\
             & $0.5$ & 7.63 & 1 & 0.00 & 1.2512$\times 10^{8}$ & 1.2512$\times 10^{8}$ & 1.2512$\times 10^{8}$ & 7.55 & 11.80 \\
             & $0.8$ & 10.80 & 1 & 0.00 & 1.1690$\times 10^{8}$ & 1.1690$\times 10^{8}$ & 1.1690$\times 10^{8}$ & 10.72 & 34.72 \\
            100 & $0.2$ & 29.58 & 1 & 0.00 & 1.3382$\times 10^{8}$ & 1.3382$\times 10^{8}$ & 1.3382$\times 10^{8}$ & 29.50 & 15.04 \\
             & $0.5$ & 31.48 & 1 & 0.00 & 1.2545$\times 10^{8}$ & 1.2545$\times 10^{8}$ & 1.2545$\times 10^{8}$ & 31.41 & 34.31 \\
             & $0.8$ & 163.34 & 15 & 0.00 & 1.1708$\times 10^{8}$ & 1.1708$\times 10^{8}$ & 1.1674$\times 10^{8}$ & 120.84 & 22.19 \\
            125 & $0.2$ & 60.99 & 1 & 0.00 & 1.3466$\times 10^{8}$ & 1.3466$\times 10^{8}$ & 1.3466$\times 10^{8}$ & 60.87 & 32.59 \\
             & $0.5$ & 146.10 & 1 & 0.00 & 1.2661$\times 10^{8}$ & 1.2661$\times 10^{8}$ & 1.2661$\times 10^{8}$ & 146.04 & 31.60 \\
             & $0.8$ & 771.24 & 504 & 0.00 & 1.2880$\times 10^{8}$ & 1.2880$\times 10^{8}$ & 1.1329$\times 10^{8}$ & 152.71 & 33.27 \\
            150 & $0.2$ & 195.91 & 1 & 0.00 & 1.3526$\times 10^{8}$ & 1.3526$\times 10^{8}$ & 1.3526$\times 10^{8}$ & 195.79 & 34.45 \\
             & $0.5$ & 356.88 & 1 & 0.00 & 1.2714$\times 10^{8}$ & 1.2714$\times 10^{8}$ & 1.2714$\times 10^{8}$ & 356.77 & 34.78 \\
             & $0.8$ & 2314.57 & 593 & 0.00 & 1.2267$\times 10^{8}$ & 1.2267$\times 10^{8}$ & 1.1566$\times 10^{8}$ & 412.26 & 34.97 \\
            175 & $0.2$ & 435.75 & 1 & 0.00 & 1.3615$\times 10^{8}$ & 1.3615$\times 10^{8}$ & 1.3615$\times 10^{8}$ & 435.59 & 35.60 \\
             & $0.5$ & 622.35 & 1 & 0.00 & 1.2788$\times 10^{8}$ & 1.2788$\times 10^{8}$ & 1.2788$\times 10^{8}$ & 622.20 & 39.29 \\
             & $0.8$ & 5986.32 & 1288 & 0.00 & 1.3322$\times 10^{8}$ & 1.1686$\times 10^{8}$ & 1.1521$\times 10^{8}$ & 818.01 & 38.85 \\
            200 & $0.2$ & 606.14 & 1 & 0.00 & 1.3593$\times 10^{8}$ & 1.3593$\times 10^{8}$ & 1.3593$\times 10^{8}$ & 605.97 & 40.16 \\
             & $0.5$ & 1887.09 & 1 & 0.00 & 1.2783$\times 10^{8}$ & 1.2783$\times 10^{8}$ & 1.2783$\times 10^{8}$ & 1886.90 & 39.63 \\
             & $0.8$ & 636.77 & 21 & 0.00 & 1.3531$\times 10^{8}$ & 1.1552$\times 10^{8}$ & 1.1484$\times 10^{8}$ & 253.47 & 38.71
            \end{tabular}}
            \caption{Results for the SA $H$-median problem for AP dataset. \label{tab:SA_Hmed_AP}}
        \end{table}

        \begin{table}[H]
            \centering
            {\begin{tabular}{cc|rrrrrrrr}
            $n$ & $\alpha$ & \multicolumn{1}{c}{CPU (s)} & \multicolumn{1}{c}{$\#Exp$} & \multicolumn{1}{c}{GAP (\%)} & \multicolumn{1}{c}{UB} & \multicolumn{1}{c}{LB} & \multicolumn{1}{c}{LB$_{root}$} & \multicolumn{1}{c}{CPU$_{root}$} & \multicolumn{1}{c}{CPU$_{SP} (\%)$} \\ \hline\hline 
            25 & $0.2$ & 1.57 & 8 & 0.00 & 1.2529$\times 10^{8}$ & 1.2529$\times 10^{8}$ & 1.2367$\times 10^{8}$ & 0.24 & 30.74 \\
             & $0.5$ & 2.36 & 23 & 0.00 & 1.1809$\times 10^{8}$ & 1.1809$\times 10^{8}$ & 1.1410$\times 10^{8}$ & 0.44 & 27.66 \\
             & $0.8$ & 3.44 & 80 & 0.00 & 1.0923$\times 10^{8}$ & 1.0923$\times 10^{8}$ & 1.0273$\times 10^{8}$ & 0.52 & 33.30 \\
            50 & $0.2$ & 11.91 & 1 & 0.00 & 1.3102$\times 10^{8}$ & 1.3102$\times 10^{8}$ & 1.3102$\times 10^{8}$ & 0.54 & 10.45 \\
             & $0.5$ & 11.02 & 108 & 0.00 & 1.2305$\times 10^{8}$ & 1.2305$\times 10^{8}$ & 1.2049$\times 10^{8}$ & 1.97 & 20.06 \\
             & $0.8$ & 21.23 & 403 & 0.00 & 1.1508$\times 10^{8}$ & 1.1508$\times 10^{8}$ & 1.0813$\times 10^{8}$ & 3.05 & 16.14 \\
            75 & $0.2$ & 27.54 & 1 & 0.00 & 1.3335$\times 10^{8}$ & 1.3335$\times 10^{8}$ & 1.3335$\times 10^{8}$ & 4.32 & 18.28 \\
             & $0.5$ & 38.40 & 15 & 0.00 & 1.2512$\times 10^{8}$ & 1.2512$\times 10^{8}$ & 1.2342$\times 10^{8}$ & 8.16 & 12.95 \\
             & $0.8$ & 68.46 & 269 & 0.00 & 1.1690$\times 10^{8}$ & 1.1690$\times 10^{8}$ & 1.1149$\times 10^{8}$ & 11.90 & 24.09 \\
            100 & $0.2$ & 121.72 & 1 & 0.00 & 1.3382$\times 10^{8}$ & 1.3382$\times 10^{8}$ & 1.3382$\times 10^{8}$ & 14.10 & 25.89 \\
             & $0.5$ & 165.28 & 40 & 0.00 & 1.2545$\times 10^{8}$ & 1.2545$\times 10^{8}$ & 1.2346$\times 10^{8}$ & 26.04 & 27.16 \\
             & $0.8$ & 314.52 & 862 & 0.00 & 1.1708$\times 10^{8}$ & 1.1708$\times 10^{8}$ & 1.1108$\times 10^{8}$ & 51.30 & 35.32 \\
            125 & $0.2$ & 177.60 & 41 & 0.00 & 1.2269$\times 10^{8}$ & 1.2269$\times 10^{8}$ & 1.1909$\times 10^{8}$ & 128.30 & 36.45 \\
             & $0.5$ & 387.32 & 174 & 0.00 & 1.2632$\times 10^{8}$ & 1.2632$\times 10^{8}$ & 1.2024$\times 10^{8}$ & 182.54 & 33.35 \\
             & $0.8$ & 597.73 & 388 & 0.00 & 1.1797$\times 10^{8}$ & 1.1797$\times 10^{8}$ & 1.1083$\times 10^{8}$ & 181.63 & 35.97 \\
            150 & $0.2$ & 749.87 & 25 & 0.00 & 1.1714$\times 10^{8}$ & 1.1714$\times 10^{8}$ & 1.1489$\times 10^{8}$ & 444.51 & 35.94 \\
             & $0.5$ & 862.87 & 205 & 0.00 & 1.2047$\times 10^{8}$ & 1.2047$\times 10^{8}$ & 1.1552$\times 10^{8}$ & 433.11 & 38.15 \\
             & $0.8$ & 813.61 & 238 & 0.00 & 1.1903$\times 10^{8}$ & 1.1903$\times 10^{8}$ & 1.1593$\times 10^{8}$ & 382.04 & 39.78 \\
            175 & $0.2$ & 1918.21 & 9 & 0.00 & 1.1501$\times 10^{8}$ & 1.1501$\times 10^{8}$ & 1.1272$\times 10^{8}$ & 1766.02 & 38.16 \\
             & $0.5$ & 1475.44 & 50 & 0.00 & 1.1911$\times 10^{8}$ & 1.1911$\times 10^{8}$ & 1.1379$\times 10^{8}$ & 1173.69 & 38.88 \\
             & $0.8$ & 1788.90 & 156 & 0.00 & 1.2114$\times 10^{8}$ & 1.2114$\times 10^{8}$ & 1.1357$\times 10^{8}$ & 785.55 & 40.53 \\
            200 & $0.2$ & 3114.74 & 21 & 0.00 & 1.2241$\times 10^{8}$ & 1.2241$\times 10^{8}$ & 1.1949$\times 10^{8}$ & 2676.94 & 40.02 \\
             & $0.5$ & 4025.22 & 154 & 0.00 & 1.2677$\times 10^{8}$ & 1.2677$\times 10^{8}$ & 1.2020$\times 10^{8}$ & 2238.12 & 40.58 \\
             & $0.8$ & TL & 402 & 1.96 & 1.2900$\times 10^{8}$ & 1.2647$\times 10^{8}$ & 1.2047$\times 10^{8}$ & 2479.42 & 40.50
            \end{tabular}}
            \caption{Results for MA $H$-median problem for AP dataset. \label{tab:MA_Hmed_AP}}
        \end{table}

        \begin{table}[H]
            \centering
            {\begin{tabular}{cc|rrrrrrrr}
            $n$ & $\alpha$ & \multicolumn{1}{c}{CPU (s)} & \multicolumn{1}{c}{$\#Exp$} & \multicolumn{1}{c}{GAP (\%)} & \multicolumn{1}{c}{UB} & \multicolumn{1}{c}{LB} & \multicolumn{1}{c}{LB$_{root}$} & \multicolumn{1}{c}{CPU$_{root}$} & \multicolumn{1}{c}{CPU$_{SP} (\%)$} \\ \hline\hline 
            25 & $0.2$ & 0.74 & 3 & 0.00 & 1.2826$\times 10^{8}$ & 1.2826$\times 10^{8}$ & 1.2822$\times 10^{8}$ & 0.42 & 55.43 \\
             & $0.5$ & 0.17 & 1 & 0.00 & 1.2089$\times 10^{8}$ & 1.2089$\times 10^{8}$ & 1.2089$\times 10^{8}$ & 0.16 & 25.22 \\
             & $0.8$ & 0.92 & 3 & 0.00 & 1.1351$\times 10^{8}$ & 1.1351$\times 10^{8}$ & 1.1333$\times 10^{8}$ & 0.59 & 29.69 \\
            50 & $0.2$ & 0.93 & 1 & 0.00 & 1.3412$\times 10^{8}$ & 1.3412$\times 10^{8}$ & 1.3412$\times 10^{8}$ & 0.90 & 54.71 \\
             & $0.5$ & 1.70 & 1 & 0.00 & 1.2597$\times 10^{8}$ & 1.2597$\times 10^{8}$ & 1.2597$\times 10^{8}$ & 1.66 & 17.31 \\
             & $0.8$ & 10.78 & 3 & 0.00 & 1.1781$\times 10^{8}$ & 1.1781$\times 10^{8}$ & 1.1774$\times 10^{8}$ & 6.20 & 21.59 \\
            75 & $0.2$ & 10.17 & 1 & 0.00 & 1.3651$\times 10^{8}$ & 1.3651$\times 10^{8}$ & 1.3651$\times 10^{8}$ & 10.08 & 30.49 \\
             & $0.5$ & 16.07 & 1 & 0.00 & 1.2808$\times 10^{8}$ & 1.2808$\times 10^{8}$ & 1.2808$\times 10^{8}$ & 15.98 & 21.96 \\
             & $0.8$ & 28.93 & 1 & 0.00 & 1.1967$\times 10^{8}$ & 1.1967$\times 10^{8}$ & 1.1967$\times 10^{8}$ & 28.89 & 16.16 \\
            100 & $0.2$ & 30.75 & 1 & 0.00 & 1.3699$\times 10^{8}$ & 1.3699$\times 10^{8}$ & 1.3699$\times 10^{8}$ & 30.66 & 24.11 \\
             & $0.5$ & 29.93 & 1 & 0.00 & 1.2842$\times 10^{8}$ & 1.2842$\times 10^{8}$ & 1.2842$\times 10^{8}$ & 29.85 & 19.89 \\
             & $0.8$ & 134.93 & 11 & 0.00 & 1.1985$\times 10^{8}$ & 1.1985$\times 10^{8}$ & 1.1951$\times 10^{8}$ & 92.77 & 16.39 \\
            125 & $0.2$ & 63.22 & 1 & 0.00 & 1.3785$\times 10^{8}$ & 1.3785$\times 10^{8}$ & 1.3785$\times 10^{8}$ & 63.13 & 17.75 \\
             & $0.5$ & 172.88 & 1 & 0.00 & 1.2931$\times 10^{8}$ & 1.2931$\times 10^{8}$ & 1.2931$\times 10^{8}$ & 172.78 & 16.59 \\
             & $0.8$ & 775.31 & 523 & 0.00 & 1.2076$\times 10^{8}$ & 1.2076$\times 10^{8}$ & 1.1597$\times 10^{8}$ & 152.44 & 15.66 \\
            150 & $0.2$ & 158.12 & 1 & 0.00 & 1.3846$\times 10^{8}$ & 1.3846$\times 10^{8}$ & 1.3846$\times 10^{8}$ & 158.00 & 16.54 \\
             & $0.5$ & 350.95 & 1 & 0.00 & 1.3015$\times 10^{8}$ & 1.3015$\times 10^{8}$ & 1.3015$\times 10^{8}$ & 350.79 & 16.53 \\
             & $0.8$ & 3719.71 & 1031 & 0.00 & 1.2185$\times 10^{8}$ & 1.2185$\times 10^{8}$ & 1.1868$\times 10^{8}$ & 399.90 & 16.23 \\
            175 & $0.2$ & 373.86 & 1 & 0.00 & 1.3938$\times 10^{8}$ & 1.3938$\times 10^{8}$ & 1.3938$\times 10^{8}$ & 373.71 & 16.67 \\
             & $0.5$ & 635.52 & 1 & 0.00 & 1.3091$\times 10^{8}$ & 1.3091$\times 10^{8}$ & 1.3091$\times 10^{8}$ & 635.39 & 16.18 \\
             & $0.8$ & 6958.72 & 1141 & 0.00 & 1.3653$\times 10^{8}$ & 1.1985$\times 10^{8}$ & 1.1794$\times 10^{8}$ & 813.88 & 16.84 \\
            200 & $0.2$ & 840.34 & 1 & 0.00 & 1.3915$\times 10^{8}$ & 1.3915$\times 10^{8}$ & 1.3915$\times 10^{8}$ & 840.18 & 16.03 \\
             & $0.5$ & 1387.11 & 1 & 0.00 & 1.3086$\times 10^{8}$ & 1.3086$\times 10^{8}$ & 1.3086$\times 10^{8}$ & 1386.94 & 15.82 \\
             & $0.8$ & TL & 555 & 2.15 & 1.3852$\times 10^{8}$ & 1.3554$\times 10^{8}$ & 1.1757$\times 10^{8}$ & 1513.25 & 15.32
            \end{tabular}}
            \caption{Results for SA $G$-median problem for AP dataset. \label{tab:SA_Gmed_AP}}
        \end{table}
 
        \begin{table}[H]
            \centering
            {\begin{tabular}{cc|rrrrrrrr}
            $n$ & $\alpha$ & \multicolumn{1}{c}{CPU (s)} & \multicolumn{1}{c}{$\#Exp$} & \multicolumn{1}{c}{GAP (\%)} & \multicolumn{1}{c}{UB} & \multicolumn{1}{c}{LB} & \multicolumn{1}{c}{LB$_{root}$} & \multicolumn{1}{c}{CPU$_{root}$} & \multicolumn{1}{c}{CPU$_{SP} (\%)$} \\ \hline\hline 
            25 & $0.2$ & 0.47 & 8 & 0.00 & 1.2701$\times 10^{8}$ & 1.2701$\times 10^{8}$ & 1.2546$\times 10^{8}$ & 0.23 & 51.17 \\
             & $0.5$ & 1.37 & 58 & 0.00 & 1.1971$\times 10^{8}$ & 1.1971$\times 10^{8}$ & 1.1574$\times 10^{8}$ & 0.44 & 33.25 \\
             & $0.8$ & 1.77 & 71 & 0.00 & 1.1091$\times 10^{8}$ & 1.1091$\times 10^{8}$ & 1.0467$\times 10^{8}$ & 0.48 & 35.90 \\
            50 & $0.2$ & 0.68 & 1 & 0.00 & 1.3281$\times 10^{8}$ & 1.3281$\times 10^{8}$ & 1.3281$\times 10^{8}$ & 0.68 & 53.24 \\
             & $0.5$ & 17.32 & 154 & 0.00 & 1.2473$\times 10^{8}$ & 1.2473$\times 10^{8}$ & 1.2225$\times 10^{8}$ & 2.38 & 21.99 \\
             & $0.8$ & 37.00 & 488 & 0.00 & 1.1666$\times 10^{8}$ & 1.1666$\times 10^{8}$ & 1.0977$\times 10^{8}$ & 3.91 & 23.05 \\
            75 & $0.2$ & 4.97 & 1 & 0.00 & 1.3518$\times 10^{8}$ & 1.3518$\times 10^{8}$ & 1.3518$\times 10^{8}$ & 4.86 & 30.69 \\
             & $0.5$ & 14.95 & 15 & 0.00 & 1.2683$\times 10^{8}$ & 1.2683$\times 10^{8}$ & 1.2517$\times 10^{8}$ & 11.09 & 31.25 \\
             & $0.8$ & 300.17 & 858 & 0.00 & 1.1850$\times 10^{8}$ & 1.1850$\times 10^{8}$ & 1.1315$\times 10^{8}$ & 21.36 & 17.42 \\
            100 & $0.2$ & 18.67 & 1 & 0.00 & 1.3565$\times 10^{8}$ & 1.3565$\times 10^{8}$ & 1.3565$\times 10^{8}$ & 18.57 & 23.52 \\
             & $0.5$ & 55.86 & 38 & 0.00 & 1.2717$\times 10^{8}$ & 1.2717$\times 10^{8}$ & 1.2522$\times 10^{8}$ & 40.24 & 24.11 \\
             & $0.8$ & 982.69 & 1779 & 0.00 & 1.1868$\times 10^{8}$ & 1.1868$\times 10^{8}$ & 1.1251$\times 10^{8}$ & 61.74 & 17.40 \\
            125 & $0.2$ & 74.31 & 1 & 0.00 & 1.3650$\times 10^{8}$ & 1.3650$\times 10^{8}$ & 1.3650$\times 10^{8}$ & 74.20 & 18.16 \\
             & $0.5$ & 178.15 & 43 & 0.00 & 1.2805$\times 10^{8}$ & 1.2805$\times 10^{8}$ & 1.2635$\times 10^{8}$ & 97.86 & 21.29 \\
             & $0.8$ & 1148.44 & 839 & 0.00 & 1.1959$\times 10^{8}$ & 1.1959$\times 10^{8}$ & 1.1368$\times 10^{8}$ & 130.51 & 17.38 \\
            150 & $0.2$ & 222.03 & 1 & 0.00 & 1.3711$\times 10^{8}$ & 1.3711$\times 10^{8}$ & 1.3711$\times 10^{8}$ & 221.94 & 16.58 \\
             & $0.5$ & 317.37 & 14 & 0.00 & 1.2888$\times 10^{8}$ & 1.2888$\times 10^{8}$ & 1.2796$\times 10^{8}$ & 210.73 & 21.34 \\
             & $0.8$ & 3082.97 & 1017 & 0.00 & 1.2066$\times 10^{8}$ & 1.2066$\times 10^{8}$ & 1.1494$\times 10^{8}$ & 265.64 & 16.33 \\
            175 & $0.2$ & 389.00 & 1 & 0.00 & 1.3801$\times 10^{8}$ & 1.3801$\times 10^{8}$ & 1.3801$\times 10^{8}$ & 388.86 & 16.49 \\
             & $0.5$ & 686.21 & 55 & 0.00 & 1.2963$\times 10^{8}$ & 1.2963$\times 10^{8}$ & 1.2807$\times 10^{8}$ & 423.90 & 19.09 \\
             & $0.8$ & 2365.66 & 688 & 0.00 & 1.2352$\times 10^{8}$ & 1.1712$\times 10^{8}$ & 1.1412$\times 10^{8}$ & 557.24 & 17.27 \\
            200 & $0.2$ & 869.64 & 1 & 0.00 & 1.3779$\times 10^{8}$ & 1.3779$\times 10^{8}$ & 1.3779$\times 10^{8}$ & 869.46 & 16.05 \\
             & $0.5$ & 1540.95 & 24 & 0.00 & 1.2958$\times 10^{8}$ & 1.2958$\times 10^{8}$ & 1.2859$\times 10^{8}$ & 976.88 & 17.81 \\
             & $0.8$ & TL & 131 & 6.33 & 1.2430$\times 10^{8}$ & 1.1643$\times 10^{8}$ & 1.1466$\times 10^{8}$ & 1330.18 & 15.53
            \end{tabular}}
            \caption{Results for MA $G$-median problem for AP dataset. \label{tab:MA_Gmed_AP}}
        \end{table}

    \subsection{Complementary material for Section 8.6.}
        Figure \ref{fig:Gmed_CGmed_Comparison} illustrates the network design changes for an instance of the CAB dataset with $n=20$ nodes and $\alpha=0.5$, comparing the G-median and the Flow Bound G-median ($G$-FB) problems under SA and MA allocation policies. Hubs are represented as red squares and non-hubs as blue circles. The hub that differs between the G-median and $G$-FB solutions is highlighted as a yellow square to improve readability. This figure is referenced in Section 8.6 of the manuscript.

        Tables \ref{tab:matrix_SA_Gmed} to \ref{tab:matrix_SAFB_Gmed} show optimal flows for the networks shown in Figure \ref{fig:Gmed_CGmed_Comparison}.
        
        \begin{figure}[H]
            \centering
            {\includegraphics[width=1\textwidth]{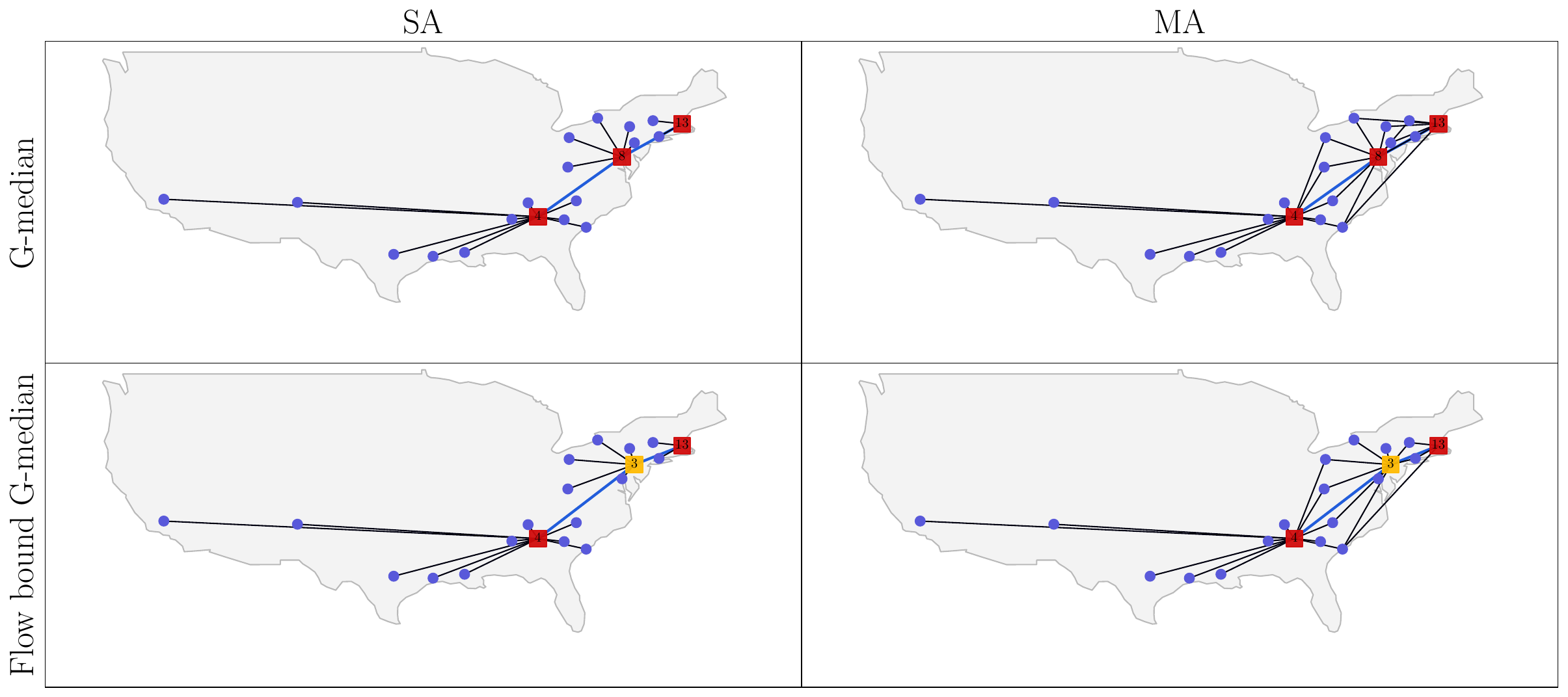}}
            \caption{Comparison of results obtained for the G-median and the Flow bound G-median on CAB dataset. \label{fig:Gmed_CGmed_Comparison}}
            {\footnotesize Instance data: $n=20$, $\alpha=0.5$.}
        \end{figure}
        
        \begin{table}[H]
            \centering
            {
            \resizebox{.95\textwidth}{!}{
                \begin{tabular}{c|rrrrrrrrrrrrrrrrrrrr}
                    Flows & \multicolumn{1}{c}{0} & \multicolumn{1}{c}{1} & \multicolumn{1}{c}{2} & \multicolumn{1}{c}{3} & \multicolumn{1}{c}{4} & \multicolumn{1}{c}{5} & \multicolumn{1}{c}{6} & \multicolumn{1}{c}{7} & \multicolumn{1}{c}{8} & \multicolumn{1}{c}{9} & \multicolumn{1}{c}{10} & \multicolumn{1}{c}{11} & \multicolumn{1}{c}{12} & \multicolumn{1}{c}{13} & \multicolumn{1}{c}{14} & \multicolumn{1}{c}{15} & \multicolumn{1}{c}{16} & \multicolumn{1}{c}{17} & \multicolumn{1}{c}{18} & \multicolumn{1}{c}{19} \\ \hline\hline
                    0 &  &  &  &  &  &  &  &  & 244 &  &  &  &  &  &  &  &  &  &  &  \\
        1 &  &  &  &  &  &  &  &  &  &  &  &  &  & 16236 &  &  &  &  &  &  \\
        2 &  &  &  &  & 4 &  &  &  &  &  &  &  &  &  &  &  &  &  &  &  \\
        3 &  &  &  &  &  &  &  &  & 3719 &  &  &  &  &  &  &  &  &  &  &  \\
        4 &  &  & 4 &  &  & 646 & 2559 & 32 & 38438 & 2483 & 1278 &  & 14158 &  &  &  & 9727 &  & 18673 & 4887 \\
        5 &  &  &  &  & 646 &  &  &  &  &  &  &  &  &  &  &  &  &  &  &  \\
        6 &  &  &  &  & 2559 &  &  &  &  &  &  &  &  &  &  &  &  &  &  &  \\
        7 &  &  &  &  & 32 &  &  &  &  &  &  &  &  &  &  &  &  &  &  &  \\
        8 & 244 &  &  & 3719 & 38438 &  &  &  &  &  &  & 1595 &  & 5445 &  & 24128 &  & 4281 &  &  \\
        9 &  &  &  &  & 2483 &  &  &  &  &  &  &  &  &  &  &  &  &  &  &  \\
        10 &  &  &  &  & 1278 &  &  &  &  &  &  &  &  &  &  &  &  &  &  &  \\
        11 &  &  &  &  &  &  &  &  & 1595 &  &  &  &  &  &  &  &  &  &  &  \\
        12 &  &  &  &  & 14158 &  &  &  &  &  &  &  &  &  &  &  &  &  &  &  \\
        13 &  & 16236 &  &  &  &  &  &  & 5445 &  &  &  &  &  & 2589 &  &  &  &  &  \\
        14 &  &  &  &  &  &  &  &  &  &  &  &  &  & 2589 &  &  &  &  &  &  \\
        15 &  &  &  &  &  &  &  &  & 24128 &  &  &  &  &  &  &  &  &  &  &  \\
        16 &  &  &  &  & 9727 &  &  &  &  &  &  &  &  &  &  &  &  &  &  &  \\
        17 &  &  &  &  &  &  &  &  & 4281 &  &  &  &  &  &  &  &  &  &  &  \\
        18 &  &  &  &  & 18673 &  &  &  &  &  &  &  &  &  &  &  &  &  &  &  \\
        19 &  &  &  &  & 4887 &  &  &  &  &  &  &  &  &  &  &  &  &  &  & 
                \end{tabular}}
            }
            \caption{Flow solution for the SA $G$-median from Figure \ref{fig:Gmed_CGmed_Comparison}. \label{tab:matrix_SA_Gmed}}
        \end{table}
        
        \begin{table}[H]
            \centering
            {
            \resizebox{.95\textwidth}{!}{
                \begin{tabular}{c|rrrrrrrrrrrrrrrrrrrr}
                    Flows & \multicolumn{1}{c}{0} & \multicolumn{1}{c}{1} & \multicolumn{1}{c}{2} & \multicolumn{1}{c}{3} & \multicolumn{1}{c}{4} & \multicolumn{1}{c}{5} & \multicolumn{1}{c}{6} & \multicolumn{1}{c}{7} & \multicolumn{1}{c}{8} & \multicolumn{1}{c}{9} & \multicolumn{1}{c}{10} & \multicolumn{1}{c}{11} & \multicolumn{1}{c}{12} & \multicolumn{1}{c}{13} & \multicolumn{1}{c}{14} & \multicolumn{1}{c}{15} & \multicolumn{1}{c}{16} & \multicolumn{1}{c}{17} & \multicolumn{1}{c}{18} & \multicolumn{1}{c}{19} \\ \hline\hline
                    0 &  &  &  & 244 &  &  &  &  &  &  &  &  &  &  &  &  &  &  &  &  \\
        1 &  &  &  &  &  &  &  &  &  &  &  &  &  & 16236 &  &  &  &  &  &  \\
        2 &  &  &  &  & 4 &  &  &  &  &  &  &  &  &  &  &  &  &  &  &  \\
        3 & 244 &  &  &  & 38438 &  &  &  & 3738 &  &  & 1595 &  & 5445 &  & 24128 &  & 4281 &  &  \\
        4 &  &  & 4 & 38438 &  & 646 & 2559 & 32 &  & 2483 & 1278 &  & 14158 &  &  &  & 9727 &  & 18673 & 4887 \\
        5 &  &  &  &  & 646 &  &  &  &  &  &  &  &  &  &  &  &  &  &  &  \\
        6 &  &  &  &  & 2559 &  &  &  &  &  &  &  &  &  &  &  &  &  &  &  \\
        7 &  &  &  &  & 32 &  &  &  &  &  &  &  &  &  &  &  &  &  &  &  \\
        8 &  &  &  & 3738 &  &  &  &  &  &  &  &  &  &  &  &  &  &  &  &  \\
        9 &  &  &  &  & 2483 &  &  &  &  &  &  &  &  &  &  &  &  &  &  &  \\
        10 &  &  &  &  & 1278 &  &  &  &  &  &  &  &  &  &  &  &  &  &  &  \\
        11 &  &  &  & 1595 &  &  &  &  &  &  &  &  &  &  &  &  &  &  &  &  \\
        12 &  &  &  &  & 14158 &  &  &  &  &  &  &  &  &  &  &  &  &  &  &  \\
        13 &  & 16236 &  & 5445 &  &  &  &  &  &  &  &  &  &  & 2589 &  &  &  &  &  \\
        14 &  &  &  &  &  &  &  &  &  &  &  &  &  & 2589 &  &  &  &  &  &  \\
        15 &  &  &  & 24128 &  &  &  &  &  &  &  &  &  &  &  &  &  &  &  &  \\
        16 &  &  &  &  & 9727 &  &  &  &  &  &  &  &  &  &  &  &  &  &  &  \\
        17 &  &  &  & 4281 &  &  &  &  &  &  &  &  &  &  &  &  &  &  &  &  \\
        18 &  &  &  &  & 18673 &  &  &  &  &  &  &  &  &  &  &  &  &  &  &  \\
        19 &  &  &  &  & 4887 &  &  &  &  &  &  &  &  &  &  &  &  &  &  & 
                \end{tabular}}
            }
            \caption{Flow solution for the SA $GHLP$ with flow bounds ($G-FB$) from Figure \ref{fig:Gmed_CGmed_Comparison}.\label{tab:matrix_SAFB_Gmed}}
        \end{table}
        
        \begin{table}[H]
            \centering
            {
            \resizebox{.95\textwidth}{!}{
                \begin{tabular}{c|rrrrrrrrrrrrrrrrrrrr}
                    Flows & \multicolumn{1}{c}{0} & \multicolumn{1}{c}{1} & \multicolumn{1}{c}{2} & \multicolumn{1}{c}{3} & \multicolumn{1}{c}{4} & \multicolumn{1}{c}{5} & \multicolumn{1}{c}{6} & \multicolumn{1}{c}{7} & \multicolumn{1}{c}{8} & \multicolumn{1}{c}{9} & \multicolumn{1}{c}{10} & \multicolumn{1}{c}{11} & \multicolumn{1}{c}{12} & \multicolumn{1}{c}{13} & \multicolumn{1}{c}{14} & \multicolumn{1}{c}{15} & \multicolumn{1}{c}{16} & \multicolumn{1}{c}{17} & \multicolumn{1}{c}{18} & \multicolumn{1}{c}{19} \\ \hline\hline
                    0 &  &  &  &  & 671 &  &  &  & 1733 &  &  &  &  &  &  &  &  &  &  &  \\
        1 &  &  &  &  &  &  &  &  & 326 &  &  &  &  & 12976 &  &  &  &  &  &  \\
        2 &  &  &  &  & 4 &  &  &  &  &  &  &  &  &  &  &  &  &  &  &  \\
        3 &  &  &  &  &  &  &  &  & 1165 &  &  &  &  & 2554 &  &  &  &  &  &  \\
        4 & 671 &  & 4 &  &  & 646 & 2559 & 32 & 27153 & 2483 & 1278 &  & 14158 &  &  &  & 6641 & 1938 & 1383 & 4887 \\
        5 &  &  &  &  & 646 &  &  &  &  &  &  &  &  &  &  &  &  &  &  &  \\
        6 &  &  &  &  & 2559 &  &  &  &  &  &  &  &  &  &  &  &  &  &  &  \\
        7 &  &  &  &  & 32 &  &  &  &  &  &  &  &  &  &  &  &  &  &  &  \\
        8 & 1733 & 326 &  & 1165 & 27153 &  &  &  &  &  &  & 583 &  & 28754 & 276 & 6942 & 1723 & 2343 & 559 &  \\
        9 &  &  &  &  & 2483 &  &  &  &  &  &  &  &  &  &  &  &  &  &  &  \\
        10 &  &  &  &  & 1278 &  &  &  &  &  &  &  &  &  &  &  &  &  &  &  \\
        11 &  &  &  &  &  &  &  &  & 583 &  &  &  &  & 112 &  &  &  &  &  &  \\
        12 &  &  &  &  & 14158 &  &  &  &  &  &  &  &  &  &  &  &  &  &  &  \\
        13 &  & 12976 &  & 2554 &  &  &  &  & 28754 &  &  & 112 &  &  & 2313 & 17186 & 1363 &  &  &  \\
        14 &  &  &  &  &  &  &  &  & 276 &  &  &  &  & 2313 &  &  &  &  &  &  \\
        15 &  &  &  &  &  &  &  &  & 6942 &  &  &  &  & 17186 &  &  &  &  &  &  \\
        16 &  &  &  &  & 6641 &  &  &  & 1723 &  &  &  &  & 1363 &  &  &  &  &  &  \\
        17 &  &  &  &  & 1938 &  &  &  & 2343 &  &  &  &  &  &  &  &  &  &  &  \\
        18 &  &  &  &  & 1383 &  &  &  & 559 &  &  &  &  &  &  &  &  &  &  &  \\
        19 &  &  &  &  & 4887 &  &  &  &  &  &  &  &  &  &  &  &  &  &  & 
                \end{tabular}}
            }
            \caption{Flow solution for the MA $G$-median from Figure \ref{fig:Gmed_CGmed_Comparison}.\label{tab:matrix_MA_Gmed}}
        \end{table}
        
        \begin{table}[H]
            \centering
            {
            \resizebox{.95\textwidth}{!}{
                \begin{tabular}{c|rrrrrrrrrrrrrrrrrrrr}
                    Flows & \multicolumn{1}{c}{0} & \multicolumn{1}{c}{1} & \multicolumn{1}{c}{2} & \multicolumn{1}{c}{3} & \multicolumn{1}{c}{4} & \multicolumn{1}{c}{5} & \multicolumn{1}{c}{6} & \multicolumn{1}{c}{7} & \multicolumn{1}{c}{8} & \multicolumn{1}{c}{9} & \multicolumn{1}{c}{10} & \multicolumn{1}{c}{11} & \multicolumn{1}{c}{12} & \multicolumn{1}{c}{13} & \multicolumn{1}{c}{14} & \multicolumn{1}{c}{15} & \multicolumn{1}{c}{16} & \multicolumn{1}{c}{17} & \multicolumn{1}{c}{18} & \multicolumn{1}{c}{19} \\ \hline\hline
                    0 &  &  &  & 97 & 1434 &  &  &  &  &  &  &  &  &  &  &  &  &  &  &  \\
                    1 &  &  &  & 1399 &  &  &  &  &  &  &  &  &  & 5837 &  &  &  &  &  &  \\
                    2 &  &  &  &  & 4 &  &  &  &  &  &  &  &  &  &  &  &  &  &  &  \\
                    3 & 97 & 1399 &  &  & 27153 &  &  &  & 3738 &  &  & 1595 &  & 42366 & 26 & 24128 & 1586 & 2343 & 4947 &  \\
                    4 & 1434 &  & 4 & 27153 &  & 646 & 2559 & 32 &  & 2483 & 1278 &  & 14158 &  &  &  & 6761 & 1938 & 13726 & 4887 \\
                    5 &  &  &  &  & 646 &  &  &  &  &  &  &  &  &  &  &  &  &  &  &  \\
                    6 &  &  &  &  & 2559 &  &  &  &  &  &  &  &  &  &  &  &  &  &  &  \\
                    7 &  &  &  &  & 32 &  &  &  &  &  &  &  &  &  &  &  &  &  &  &  \\
                    8 &  &  &  & 3738 &  &  &  &  &  &  &  &  &  &  &  &  &  &  &  &  \\
                    9 &  &  &  &  & 2483 &  &  &  &  &  &  &  &  &  &  &  &  &  &  &  \\
                    10 &  &  &  &  & 1278 &  &  &  &  &  &  &  &  &  &  &  &  &  &  &  \\
                    11 &  &  &  & 1595 &  &  &  &  &  &  &  &  &  &  &  &  &  &  &  &  \\
                    12 &  &  &  &  & 14158 &  &  &  &  &  &  &  &  &  &  &  &  &  &  &  \\
                    13 &  & 5837 &  & 42366 &  &  &  &  &  &  &  &  &  &  & 2329 &  & 138 &  &  &  \\
                    14 &  &  &  & 26 &  &  &  &  &  &  &  &  &  & 2329 &  &  &  &  &  &  \\
                    15 &  &  &  & 24128 &  &  &  &  &  &  &  &  &  &  &  &  &  &  &  &  \\
                    16 &  &  &  & 1586 & 6761 &  &  &  &  &  &  &  &  & 138 &  &  &  &  &  &  \\
                    17 &  &  &  & 2343 & 1938 &  &  &  &  &  &  &  &  &  &  &  &  &  &  &  \\
                    18 &  &  &  & 4947 & 13726 &  &  &  &  &  &  &  &  &  &  &  &  &  &  &  \\
                    19 &  &  &  &  & 4887 &  &  &  &  &  &  &  &  &  &  &  &  &  &  & 
                \end{tabular}}
            }
            \caption{Flow solution for the MA $GHLP$ with flow bounds ($G-FB$)  from Figure \ref{fig:Gmed_CGmed_Comparison}.\label{tab:matrix_MAFB_Gmed}}
        \end{table}
        
        Finally, Tables \ref{tab:SA_GFB} and \ref{tab:MA_GFB} summarize computational results for the $G$-FB problem. Their structure is similar to Tables \ref{tab:SA_Gmed_CAB} to \ref{tab:MA_Hmed_AP}, with two differences: (i) both datasets are included in the same table (see column ``Dataset''), and (ii) column ``CPU$_{Aux}(\%)$'' reports the percentage of computing time,  relative to the total computing time, spent solving the subproblem for finding optimal flows associated with a given design solution.
        
        \begin{table}[H]
            \centering
            {
            \resizebox{1\textwidth}{!}{
                \begin{tabular}{ccc|rrrrrrrr}
        Dataset & $n$ & $\alpha$ & \multicolumn{1}{c}{CPU (s)} & \multicolumn{1}{c}{$\#Nod$} & \multicolumn{1}{c}{GAP (\%)} & \multicolumn{1}{c}{UB} & \multicolumn{1}{c}{LB} & \multicolumn{1}{c}{LB$_{root}$} & \multicolumn{1}{c}{CPU$_{root}$} & \multicolumn{1}{c}{CPU$_{Aux} (\%)$} \\ \hline\hline
        CAB & 25 & 0.2 & 3.50 & 64 & 0 & 4.3870$\times 10^8$ & 4.3870$\times 10^8$ & 4.2089$\times 10^8$ & 1.26 & 44.57 \\
         &  & 0.5 & 5.78 & 934 & 0 & 4.9065$\times 10^8$ & 4.9065$\times 10^8$ & 4.6200$\times 10^8$ & 0.56 & 50.87 \\
         &  & 0.8 & 32.43 & 9866 & 0 & 5.4538$\times 10^8$ & 5.4538$\times 10^8$ & 5.0341$\times 10^8$ & 0.61 & 59.73 \\
         & 50 & 0.2 & 639.89 & 6656 & 0 & 1.9739$\times 10^9$ & 1.9739$\times 10^9$ & 1.8913$\times 10^9$ & 7.77 & 14.38 \\
         &  & 0.5 & 1884.64 & 35947 & 0 & 2.1383$\times 10^9$ & 2.1383$\times 10^9$ & 1.9821$\times 10^9$ & 9.03 & 20.81 \\
         &  & 0.8 & 2863.64 & 41962 & 0 & 2.2451$\times 10^9$ & 2.2451$\times 10^9$ & 2.0194$\times 10^9$ & 12.84 & 13.50 \\ 
         \cdashline{1-11}[0.5pt/3pt]
        AP & 25 & 0.2 & 11.26 & 2504 & 0 & 7.5411$\times 10^7$ & 7.5411$\times 10^7$ & 7.1321$\times 10^7$ & 0.52 & 62.34 \\
         &  & 0.5 & 89.30 & 23242 & 0 & 8.2517$\times 10^7$ & 8.2517$\times 10^7$ & 7.4654$\times 10^7$ & 0.68 & 55.06 \\
         &  & 0.8 & 392.38 & 89224 & 0 & 8.6561$\times 10^7$ & 8.6561$\times 10^7$ & 7.4883$\times 10^7$ & 0.68 & 42.77 \\
         & 50 & 0.2 & 265.92 & 5414 & 0 & 8.1037$\times 10^7$ & 8.1037$\times 10^7$ & 7.5301$\times 10^7$ & 4.50 & 36.58 \\
         &  & 0.5 & 3885.71 & 57096 & 0 & 8.6303$\times 10^7$ & 8.6303$\times 10^7$ & 7.3954$\times 10^7$ & 5.74 & 16.52 \\
         &  & 0.8 & 7101.86 & 56575 & 0 & 9.0593$\times 10^7$ & 9.0593$\times 10^7$ & 7.5981$\times 10^7$ & 7.01 & 13.93
        \end{tabular}}
            }
            \caption{Results for the SA $GHLP$ with flow bounds ($G-FB$).\label{tab:SA_GFB}}
        \end{table}
        
        \begin{table}[H]
            \centering
            {
            \resizebox{1\textwidth}{!}{
                \begin{tabular}{ccc|rrrrrrrr}
        Dataset & $n$ & $\alpha$ & \multicolumn{1}{c}{CPU (s)} & \multicolumn{1}{c}{$\#Nod$} & \multicolumn{1}{c}{GAP (\%)} & \multicolumn{1}{c}{UB} & \multicolumn{1}{c}{LB} & \multicolumn{1}{c}{LB$_{root}$} & \multicolumn{1}{c}{CPU$_{root}$} &  \multicolumn{1}{c}{CPU$_{Aux} (\%)$} \\ \hline\hline
        CAB & 25 & 0.2 & 2.32 & 276 & 0 & 4.2888$\times 10^8$ & 4.2888$\times 10^8$ & 4.1206$\times 10^8$ & 0.44 & 59.91 \\
         &  & 0.5 & 6.17 & 1197 & 0 & 4.7172$\times 10^8$ & 4.7172$\times 10^8$ & 4.3462$\times 10^8$ & 0.71 & 53.81 \\
         &  & 0.8 & 49.39 & 3863 & 0 & 4.9162$\times 10^8$ & 4.9162$\times 10^8$ & 4.3357$\times 10^8$ & 0.61 & 43.79 \\
         & 50 & 0.2 & 403.98 & 7144 & 0 & 1.9505$\times 10^9$ & 1.9505$\times 10^9$ & 1.8265$\times 10^9$ & 4.98 & 21.63 \\
         &  & 0.5 & 3308.30 & 30775 & 0 & 2.0495$\times 10^9$ & 2.0495$\times 10^9$ & 1.7871$\times 10^9$ & 3.76 & 12.43 \\
         &  & 0.8 & 6598.24 & 23191 & 0 & 2.0744$\times 10^9$ & 2.0744$\times 10^9$ & 1.7948$\times 10^9$ & 5.60 & 8.01 \\
        \cdashline{1-11}[0.5pt/3pt]
        AP & 25 & 0.2 & 13.83 & 2742 & 0 & 7.4905$\times 10^7$ & 7.4905$\times 10^7$ & 6.9816$\times 10^7$ & 0.40 & 61.75 \\
         &  & 0.5 & 131.87 & 9502 & 0 & 7.8169$\times 10^7$ & 7.8169$\times 10^7$ & 6.5415$\times 10^7$ & 0.36 & 29.19 \\
         &  & 0.8 & 1722.12 & 41177 & 0 & 7.8726$\times 10^7$ & 7.8726$\times 10^7$ & 6.3151$\times 10^7$ & 0.65 & 14.59 \\
         & 50 & 0.2 & 397.25 & 9015 & 0 & 8.0693$\times 10^7$ & 8.0693$\times 10^7$ & 7.5083$\times 10^7$ & 2.54 & 30.58 \\
         &  & 0.5 & 4698.33 & 11249 & 0 & 8.6051$\times 10^7$ & 8.6051$\times 10^7$ & 7.9533$\times 10^7$ & 3.15 & 10.42 \\
         &  & 0.8 & 7166.48 & 27521 & 0 & 8.8683$\times 10^7$ & 8.8683$\times 10^7$ & 7.6443$\times 10^7$ & 4.05 & 10.46
        \end{tabular}}
            }
            \caption{Results for the MA $GHLP$ with flow bounds ($G-FB$). \label{tab:MA_GFB}}
        \end{table}
\end{document}